\documentclass[12pt]{amsart}
\usepackage{url, 
	amssymb,setspace, mathrsfs,fontenc, comment}
\usepackage[alphabetic]{amsrefs}
\usepackage{fullpage} 
\usepackage{color}
\usepackage{tikz-cd}
\usepackage[all]{xy}
\usepackage{amsmath}

\usepackage{url, 
	amssymb,setspace, mathrsfs,fontenc}
\usepackage{amsrefs}
\usepackage{graphicx}
\usepackage[mathcal]{euscript}
\usepackage{verbatim}
\usepackage{hyperref}
\hypersetup{backref=true}
\usepackage{mathtools}
\usepackage{tikz}
\usetikzlibrary{chains}

\tikzset{node distance=2em, ch/.style={circle,draw,on chain,inner sep=2pt},chj/.style={ch,join},every path/.style={shorten >=4pt,shorten <=4pt},line width=1pt,baseline=-1ex}

\newtheorem{thm}{Theorem}
\newtheorem{lem}[thm]{Lemma}
\newtheorem{prop}[thm]{Proposition}
\newtheorem{conj}[thm]{Conjecture}
\newtheorem{cor}[thm]{Corollary}
\newtheorem{defe}[thm]{Definition}

\theoremstyle{remark}
\newtheorem{rem}[thm]{Remark}
\newtheorem{exam}[thm]{Example}

\DefineSimpleKey{bib}{myurl}
\newcommand\myurl[1]{\url{#1}}
\BibSpec{webpage}{
	+{}{\PrintAuthors} {author}
	+{,}{ \textit} {title}
	+{}{ \parenthesize} {date}
	+{,}{ \myurl} {myurl}
}
\usepackage{arydshln}

\newcommand{\nc}{\newcommand}

\nc{\ssec}{\subsection}

\nc{\on}{\operatorname}

\nc{\sE}{\mathscr{E}}
\nc{\sF}{\mathscr{F}}
\nc{\sL}{\mathscr{L}}
\nc{\sD}{\mathscr{D}}
\nc{\sA}{\mathscr{A}}

\nc{\cC}{\mathcal{C}}
\nc{\cG}{\mathcal{G}}
\nc{\cV}{\mathcal{V}}
\nc{\CB}{\mathcal{B}}
\nc {\K}{\mathcal{K}}

\nc{\cE} {\mathcal{E}}
\nc{\Kl}{\mathrm{Kl}}
\nc{\cO}{\mathcal{O}}
\nc{\cF}{\mathcal{F}}
\nc{\cZ}{\mathcal{Z}}
\nc{\bcZ}{\overline{\mathcal{Z}}}
\nc{\bcB}{\overline{\mathcal{B}}}
\nc{\cD}{\mathcal{D}}
\nc{\cDt}{\mathcal{D}^\times}
\nc{\cH}{\mathcal{H}}
\nc{\bZ}{\mathbb{Z}}
\nc{\bH}{\mathbb{H}}
\nc{\bQ}{\mathbb{Q}}
\nc{\bR}{\mathbb{R}}
\nc{\bC}{\mathbb{C}}
\nc{\bQl}{\overline{\mathbb{Q}}_\ell}
\nc{\bQlt}{\bQl^\times} 
\nc{\FG}{\mathrm{FG}}
\nc{\dR}{\mathrm{dR}}
\nc{\dv}{\dot{v}}
\nc{\du}{\dot{u}}

\nc{\uG}{\underline{G}}
\nc{\uc}{\underline{c}}
\nc{\uu}{\underline{u}}
\nc{\cU}{\mathcal{U}}
\nc{\rat}{\mathrm{rat}}
\nc{\Hyp}{\mathrm{Hyp}}
\nc{\Lie}{\mathrm{Lie}}
\nc{\ctheta}{\check{\theta}}
\nc{\nil}{\mathrm{nil}}
\nc{\bLX}{\overline{LX}}
\nc{\bOmega}{\overline{\Omega}}
\nc{\tOmega}{\widetilde{\Omega}}

\nc{\fF}{\mathfrak{F}}
\nc{\fB}{\mathfrak{B}}
\nc{\fZ}{\mathfrak{Z}}
\nc{\fx}{\mathfrak{x}}
\nc{\fy}{\mathfrak{y}}
\nc{\fb}{\mathfrak{b}}
\nc{\fk}{\mathfrak{k}}
\nc{\fI}{\mathfrak{i}}
\nc{\fj}{\mathfrak{j}}
\nc{\fg}{\mathfrak{g}}
\nc{\fu}{\mathfrak{u}}
\nc{\fl}{\mathfrak{l}}
\nc{\fn}{\mathfrak{n}}
\nc{\cP}{\mathcal{P}}
\nc{\cQ}{\mathcal{Q}}
\nc{\ft}{\mathfrak{t}}
\nc{\fz}{\mathfrak{z}}
\nc{\fc}{\mathfrak{c}}
\nc{\cfc}{\check{\mathfrak{c}}}
\nc{\fh}{\mathfrak{h}}
\nc{\fp}{\mathfrak{p}}
\nc{\cfp}{\mathring{\mathfrak{p}}}
\nc{\bone}{\mathbf{1}}
\nc{\tg}{\mathtt{g}}
\nc{\hfg}{\widehat{\fg}}
\nc{\ch}{\check{\fh}}
\nc{\hP}{\hat{P}}
\nc{\hg}{\widehat{\mathfrak{g}}}
\nc{\gO}{\mathfrak{g}[\![t]\!]}
\nc{\Ug}{\widehat{U}(\mathfrak{g})}
\nc{\dl}{/\!\!/}

\nc{\bGm}{\mathbb{G}_m}
\nc{\bGa}{\mathbb{G}_a}
\nc{\bL}{\mathbf{L}}
\nc{\bK}{\mathbf{K}}
\nc{\bJ}{\mathbf{J}}
\nc{\bI}{\mathbf{I}}
\nc{\bV}{\mathbb{V}}
\nc{\bM}{\mathbb{M}}
\nc{\bP}{\mathbb{P}}
\nc{\bA}{\mathbb{A}}
\nc{\bN}{\mathbb{N}}

\nc {\Q}{\mathrm{Q}}
\nc{\diag}{\mathrm{diag}}
\nc{\diff}{\mathrm{diff}}
\nc{\ev}{\mathrm{ev}}
\nc{\Res}{\mathrm{Res}}
\nc{\Fl}{\mathcal{F}\ell}
\nc{\Ad}{\mathrm{Ad}}
\nc{\ad}{\mathrm{ad}}
\nc{\pr}{\mathrm{pr}}
\nc{\Sl}{\mathfrak{sl}}
\nc{\gl}{\mathfrak{gl}}
\nc{\ra}{\rightarrow}
\nc{\tra}{\twoheadrightarrow}
\nc{\hra}{\hookrightarrow}
\nc{\quo}{\mathopen{ /\!/}}
\nc{\GL}{\mathrm{GL}}
\nc{\SL}{\mathrm{SL}}
\nc{\Sp}{\mathrm{Sp}}
\nc{\SO}{\mathrm{SO}}
\nc{\so}{\mathfrak{so}}
\nc{\PGL}{\mathrm{PGL}}
\nc{\Bun}{\mathrm{Bun}}
\nc{\supp}{\mathrm{supp}}
\nc{\bgamma}{\bar{\gamma}}
\nc{\ab}{\mathrm{ab}}
\nc{\td}{\mathrm{d}}
\nc{\Ht}{\mathrm{ht}}
\nc{\tX}{\tilde{X}}

\nc         {\rar}[1]       {\stackrel{#1}{\longrightarrow}}

\nc{\fa}{\mathfrak{a}}
\nc{\Hit}{\mathrm{Hit}}

\nc{\RS}{\mathrm{RS}}
\nc{\Loc}{\mathrm{Loc}}
\nc{\tLoc}{\widetilde{\mathrm{Loc}}}
\nc{\reg}{\mathrm{reg}}
\nc{\im}{\mathrm{Im}}

\nc{\tp}{\mathfrak{p}}
\nc{\cA}{\mathcal{A}}
\nc{\cY}{\mathcal{Y}}

\nc{\opp}{\mathrm{opp}}
\nc{\Ind}{\mathrm{Ind}}
\nc{\sAn}{\mathrm{can}}
\nc{\Lg}{\check{\fg}}
\nc{\cDelta}{\check{\Delta}}
\nc{\cPhi}{\check{\Phi}}
\nc{\LV}{\check{V}}
\nc{\Lh}{\check{h}}
\nc{\LG}{\check{G}}
\nc{\cT}{\check{T}}
\nc{\ct}{\check{\ft}}
\nc{\cB}{\check{B}}
\nc{\cb}{\check{\fb}}
\nc{\cN}{\check{N}}
\nc{\sN}{\mathcal{N}}
\nc{\cn}{\check{\fn}}
\nc{\Spec}{\mathrm{Spec}}
\nc{\End}{\mathrm{End}}
\nc{\crho}{\check{\rho}}
\nc{\clambda}{\check{\lambda}}

\nc{\rX}{\mathring{X}}
\nc{\ru}{\mathring{u}}

\nc{\sW}{\mathscr{W}}
\nc{\sH}{\mathscr{H}}
\nc{\sV}{\mathscr{V}}
\nc{\geom}{\mathrm{geom}}
\nc{\Irr}{\mathrm{Irr}}
\nc{\fm}{\mathfrak{m}}
\nc{\aff}{\mathrm{aff}}
\nc{\Aut}{\mathrm{Aut}}
\nc{\cJ}{\mathcal{J}}
\nc{\fs}{\mathfrak{s}}
\nc{\Stab}{\mathrm{Stab}}
\nc{\st}{\mathrm{st}}
\nc{\tw}{{\widetilde{w}}}
\nc{\gen}{\mathrm{gen}}
\nc{\genn}{\mathrm{genn}}
\nc{\sss}{\mathrm{ss}}
\nc{\fsp}{\mathfrak{sp}}
\nc{\Hom}{\mathrm{Hom}}
\nc{\bm}{\mathbf{m}}
\nc{\HG}{\mathcal{HG}}
\nc{\Gal}{\mathrm{Gal}}
\nc{\Sym}{\mathrm{Sym}}
\nc{\rank}{\mathrm{rank}}

\nc{\calX}{\mathcal{X}}
\nc{\tP}{\mathtt{P}}
\nc{\tL}{\mathtt{L}}
\nc{\tU}{\mathtt{U}}

\nc{\tW}{\widetilde{W}}
\nc{\tdb}{\tilde{b}}
\nc{\tdd}{\tilde{d}}
\nc{\tv}{\tilde{v}}
\nc{\Hk}{\on{Hk}}
\nc{\cL}{\mathcal{L}}
\nc{\talpha}{\widetilde{\alpha}}
\nc{\tQ}{{\widetilde{Q}}}
\nc{\ochi}{\overline{\chi}}
\nc{\tdelta}{\widetilde{\Delta}}
\nc{\wt}{\mathrm{wt}}
\nc{\fQ}{\mathfrak{Q}}

\nc{\Rep}{\mathrm{Rep}}
\nc{\Conn}{\mathrm{Conn}}
\nc{\Hecke}{\mathrm{Hecke}}
\nc{\Gr}{\mathrm{Gr}}
\nc{\GR}{\mathrm{GR}}
\nc{\IC}{\mathrm{IC}}
\nc{\Std}{\mathrm{Std}} 
\nc{\Db}{\mathrm{D}^{\mathrm{b}}}
\nc{\tr}{\mathrm{tr}}
\nc{\gr}{\mathrm{gr}}
\nc{\tmin}{\mathrm{min}}
\nc{\Fun}{\mathrm{Fun}~}

\nc{\bbA}{\mathbb{A}}
\nc{\mO}{\mathrm{O}}

\newcommand{\quash}[1]{}

\AtEndDocument{\bigskip{\footnotesize

\textsc{Tsao-Hsien Chen, School of Mathematics, University of Minnesota, Twin cities, Minneapolis, MN 55455 } \par
\textit{E-mail address}: \texttt{chenth@umn.edu} \par
		
\textsc{Lingfei Yi, Shanghai Center for Mathematical Sciences, Fudan University, Shanghai 200438, China} \par
\textit{E-mail address}: \texttt{yilingfei@fudan.edu.cn} \par
}}

\begin{document} 
\renewcommand{\thepart}{\Roman{part}}

\renewcommand{\partname}{\hspace*{20mm} Part}
	
\begin{abstract}
We study the 
singularities of closures of Iwahori orbits
on loop spaces of symmetric varieties 
extending the celebrated work of Lusztig-Vogan  
to the affine setting.
We show that the
$\IC$-complexes of orbit closures (with possible non-trivial coefficients) are pointwise pure and 
satisfy a
parity vanishing property. We apply those geometric results to study the affine Lusztig-Vogan modules and obtain 
fundational results about them including 
 the positivity properties of the affine Kazhdan-Lusztig-Vogan polynomials.
Along the way, we 
construct conical transversal slices 
inside loop spaces of symmetric varieties generalizing the work
of Mars-Springer in the finite dimensional setting.
Our results answer a question of Lusztig. 

We deduce results  for singularities of spherical orbit closures
and provide applications to
relative Langlands duality including the 
positivity for the 
relative Kostka-Foulkes polynomials
and the formality conjecture.
\end{abstract}

\title{Singularities of  
orbit closures in
loop spaces of symmetric varieties} 
\author{Tsao-Hsien Chen and Lingfei Yi}
\date{\today} 
\maketitle
	
\tableofcontents

\section{Introduction}
Let $G$ be a  connected reductive group over $k=\overline{\mathbb F}_p$
of characteristic $p\neq 2$,
and $\theta:G\to G$ an involution.
Let $K=G^\theta$ be the symmetric subgroup of fixed point set of $\theta$.
In their celebrated work \cite{LV}, Lusztig-Vogan 
explore connections between the geometry of $K$-orbits on the flag variety 
of $G$ and representation theory of 
real groups.
They
establish deep results on 
singularities of $K$-orbit closures on 
flag varieties, along the way, they introduce several fundamental objects in representation theory, such as Lusztig-Vogan modules and Kazhdan-Lusztig-Vogan polynomials, 
providing representation theoretic and combinatorial interpretation
for the singularities.

In \cite[page 27]{L2}, Lusztig asked a question about extending the work \cite{LV} to the affine setting,  namely, one considers the $LK$-orbits on
 the affine flag variety 
$\Fl=LG/I_0$, where $LK=K(\!(t)\!)$ (resp.  $LG=G(\!(t)\!)$)
is the loop group of $K$ (resp. $G$) and 
$I_0\subset L^+G=G[\![t]\!]$ is an Iwahori subgorup.
The main results of the paper provide an answer to Lusztig's question.\footnote{In \emph{loc. cit.}, Lusztig remarked that, in 1993,  he  communicated the  question to 
 M. Finkelberg  as a research problem.}

We will be working with  $I_0$-orbits on the loop space $LX=X(\!(t)\!)$ of the symmetric variety $X=G/K$
instead of $LK$-orbits on $\Fl$.\footnote{
The reason for working with $I_0$-orbits on $LX$
is two-fold: first, 
there is a canonical 
injection 
$LK\backslash\Fl\cong I_0\backslash LG/LK\to I_0\backslash LX$
from the set of  $LK$-orbits on $\Fl$ to the set of 
$I_0$-orbits on $LX$, which is a bijection if $K$ is connected. 
Thus we are dealing with a more 
general setting.
Secondly, as we shall see, the geometry of $I_0$-orbits 
closures on $LX$
behaves much nicer than $LK$-orbit closures on $\Fl$, see Section \ref{placid}.}
Let $\sD$ be the set of pairs $(\cL,\cO)$ where $\cL$ is an $I_0$-equivariant local system on an 
$I_0$-orbit $\cO$ in $LX$.
We show that the free $\mathbb Z[q,q^{-1}]$-module $M$ (here $q$ is an indeterminate) with  basis indexed by
$\sD$ has a canonical module structure over the affine Hecke algebra of $G$, to be called the \emph{affine Lusztig-Vogan module}. 
We show that the $\IC$-complexes for closures of $I_0$-orbits on $LX$ give rise to another basis of 
$M$, called the 
\emph{Kazhdan-Lusztig basis}, 
with several remarkable properties.
The entries of the 
transition matrix between the 
standard basis and the Kazhdan-Lusztig basis
are polynomials 
in $q$, to be called the \emph{affine Kazhdan-Lusztig-Vogan polynomials}, and 
we provide an algorithm to compute them. 
We further show that the  $\IC$-complexes of orbit closures 
are pointwise pure and satisfy a parity vanishing property
and we 
deduce from it the positivity of 
the affine Kazhdan-Lusztig-Vogan polynomials.

The proof strategy follows the one in \cite{LV} 
but there are several new difficulties in the affine setting and hence requires new ideas.
First,  the combinatorics and geometry of Iwahori orbits are much more complicated.
For example, Iwahori orbits 
and their closures are all semi-infinite (dimension and codimension are
all infinite) and the theory of perverse sheaves and Verdier duality on them are  delicate. 
To this end, we prove 
several fundational results 
on Iwahori orbits
including 
orbit parametrizations, characterization of closed orbits, and
the placidness of orbit closures. The latter enables us to reduce  many technical difficulties in sheaf theory 
to the finite dimensional setting.
Secondly, the proof of parity vanishing of $\IC$-complexes
in \emph{loc. cit.}
uses representation theory of real groups, which is not available 
in the affine setting.  In \cite{MS}, Mars-Springer found a more geometric argument for the parity vanishing which bypass the use of real groups.\footnote{That being said, the  use of real groups in \cite{LV} should be viewed 
as an important feature of symmetric varieties since it provides
representation interpretations of the singularities.
See Section \ref{Intro: RKF} for relevant discussions in the affine setting.} 
We extend Mars-Springer's argument to the affine setting. A key geometric ingredient here is a construction of certain conical transversal 
slices in $LX$
generalizing  the construction due to Mars-Springer in the case of $X$.

We deduce results for singularities of 
spherical orbit closures  in $LX$
and provide applications to the relative Langlands duality including 
the  positivity for the 
relative Kostka-Foulkes polynomials,
a semisimplicity criterion for the (abelian) relative Satake category,
and the formality conjecture.

The recent works \cite{CN3}, \cite{QuantumSL2}, \cite{BZSV}
relate the singularities of orbit closures in $LX$
with harmonic analysis, real groups, and quantum groups. We mention some results in these directions.

We now describe the paper in more details.

\subsection{Main results}
To simplify the discussions, we will assume $G$ is simply connected in the introduction.
We fix a $\theta$-stable Borel pair $(B_0,T_0)$ 
and let $I_0\subset L^+G$ be the corresponding Iwahori subgroup.  
Let $\mathrm W_\aff$ be the affine Weyl group of $G$.  
Let $H$ be the  affine Hecke algebra associated to $\mathrm W_\aff$
 over $\mathbb Z[q,q^{-1}]$ with standard basis elements $T_w,w\in\mathrm W_\aff$.

Recall the free $\mathbb Z[q^{},q^{-1}]$-module 
$M$ with basis indexed by $(\cL,\cO)\in\sD$. 
\begin{thm}\label{main 1}
The module $M$ has a unique module structure over the 
affine Hecke algebra $H$ such that $T_w$ acts according to the formula
~\eqref{eq:aff Hecke act}.
\end{thm}
Theorem \ref{main 1} is proved in 
Proposition \ref{construction of M}
and Lemma \ref{l:Omega Hecke action}.
Along the way, we obtain a combinatoric
parameterization of $I_0$-orbits on $LX$
and a characterization of closed orbits (see Lemma \ref{l:Iwahori orbirs on LX}
and Lemma \ref{l:orbit type b}).
The module $M$ is an affine version of the Lusztig-Vogan module in \cite[Proposition 1.7]{LV}, to be called the \emph{affine Lusztig-Vogan module}.

To state the next result, we recall some geometric facts 
on $I_0$-orbits in $LX$.
The $I_0$-orbits on $LX$ in general are infinite dimensional.
However, the placidness results in Section \ref{placid} imply that, for a pair of $I_0$-orbits 
$\cO_u,\cO_v\subset LX$ such that $\cO_u\subset\overline\cO_v$,
there is well-defined notion of codimension 
$\on{codim}_{\overline\cO_v}(\overline\cO_u)\in\mathbb Z_{\geq0}$
between their orbit closures. Moreover, in Proposition \ref{dim for LX}
we show that there exists a  dimension theory 
of $LX$, that is,  
one can assign an integer $\delta(v)\in\mathbb Z$ for each 
orbit $\cO_v$ such that 
\begin{equation}\label{delta}
\delta(v)-\delta(u)=\on{codim}_{\overline\cO_v}(\overline\cO_u)
\end{equation}
if $\cO_u$ is in the closure of $\cO_v$.
The integers $\delta(v)$ can be thought as 
a replacement of the dimension of $\cO_v$.

To simplify notations, we will write $(\xi,v)=(\cL_{\xi,v},\cO_v)\in\sD$
and $m_{\xi,v}\in M$ for the corresponding basis element.
We write $u\leq v$ if $\cO_u\subset\overline\cO_v$
and $u<v$ if $u\leq v$ and $\cO_u\neq\cO_v$.

\begin{thm}\label{main 2}
A choice of a dimension theory $\delta=\{\delta(v)\}$ 
of $LX$ gives rise to a unique $\mathbb Z$-linear map 
$D_\delta: M\to M$ such that 
\begin{itemize}
\item [(i)] $D_\delta(qm)=q^{-1}D_\delta(m)$ for $m\in M$.
\item [(ii)] $D_\delta((T_{s_\alpha}+1)m)=q^{-1}(T_{s_\alpha}+1)D_\delta(m)$.
\item [(iii)] For any $(\xi,v)\in\sD$, we have 
\[D_\delta(m_{\xi,v})=q^{-\delta(v)}m_{\xi,v}+\sum_{u<v} b_{\eta,u,\xi,v}m_{\eta,\mu}\]
where $q^{\delta(v)}b_{\eta,u,\xi,v}$ are polynomials in $q$ of degree at most 
$\delta(v)-\delta(u)$. Moreover, there is an 
algorithm to compute $b_{\eta,u,\xi,v}$.

\end{itemize}

\end{thm}
Theorem \ref{main 2} is proved in Section \ref{polynomials b and c}
and Section \ref{ss:algorithm}.
The existence of $D_\delta$ is a combinatorial problem, 
but we have to use sheaf theory on 
$LX$ (namely, the Verdier duality) to solve it.
The polynomial $q^{\delta(v)}b_{\eta,u,\xi,v}$
is an affine analogy of the polynomial $R_{\gamma,\delta}$
in \cite[Theorem 1.10]{LV}.

\begin{thm}\label{main 3}
There is 
a unique family of polynomials $P_{\eta,u;\xi,v}\in\mathbb Z_{\geq0}[q]$ 
 in $q$ with non-negative integer coefficients, indexed by pairs $(\eta,u),(\xi,v)\in\sD$,
satisfying the following conditions:
	\begin{itemize}
		\item [(i)] $P_{\xi,v;\xi,v}=1$.
		\item [(ii)] If $u\neq v$, 
		$\deg(P_{\eta,u;\xi,v})\leq\frac{1}{2}(\delta(v)-\delta(u)-1)$.
		\item [(iii)] $C_{\xi,v}:=\sum_{\eta,u}P_{\eta,u;\xi,v}m_{\eta,u}\in M$
		satisfies $D_\delta C_{\xi,v}=q^{-\delta(v)}C_{-\xi,v}$.
	\end{itemize}
\end{thm}

Theorem \ref{main 3} is restated in Theorem \ref{t:I_0 Poincare poly uniqueness}.
In fact in \emph{loc. cit.} we prove a slightly stronger uniqueness result: 
the family of polynomials $P_{\eta,u;\xi,v}$ 
are unique over the larger ring $\mathbb Z[q^{\frac{1}{2}},q^{\frac{-1}{2}}]$.

Note that part (i) and (ii) imply that the family of elements $C_{\xi,v}$, $(\xi,v)\in\sD$
forms a basis of $M$ to be called the Kazhdan-Lusztig basis of $M$
and the entries of the transition matrix between the standard basis $m_{\xi,v}$
and Kazhdan-Lusztig basis $C_{\xi,v}$
are given by the family of polynomials $P_{\eta,u;\xi,v}$.
We will call the family $P_{\eta,u;\xi,v}$ 
the \emph{affine Kazhdan-Lusztig-Vogan polynomials}
associated to $X$
since they include the Kazhdan-Lusztig-Vogan polynomials
$P_{\gamma,\delta}$ in \cite[Theorem 1.11]{LV} (see Lemma \ref{comparison with MS}).
Theorem \ref{main 2} (iii) provides an algorithm to compute them.

The most difficult part in the proof of Theorem \ref{main 3} is the positivity property 
of $P_{\eta,u;\xi,v}$. It is a consequence of the following 
pointwise purity and  
parity vanishing results
for $\IC$-complexes for Iwahori orbit closures.
For any $(\xi,v)\in\sD$, let
$\IC_{\xi,v}=(j_v)_{!*}(\cL_{\xi,v})[\delta(v)]$ be the $\IC$-extension 
of the local system $\cL_{\xi,v}$ along the inclusion $j_v:\cO_v\to\overline\cO_v$.
For any pair $(\eta,u),(\xi,v)\in\sD$ we write $[\cL_{\eta,u}:\sH^{i}(\IC_{\xi,v})|_{\cO_u}]$ for the multiplicity of 
$\cL_{\eta,u}$ in the Jordan-H\"older series of
$\sH^{i}(\IC_{\xi,v})|_{\cO_u}$.

\begin{thm}\label{main 4}
For any pair $(\eta,u),(\xi,v)\in\sD$ we have 
\begin{itemize}
\item [(i)] $\IC_{\xi,v}$ is pointwise pure of weight $\delta(v)$.
\item [(ii)]
$\sH^i(\IC_{\xi,v})=0$ if $i+\delta(v)$ is odd.
\item [(iii)]
$P_{\eta,u,\xi,v}=\sum_{i} [\cL_{\eta,u}:\sH^{2i-\delta(v)}(\IC_{\xi,v})|_{\cO_u}]q^{i}$.
\end{itemize}
\end{thm}
Theorem \ref{main 4} is a combination of   
Theorem \ref{purity}, Theorem \ref{t:IC parity vanishing},
and Theorem \ref{t:I_0 Poincare poly uniqueness}.
Part (ii) extends the parity vanishing results in \cite[Theorem 1.12]{LV}
to the affine setting. The proof in \emph{loc. cit.} uses the theory of weights and 
a representation theoretic interpretation of the multiplicity.
In the affine setting, a representation theoretic interpretation of the multiplicity
$[\cL_{\eta,u}:\sH^{2i-\delta(v)}(\IC_{\xi,v})|_{\cO_u}]$ is not available at the moment 
(but it would be very interesting to find one, see Section \ref{Intro: RKF} for some related discussions).
Thus we follow the approach in \cite{MS} where Mars-Springer exploit the
pointwise purity property in (i).
The key geometric ingredient in \cite{MS}  is the construction of 
a conical transversal slices for $B_0$-orbits on $X$, to be called the Mars-Springer  slices, 
where the pointwise purity property is a direct consequence.

We generalize Mars-Springer's construction to the affine setting. 
\begin{thm}\label{main 5}
There exists conical transversal slices to the $L^+G$-orbits and $I_0$-orbits on $LX$.
\end{thm}

Theorem \ref{main 5} is stated in Corollary \ref{c:LX equi-singular} and Corollary \ref{c:bar O equi-singular}.
We refer to Section \ref{Transversal slices} for a more detailed explanation of the
statement including the 
identification of the spherical orbits slices with certain Lagrangian subvarieties of the affine Grassmannian slices  
(see Proposition \ref{t:Lagrangian}).

Note that since the orbits are in general semi-infinite,
the meaning of transversal slices
and  pointwise purity  
require some treatment (see Section \ref{A: pointwise purity}).

\subsection{Applications}
We deduce the pointwise purity property and parity vanishing for spherical orbits in $LX$ from the results  for Iwahori orbits.\footnote{In fact, one can  give a direct proof of pointwise purity property using the transversal slices for  spherical orbits. 
However, we do not know a proof of parity vanishing
without appealing to the results for Iwahori orbits (see Remark \ref{remark on parity vanishing}).}
The singularities of spherical orbit closures in $LX$ are related to the so called relative Langlands duality \cite{BZSV} 
and we discuss applications of our main results to the subject.
We also discuss connections with 
quantum groups and real groups, providing 
(possible) representation theoretic meaning of the
singularities.

\subsubsection{Formality conjecture}
One application of the pointwise purity for spherical orbits is the formality conjecture in relative Langlands duality. 
In \cite[Conjecture 8.1.8]{BZSV}, the authors conjecture
a version of derived Satake equivalence for 
a spherical variety $X$ satisfying some technical assumptions.
Their conjecture implies that  
the so called de-equivariantized Ext algebra 
\[A_X:=R\Hom_{D^{}(L^+G\backslash LX)}(\mathrm{e}_{L^+X},\IC_{\cO(\check G)}\star\mathrm{e}_{L^+X})\]
(a.k.a the Plancherel algebra)
for the derived relative Satake category $D^{}(L^+G\backslash LX)$
of $X$ is formal.
A general approach to proving formality 
is to prove the  purity of the 
Ext algebra $H^\bullet(A_X)$.
We confirm this conjecture in the case of symmetric varieties.
\begin{cor}\label{appl 1}
	For $X$ a symmetric variety, $A_X$ is formal.
\end{cor}
Theorem \ref{appl 1} is stated in Theorem \ref{t:formality}.
We refer to Section \ref{Ext algebras} for a more detailed explanation of the
statement.
\begin{rem}
    In \cite{BZSV},
the authors consider a more general setting of spherical varieties $X$ with the assumption that 
the cotangent bundles $M=T^*X$ are the so called \emph{hyperspherical varieties}. Corollary \ref{appl 1} implies that the
formality conjecture holds true 
for \emph{all} symmetric varieties $X$
without the 
hyperspherical assumption.
\end{rem}

\begin{rem}\label{splitting rank}
Corollary \ref{appl 1} was previously known in the 
 case of symmetric varieties of splitting ranks
(see \cite{QuaternionicSatake} and \cite{LOSatake}).
The proof in \emph{loc. cit.}
uses some special features of splitting rank symmetric varieties
which are not available for the general cases.
  Our proof gives a uniform argument for all cases.
\end{rem}

\subsubsection{Semisimplicity
criterion and relative Satake equivalences}
A basic result about perverse sheaves on affine Grassmannian $\Gr=LG/L^+G$
is the semisimplicity of the  Satake category 
$\on{Perv}(L^+G\backslash\Gr)$ of $G$(see, e.g., \cite[Proposition 5.1.1.]{ZhuIntroGr}). The  proof 
relies on 
the parity vanishing result 
for 
spherical orbits on $\Gr$ due to Lusztig \cite[Theorem 11.c]{L3}.
Using the 
parity vanishing for spherical orbits in Theorem \ref{t:IC parity vanishing}
and \cite{CN3}, we prove the 
following semisimplicity criterion 
and Langlands dual description of the 
relative Satake categories:
\begin{cor}\label{appl 3}
Assume the codimensions of $L^+G$-orbits in the same connected component of $LX$ are even.
\begin{itemize}
    \item [(i)] The relative Satake category  $\on{Perv}_{}(L^+G\backslash LX)$
of $L^+G$-equivariant perverse sheaves on $LX$ is semisimple.
\item [(ii)] Assume further that the 
$L^+G$-stabilizers on $LX$ are connected.
Then there is an equivalence of abelian categories
\[\on{Perv}(L^+G\backslash LX)\cong\on{Rep}(\check G_X)\]
where the $\on{Rep}(\check G_X)$
is the category of finite dimensional
complex representations of the dual group 
$\check G_X$ of X \cite{GN}.
\end{itemize}

\end{cor}

Corollary \ref{appl 3} is restated in Theorem \ref{semi}. In
Proposition \ref{codim formula} we proved a 
 codimension formula 
of spherical orbits (assuming $LX$ is connected) 
and one can use it 
to find symmetric varieties satisfying the assumptions
in (i) and (ii) above.
Interesting examples include the
splitting rank symmetric varieties
$\SL_{2n}/\Sp_{2n},\mathrm{Spin}_{2n}/\on{Spin}_{2n-1}$, $\mathrm{E}_6/\mathrm{F}_4$. The corresponding relative dual groups $\check G_X$ are 
 $\PGL_{2n},\PGL_2$, $\PGL_3$.
In particular, we obtain 
new  instances of 
relative Satake equivalence, see Example \ref{new examples}.

\begin{rem}
We suspect that 
the converse of  Corollary \ref{appl 3} (i) might be also true.
\end{rem}

\begin{rem}
The proof of the codimension formula Proposition \ref{codim formula} is quite involved, it uses real groups, quasi-maps, and a spreading out argument. It would be nice to find a more direct argument.
\end{rem}

\begin{rem}
Using the real-symmetric equivalence in \cite{CN3}, 
one can show that the equivalence in (ii) is compatible with 
the fusion product on $\on{Perv}(L^+G\backslash LX)$ 
and the tensor product on $\on{Rep}(\check G_X)$, and hence admits an upgrade
to an equivalence of symmetric monoidal 
abelian categories.
\end{rem}

\begin{rem}
The assumptions in the corollary are necessary. 
In fact, the work \cite{CN3} shows that 
$\on{Perv}(L^+G\backslash LX)$ always contains the semisimple abelian category $\on{Rep}(\check G_X)$ as a full subcategory but 
if there are  orbits  with odd codimensions or having disconnected stabilizers, 
then $\on{Perv}(L^+G\backslash LX)$
might not be semisimple 
or contain $\IC$-complexes with non-trivial coefficients not belonging to $\on{Rep}(\check G_X)$ (see Remark \ref{SO_2} for the discussion when $X=\SL_2/\SO_2$). 

Thus for general symmetric varieties, the relative Satake categories $\on{Perv}(L^+G\backslash LX)$ 
might be larger than $\on{Rep}(\check G_X)$
and in those situations
we expect  the Langlands dual descriptions of them
would be more complicated involving certain quantum
groups (see Section \ref{real and quantum}).

\end{rem}

\subsubsection{Relative
Kostka-Foulkes polynomials}\label{Intro: RKF}
The $L^+G$-orbits on $LX$
are parametrized by 
the set of dominant coweights $\Lambda_S^+$ 
for a maximal $\theta$-split torus $S$
(see Proposition \ref{p:sym var orbits}).
We write $(\chi,\lambda)=(\cL_{\chi,\lambda},LX_\lambda)$
where $\cL_{\chi,\lambda}$ is a 
 $L^+G$-equivariant local systems 
on the orbits $LX_\lambda$ indexed by $\lambda\in\Lambda_S^+$
and representations $\chi$ of the component groups of stabilizers, 
and $\IC_{\chi,\lambda}$ for the $\IC$-extension of 
$\cL_{\chi,\lambda}$ to the closure of $\overline{LX}_\lambda$.
For any pair $(\chi',\lambda'),(\chi,\lambda)$ as above, we define
\[
P_{\chi',\mu;\chi,\lambda}=\sum_{i} [\cL_{\chi',\mu}:\sH^{i-\delta(\lambda)}(\IC_{\chi,\lambda})|_{LX_{\mu}}]q^{\frac{i}{2}},
\]
here $\delta(\lambda)=\delta(v_\lambda)$,
where 
$\cO_{v_\lambda}\subset LX_\lambda$
is the unique open $I_0$-orbit (see Corollary \ref{c:Iwahori closure and G(O) into Iwahori}).

We show that 
the polynomials $P_{\chi',\mu;\chi,\lambda}$
are special cases of the affine Kazhdan-Lusztig-Vogan polynomials 
(see~\eqref{IC=IC})
and deduce from Theorem \ref{main 3} the following:
\begin{cor}\label{appl 2}
For any pair $(\chi',\mu),(\chi,\lambda)$,
$P_{\chi',\mu;\chi,\lambda}$ is a polynomial in $q$ with non-negative integer coefficients.
It is the unique family of polynomials in $q^{\frac{1}{2}}$ 
	satisfying the following conditions:
	\begin{itemize}
		\item [(i)] $P_{\chi,\lambda;\chi,\lambda}=1$.
		\item [(ii)] If $\mu\neq\lambda$, 
		$\deg(P_{\chi',\mu;\chi,\lambda})\leq\frac{1}{2}(\delta(\lambda)-\delta(\mu)-1)$.
		\item [(iii)] $C_{\chi,\lambda}:=\sum_{\chi',\mu}P_{\chi',\mu;\chi,\lambda}[\cL_{\chi',\mu}]$
		satisfies $D_\delta C_{\chi,\lambda}=q^{-\delta(\lambda)}C_{-\chi,\lambda}$.
		\end{itemize}
In particular, we have 
\begin{itemize}
\item [(iv)]
$\sH^i(\IC_{\xi,\lambda})=0$ if $i+\delta(\lambda)$ is odd.
\end{itemize}
\end{cor}
Corollary \ref{appl 2} is restated in 
Theorem \ref{t:IC parity vanishing} and
Theorem \ref{t:G(O) Poincare poly uniqueness}.
We refer to Section \ref{Def of KF poly}  for the explanation of condition (iii).
The polynomials $P_{\chi',\mu;\chi,\lambda}$  are computed explicitly in 
\cite[Theorem 1.9]{QuaternionicSatake} and
 \cite[Theorem 1.3]{LOSatake}
for splitting rank symmetric varieties $X$ and 
they coincide with the  Kostka-Foulkes polynomials
for the relative dual group $\check G_X$, generalizing 
the work \cite{L3} in the group cases.
For this reason, we will call the family of polynomials $P_{\chi',\mu;\chi,\lambda}$  the \emph{relative Kostka-Foulkes polynomials}.

\begin{rem}\label{remark on parity vanishing}
    Proving the parity vanishing in $(iv)$ was one of the original motivation for the work. 
  However, 
due to many technical difficulties, such as 
the appearance of 
non-trivial local systems on orbits, 
we do not known a direct proof 
without appealing to the results for Iwahori orbits. In a sense, the study of Iwahori orbits closures is simpler than spherical orbit closures due to the presence of affine Hecke algebra module structures.
\end{rem}

\begin{rem}
The results in \cite{L3} imply,
for splitting rank $X$, the 
 polynomials
  $P_{\chi',\mu;\chi,\lambda}$ 
 admits a representation theory interpretation as a $q$-analogy of  weight multiplicities of the dual group $\check G_X$.
\end{rem}

\subsubsection{Real and quantum groups}\label{real and quantum}
We comment on connections and applications to real and quantum groups.
Let $G_\bR$ be the real form of
$G$ corresponding to $X$ under the Cartan bijection.
The real-symmetric equivalence in \cite{CN3} says there is an equivalence
\begin{equation}\label{real-sym}
    \on{Perv}(L^+G\backslash LX)\cong\on{Perv}(L^+G_\bR\backslash\Gr_\bR)
\end{equation}
where the right hand side is the real Satake category of
equivariant perverse sheaves on the real affine Grassmannian 
$\Gr_\bR$ of $G_\bR$.\footnote{This is a 
 remarkable new phenomenon
specific to $L^+G$-orbits on $LX$.}
One can translate the results in the paper to the $G_\bR$-side of the equivalence~\eqref{real-sym} and obtain
several interesting results in Langlands duality for real groups, such as the 
 parity 
vanishing for $\IC$-complexes on 
$\Gr_\bR$ and the formality conjecture for real Satake category \cite[Corollary 1.7]{CN3}. Those results  
are not obvious since
powerful tools in algebraic geometry (e.g., 
Hodge theory or theory of weights)
are not available in the real analytic setting.

On the other hand, 
in \cite{QuantumSL2} the authors use the 
concrete geometry of the real affine Grassmannian 
 to show that the real Satake category
 $\on{Perv}(L^+\on{PGL}_2(\bR)\backslash\Gr_\bR)$ for $G_\bR=\on{PGL}_2(\bR)$
 is related to the representations of quantum $\SL_2$ 
 at primitive fourth root of unity. Through the real-symmetric equivalence~\eqref{real-sym}, their results
 provide a representation theoretic and combinatorial interpretations of the relative Kostka-Foulkes polynomials
for $X=\on{PGL}_2/\on{PO}_2$.
In particular, we see that, in contrast with 
the group cases or
the splitting rank cases,
 the polynomials
  $P_{\chi',\mu;\chi,\lambda}$ 
  for general $X$
  might be related to weight multiplicities of a certain quantum group instead of $\check G_X$.
In the forthcoming work \cite{CMT}, 
we plan to 
explore those connections for more general symmetric varieties, real groups, and quantum groups.

\subsubsection{Spherical varieties}
We comment on generalization to general spherical varieties.
A large part of the geometric approach taken in \cite{MS} is written uniformly for 
an arbitrary spherical variety.
However, Mars-Springer's 
construction of transversal slices  in symmetric varieties, and also our generalization
to the affine setting,
uses the involution $\theta$ in an essential way and hence 
does not have a direct generalization to general spherical varieties. 
Thus we expect that most of the results of the paper
have  direct 
generalizations to spherical varieties, except the 
pointwise purity and positivity results where the existence of slices is 
a crucial ingredient.

With that being said, there have been recent
works, for examples \cite{FGT,BFGT}, where the authors established, among other things, 
the pointwise purity and formality results for a particular class of spherical varieties.
In a forthcoming work, we extend the main results of the paper to 
the $\GL_{2n}$-spherical variety  
$X=\GL_{2n}/\Sp_{2n}\times\bA^{2n}$ where 
$\GL_{2n}$ acts diagonally (the so called \emph{linear periods} cases), including 
a construction of transversal slices for spherical orbits in $LX$
and  a version of 
the relative derived Satake equivalence.
We hope to explore the situation of more general spherical varieties.

\subsubsection{Characteristic}
In this paper, we work with $\ell$-adic sheaves on schemes over 
fields of positive characteristic
where Deligne's theory of weights  
plays a crucial role.
The proofs are all geometric and hence are applicable in the setting of $D$-modules on schemes over complex numbers where one uses
Saito's theory of mixed Hodge modules. 
 
\subsection{Organization}
In Section \ref{group}, we recall some basic definitions 
in Lie theory.
In Section \ref{orbits} and Section \ref{S:Iwahori orbits}, we study spherical and Iwahori orbits parametrizations
on $LX$.
In Section \ref{placid}, we prove
the placidness of orbit closures in $LX$, a key result on 
the geometry of orbit closures, 
and uses it to show the existence of 
a dimension theory on $LX$.
In Section \ref{Transversal slices}, we 
construct conical transversal slices for both spherical and Iwahori orbits inside 
$LX$ and establish several foundational results about them.
In Section \ref{Lagrangian slices}, we prove a codimension formula for spherical orbits and deduce from it the Lagrangian property of spherical slices.
In Section \ref{Affine Hecke modules}, we construct the 
affine Lusztig-Vogan modules and introduce and study the affine Kazhdan-Lusztig-Vogan
polynomials. We explain an algorithm to compute them.
In Section \ref{main results}, we prove the  pointwise purity 
and parity vanishing results for $\IC$-complexes for 
orbit closures 
and deduce from it the positivity of 
affine Kazhdan-Lusztig-Vogan
polynomials.
In Section \ref{Formality}, we discuss applications 
to the relative Langlands duality. 
In Section \ref{disconnected cases},  we 
discuss the case when $K$ is disconnected (we assume $K$ is connected  from Section \ref{S:Iwahori orbits}-Section \ref{Formality}).
In Appendix \ref{s:appendix}, 
we recall and prove some basic results 
about the geometry of placid ind-schemes and sheaf theory, including 
the discussion of dimension theory, $t$-structures and perverse sheaves,  Verdier duality, etc.

\subsection*{Acknowledgement}
We would like to thank 
Alexis Bouthier, Huanchen Bao,
Michael Finkelberg,
Mark Macerato, 
David Nadler, Jeremy Taylor, David Vogan,
and Ruotao Yang for useful discussions and comments.

The research of T.-H. Chen is supported by 
NSF grant  DMS-2143722.
The research of Lingfei Yi is supported by the Start-up Grant No. JIH1414033Y of Fudan University.

\section{Group data}\label{group}
\subsection{Symmetric varieties}\label{sym}
We recall some basic notations about symmetric varieties following  
 \cite{Springer85,Springer87} and  \cite{RSbook}.
 Let $G$ be a  connected reductive group over $k=\overline{\mathbb F}_p$
of characteristic $p\neq 2$,
and  $\theta:G\to G$ an involution.
We  fix a $\theta$-stable 
Borel pair $(B_0,T_0)$.

Let $K=G^\theta$ be the associated symmetric subgroup. 
The quotient scheme $X=G/K$ is a symmetric variety.
Let $\tau:G\to G, \tau(g)=g\theta(g)^{-1}$.
It induces a $G$-equivariant isomorphism
$\tau:X\cong G^{\on{inv}\circ\theta,\circ}$ (denoted again by $\tau$)
between $X$ and the identity component 
of the subvariety of $G$ on which $\theta$ acts as inverse.
Here $G$ acts on $G^{\on{inv}\circ\theta,\circ}$ 
by the twisted conjugation action 
$g\cdot x=gx\theta(g)^{-1}$.

A torus $S\subset G$ is called $\theta$-split 
if $\theta(s)=s^{-1}$ on $S$.
A parabolic $P$ is called $\theta$-split
if $P\cap\theta(P)$ is a Levi subgroup of $P$.
We fix a maximal $\theta$-split torus $S$ and a minimal 
$\theta$-stable parabolic 
$P$ such that 
$M=Z_G(S)=P\cap\theta(P)$.
We pick a  Borel pair $(B,T)$ such that 
$S\subset T\subset B\subset P$.
Note that $T$ is automatically $\theta$-stable. We call such a pair 
a $\theta$-split Borel pair.

Let $\fg=\Lie(G)$ be the Lie algebra.
The tangent map $d\theta:\fg\rightarrow\fg$ 
has eigenspace decomposition 
$\fg=\fk\oplus\fp$,  
where $\fk=\Lie(K)$
and $\fp$ is the $(-1)$-eigenspace.

\subsection{Loop and arc spaces}
Let $F=k(\!(t)\!)$ and $\cO=k[\![t]\!]$ be the Laurent and Taylor series ring respectively.
For a scheme $Y$ over $k$, 
we will denote its loop space and arc space by $LY$ and  $L^+Y$
such that $LY(k)=Y(F)$
and $L^+Y(k)=Y(\cO)$.
For any $j\in\mathbb Z_{\geq0}$, we have $j$-th arc space 
$L^+_jY$ such that 
$L^+_jY(k)=Y(\cO/t^j)$.
When $Y=G$, the corresponding loop and arc spaces $LG$ and $L^+G$
are called the loop group and arc group of $G$.
We denote by $I_0\subset L^+G$ the Iwahori subgroup associated to $B_0$. 
We let $\Gr=LG/L^+G$ be the affine Grassmannian
and $\Fl=LG/I_0$ the affine flag variety.

\section{Spherical orbits}\label{orbits}

\subsection{Spherical orbits parametrizations}\label{ss:recollection orbits}
We have the following parametrization of 
$L^+G$-orbits on $LX$.
Recall the embedding 
$\tau:X\cong G^{\on{inv}\circ\theta,\circ}\subset G$ 
in Section \ref{sym} induces a map 
$\tau:LX\to LG$ on the corresponding loop spaces.

We denote by $\Phi,\Phi^+
, \check\Phi,\check\Phi^+$ the set of 
roots, positive roots, coroots, positive coroots with respect to the 
$\theta$-split Borel pair $(B,T)$.
We denote by $\Delta\subset\Phi$
and $\check\Delta\subset\check\Phi$ the set of simple roots and simple coroots.
We denote by $\Lambda_T$ and $\Lambda_T^+$
the sets of coweights and dominant coweights. 
We denote by $\Lambda_S\subset\Lambda_T$
the coweights of $S$
and $\Lambda_S^+=\Lambda_S\cap\Lambda_T^+$
the set of dominant coweights of $S$. 
$\Lambda_S$ is naturally ordered by the restriction of the ordering from the coweights 
$\Lambda_T$: we have $\mu\leq\lambda\in\Lambda_S$ if and only if 
$\lambda-\mu$ is a non-negative sum of 
positive coroots.
We write $2\rho$ the sum of positive roots. 
 
\begin{prop}\label{p:sym var orbits}
\begin{itemize}
    \item [(i)] The $L^+G$-orbits on $LX$ are represented by $\Lambda_S^+$:
        for each $\mu\in\Lambda_S^+$, there exists  $x_\mu\in LX(k)=X(F)$ such that $\tau(x_\mu)=t^\mu\in S(F)$ and $LX_\mu=LG^+\cdot x_\mu$ is the corresponding orbit.
\item [(ii)] 
Assume $K$ is connected.
The natural map $LG\to LX$ induces a bijection  $G(F)/K(F)\cong X(F)$ 
\end{itemize}
\end{prop}
\begin{proof}
     Part (i) is proved in   \cite[Theorem 8.2.9]{GN}
        and \cite[Theorem 4.2.]{NadlerMatsuki}.
      Part  (ii) 
        follows from 
        \cite[Theorem 4.3.]{NadlerMatsuki}
        (Since $\cL=\Lambda^+_S$ in \emph{loc. cit.} as $K$ is connected).
\end{proof}

\subsection{Closure relations}
Let $\bLX_\mu$ be the closure of $LX_\mu$ in $LX$ (with the reduced scheme structure). The goal of this section is to prove the following closure relation.
\begin{thm}\label{t:sym var orbit closure}
   For $\forall\mu\in\Lambda_S^+$, 
   $\bLX_\mu=\bigsqcup_{\lambda\leq\mu}LX_\lambda$.
\end{thm} 

\begin{rem}
  If $k=\mathbb C$, the theorem is proved in \cite[Theorem 10.1]{NadlerMatsuki}.
It is interesting to note that 
the proof of in \emph{loc. cit.}
makes use of real affine Grassmannians.
Our argument is algebraic and works 
over an arbitrary algebraically closed field whose characteristic is not two.  
\end{rem}

\subsubsection{The case $\GL_2/\mathrm O_2$}
We begin with the fundamental case of 
$G=\GL_2$, $\theta(g)=(g^t)^{-1}$, $K=\mathrm O_2$, $X=\GL_2/\mathrm O_2$.
We will see later that the general orbit closures
can be reduced to this situation.

In this case $S=T$.
Note that in \cite{NadlerMatsuki}, 
Proposition \ref{p:sym var orbits}.(i) is proved without
assuming $K$ to be connected.
Thus $L^+G$-orbits on $LX$ are parameterized by 
$\Lambda_T^+=\{\lambda=(p,q)\in\bZ^2\,|\,p\geq q\}$.
We will work with the image of $LX$ under embedding
$\tau:LX\rightarrow LG$, $g\mapsto g\theta(g)^{-1}=gg^t$.
The orbit representatives are $\tau(x_\lambda)=t^\lambda$,
and the $L^+G$-action is $g\cdot t^\lambda=gt^\lambda g^t$.
The orbits closure relation follows from the following statement:

\begin{prop}\label{p:GL2/O2 closure curve}
For any integer $m\geq 0$, let $\mu_0=\sqrt{(-1)^{m+1}}$.
Then for any $s\in k^\times$, there exists 
$g=g(s)=
\begin{pmatrix}
	a&b\\
	c&d
\end{pmatrix}
\in\SL_2(\cO)$
such that
\begin{equation}\label{eq:GL2/O2 closure}
g\begin{pmatrix}
	t^{m+1}&\\
	&t^{-1}
\end{pmatrix}g^t
=
\begin{pmatrix}
	a^2t^{m+1}+b^2t^{-1}&act^{m+1}+bdt^{-1}\\
	act^{m+1}+bdt^{-1}&c^2t^{m+1}+d^2t^{-1}
\end{pmatrix}
=
\begin{pmatrix}
	\sum_{i=0}^{m+1}(-s^2)^{m+1-i}t^{i-1}& \mu_0 s^{m+2}t^{-1}\\
	\mu_0s^{m+2}t^{-1}&s^2t^{-1}+1
\end{pmatrix}.
\end{equation}
\end{prop}

\begin{cor}\label{c:GL2/O2 orbits closure}
	For $\forall\mu\in\Lambda_S^+$, 
	$\bLX_\mu=\bigsqcup_{\lambda\leq\mu}LX_\lambda$.
\end{cor}
\begin{proof}
	Since $L^+G\cdot t^\mu\subset L^+G t^\mu L^+G$,
	by the orbits closure relation in affine Grassmannian we know
	$\bLX_\mu\subset\bigsqcup_{\lambda\leq\mu}LX_\lambda$.
	
	For the other direction, 
	it suffices to show $t^{\mu-\alpha}\in\bLX_\mu$ 
	if $\mu-\alpha\in\Lambda_S^+=\Lambda_T^+$, $\alpha=(1,-1)$.
	Acting by center, we can assume $\mu=(m+1,-1)$, $\mu-\alpha=(m,0)$.
	By Proposition \ref{p:GL2/O2 closure curve},
	we get $t^{(m,0)}=\lim_{s\rightarrow0}g(s)\cdot t^{(m+1,-1)}\in\bLX_\mu$.
	This proves the statement.
\end{proof}

In the rest of the subsection, we prove Proposition \ref{p:GL2/O2 closure curve} in a few steps.

\subsubsection{}
Write $a=\sum_{i\geq0}a_it^i,b=\sum_{i\geq0}b_it^i,
c=\sum_{i\geq0}c_it^i,d=\sum_{i\geq0}d_it^i$.
We first show that we can find $a,b,c,d$ 
such that \refeq{eq:GL2/O2 closure} holds modulo $t^{m+1}$.
This only involves $b_0,...,b_{m+1},d_0,...,d_{m+1}$.
Explicitly, we need to find $b_0,...,b_{m+1},d_0,...,d_{m+1}\in k[s,s^{-1}]$ depending only on $m$ such that
\begin{equation}\label{eq:step 1 eqs}
	b^2\equiv \sum_{i=0}^{m+1}(-s^2)^{m+1-i}t^i,\quad
	d^2\equiv s^2+t,\quad
	bd\equiv \mu_0 s^{m+2} \mod t^{m+2}.
\end{equation}

We first solve $d_i$. 
Let $d_i=\lambda_i s^{-(2i-1)}$, $0\leq i\leq m+1$.
Then $d\equiv \sum_{i=0}^{m+1}d_it^i\mod t^{m+2}$.
We get equation 
\[
d^2\equiv 
\lambda_0^2s^2+2\lambda_0\lambda_1s^0t+(2\lambda_0\lambda_2+\cdots)s^{-2}t^2
+\cdots+(2\lambda_0\lambda_{m+1}+\cdots)s^{-2m}t^{m+1}
\equiv s^2+t\mod t^{m+2}.
\]
Let $\lambda_0=1$, $d_0=s$. 
From the above we can uniquely solve $\lambda_1=\frac{1}{2},\lambda_2=\frac{1}{2^3},...$, 
until $\lambda_{m+1}$, 
where the denominator of $\lambda_i$ is a power of $2$.

Next let $b_i=\mu_is^{m+1-2i}$, $0\leq i\leq m+1$, where $\mu_0=\sqrt{(-1)^{m+1}}\neq0$.
From equation
\begin{align*}
bd\equiv& 
\mu_0s^{m+2}+(\sum_{i=0}^1\mu_i\lambda_{1-i})s^mt
+(\sum_{i=0}^2\mu_i\lambda_{2-i})s^{m-2}t^2+\cdots
+(\sum_{i=0}^{m+1}\mu_i\lambda_{m+1-i})s^{-m}t^{m+1}\\
\equiv& \mu_0s^{m+2} \mod t^{m+2},
\end{align*}
$\mu_1,\mu_2,...,\mu_{m+1}$ are uniquely determined.

We still need to check that the above 
$b_0,...,b_{m+1},d_0,...,d_{m+1}$ 
satisfy the first equation in \eqref{eq:step 1 eqs}:
\[
b^2\equiv
\sum_{i=0}^{m+1}(\sum_{k=0}^i\mu_k\mu_{i-k})(s^2)^{m+1-i}t^i\equiv
\sum_{i=0}^{m+1}(-s^2)^{m+1-i}t^i
\mod t^{m+2}.
\]
It amounts to show 
$\sum_{k=0}^i\mu_k\mu_{i-k}=(-1)^{m+1-i}$, 
$\forall\ 0\leq i\leq m+1$.
For this,
first it is easy to prove by induction on $i$ 
using the construction of $\mu_i,\lambda_i$ that 
$\mu_i=(\sum_{k=0}^i(-1)^k\lambda_{i-k})\mu_0$, $\forall\ 0\leq i\leq m+1$.
Next we prove $N_i:=\sum_{k=0}^i\mu_k\mu_{i-k}=(-1)^{m+1-i}$ by induction.
Assume this is true for $0\leq i\leq r<m+1$, 
where $N_0=\mu_0^2=(-1)^{m+1}$. 
In $bd\equiv \mu_0s^{m+2}\mod t^{m+2}$,
we time both sides by $b$ and evaluate at $s=1$, 
we get
\[
(b|_{s=1})^2d|_{s=1}\equiv \sum_{i=0}^{m+1}N_i t^i\sum_{k=0}^{m+1}\lambda_i t^i
\equiv \mu_0\sum_{i=0}^{m+1}\mu_it^i\mod t^{m+2}.
\]
Compare the coefficient of $t^{r+1}$ and using the inductive hypothesis, 
we get
\[
\mu_0^2(\lambda_{r+1}-\lambda_r+\cdots+(-1)^r\lambda_1)+N_{r+1}\lambda_0=\mu_0\mu_{r+1}.
\]
By the previous identity 
$\mu_i=(\sum_{k=0}^i(-1)^k\lambda_{i-k})\mu_0$,
we obtain 
\[
N_{r+1}
=\mu_0(\mu_{r+1}-\mu_0(\lambda_{r+1}-\cdots+(-1)^r\lambda_1))
=(-1)^{r+1}\mu_0^2
=(-1)^{m+1-r-1}.
\]
Thus we have verified that the above 
$b_0,...,b_{m+1},d_0,...,d_{m+1}$ 
satisfy \eqref{eq:step 1 eqs}.

\subsubsection{}
Next we show that we can further find 
$a,c,b_i,d_i,i\geq m+2$ 
such that \eqref{eq:GL2/O2 closure} holds.
Denote 
$b_{\leq m+1}=\sum_{i=0}^{m+1}b_it^i$,
$\tdb_i=b_{m+1+i}$,
$\tdb=\sum_{i\geq 0}\tdb_it^i$,
then $b=b_{\leq m+1}+\tdb t^{m+2}$.
Similarly denote $d_{\leq m+1},\tdd,\tdd_i,d=d_{\leq m+1}+\tdd t^{m+2}$.
Then
\begin{align*}
	b^2&=(\sum_{i=0}^{m+1}(-s^2)^{m+1-i}t^i)+(\tdb^2t^{m+2}+2b_{\leq m+1}\tdb+\widetilde{b^2_{\leq m+1}})t^{m+2},\\
	d^2&=(s^2+t)+(\tdd^2t^{m+2}+2d_{\leq m+1}\tdd+\widetilde{d^2_{\leq m+1}})t^{m+2},\\
	bd&=(\mu_0s^{m+2})+(\tdb\tdd t^{m+2}+b_{\leq m+1}\tdd+d_{\leq m+1}\tdb+\widetilde{b_{\leq m+1}d_{\leq m+1}})t^{m+2},
\end{align*}
where $\widetilde{b^2_{\leq m+1}},\widetilde{d^2_{\leq m+1}},\widetilde{b_{\leq m+1}d_{\leq m+1}}$
are polynomials in terms of $t,b_1,...,b_{m+1},d_1,...,d_{m+1}$
that fit into the above identities.
The problem is reduced to finding $a,c,\tdb,\tdd\in\cO$ satisfying the following equations:
\begin{align*}
	a)\quad &a^2+\tdb^2t^{m+2}+\widetilde{b^2_{\leq m+1}}+2b_{\leq m+1}\tdb=0,\\
	b)\quad &c^2+\tdd^2t^{m+2}+\widetilde{d^2_{\leq m+1}}+2d_{\leq m+1}\tdd=0,\\
	c)\quad &ac+\tdb\tdd t^{m+2}+\widetilde{b_{\leq m+1}d_{\leq m+1}}+b_{\leq m+1}\tdd+d_{\leq m+1}\tdb=0.
\end{align*}

First we let $a=s^{-1}$. 
Since the constant term of $b_{\leq m+1}$ is $b_0=\mu_0s^{m+1}\neq0$, 
we can solve $\tdb$ uniquely from equation a).
From equation c), we can solve $c$ in terms of $\tdd$.
Now plug the $c$ into equation b),
we obtain an equation of $\tdd$:
\begin{align*}
	 &(t^{m+2}+s^2(\tdb t^{m+2}+b_{\leq m+1})^2)\tdd^2\\
	+&(2s^2(\widetilde{b_{\leq m+1}d_{\leq m+1}}+\tdb d_{\leq m+1})(\tdb t^{m+2}+b_{\leq m+1})+2d_{\leq m+1})\tdd\\
	+&(s^2(\widetilde{b_{\leq m+1}d_{\leq m+1}}+\tdb d_{\leq m+1})^2+\widetilde{d^2_{\leq m+1}})=0.
\end{align*}

To guarantee a solution of $\tdd$, 
it suffices to show that the determinant $\Delta$ 
of the above quadratic equation in $\tdd$ 
has a square root in $\cO^\times$,
for which it suffices to show $\Delta$ has nonzero constant term.
Denote the constant term of $\widetilde{X}$ by $\widetilde{X}^0$, $X=b^2_{\leq m+1},d^2_{\leq m+1},b_{\leq m+1}d_{\leq m+1}$.
Note that $\tdb_0=-\frac{s^{-2}+\widetilde{b^2_{\leq m+1}}^0}{2b_0}$.
Then the constant term of the determinant is
\begin{align*}
\Delta_0&=(2s^2(\widetilde{b_{\leq m+1}d_{\leq m+1}}^0+\tdb_0d_0)b_0+2d_0)^2
-4s^2b_0^2(s^2(\widetilde{b_{\leq m+1}d_{\leq m+1}}^0+\tdb_0 d_0)^2+\widetilde{d^2_{\leq m+1}}^0)\\
&=4s^2(2b_0d_0\widetilde{b_{\leq m+1}d_{\leq m+1}}^0-\widetilde{b^2_{\leq m+1}}^0d_0^2-b_0^2\widetilde{d^2_{\leq m+1}}^0)\\
&=4s^2(2\mu_0\lambda_0\sum_{i=1}^{m+1}\mu_i\lambda_{m+2-i}-\sum_{i=1}^{m+1}\mu_i\mu_{m+2-i}\lambda_0^2-\mu_0^2\sum_{i=1}^{m+1}\lambda_i\lambda_{m+2-i})\\
&=-4s^2\mu_0^2\sum_{i=1}^{m+1}(\mu_i\mu_0^{-1}-\lambda_i)(\mu_{m+2-i}\mu_0^{-1}-\lambda_{m+2-i})\\
&=-4s^2\sum_{i=0}^m\mu_i\mu_{m-i}.
\end{align*}
The last identity above used 
$\mu_i=(\sum_{k=0}^i(-1)^k\lambda_{i-k})\mu_0$.
By induction on $m$, it is easy to show $\sum_{i=0}^m\mu_i\mu_{m-i}=(-1)^m\mu_0^2=-1$.
Thus $\Delta_0=4s^2\neq 0$ for $s\neq 0$.
Therefore, we can solve $\tdd$ out of the above equation 
unique up to a choice of the constant term $\tdd_0$
satisfying a quadratic equation.

\subsubsection{}
Now we have found $a,b,c,d\in\cO$ satisfying \eqref{eq:GL2/O2 closure}.
It remains to show $g\in\SL_2(\cO)$,
i.e. $\det(g)=ad-bc=1$.
In fact, taking determinant on both sides of \eqref{eq:GL2/O2 closure}, 
we obtain
\[
\det(g)^2t^m
=(\sum_{i=0}^{m+1}(-s^2)^{m+1-i}t^{i-1})(s^2t^{-1}+1)-\mu_0 s^{m+2}t^{-1}\mu_0s^{m+2}t^{-1}
=t^m.
\]

Thus $\det(g)=\pm 1$.
We can choose the constant term of $\tdd$ such that the constant term of $ad-bc$ is $1$, 
then $ad-bc=1$.
This completes the proof of Proposition \ref{p:GL2/O2 closure curve}.

\subsubsection{Proof of Theorem \ref{t:sym var orbit closure}}\label{Orbits closure for general symmetric varieties}
As in the proof of Proposition \ref{c:GL2/O2 orbits closure},
we know $\bLX_\mu\subset\bigsqcup_{\lambda\leq\mu}LX_\lambda$.
It remains to show that for $\lambda<\mu$, $\lambda\in\Lambda_S^+$,
we have $x_\lambda\in\bLX_\mu$.
By \cite[Lemma 2.1.1]{NadlerRealGr},
it suffices to prove for two cases:
$\mu-\lambda=\alpha$ is a positive coroot $\alpha\in\Lambda_S^+\cap\check\Phi^+$,
or $\mu-\lambda=\alpha-\theta(\alpha)$ is not a multiple of coroots where $\alpha\in\check\Phi^+$.

Case 1: $\mu-\lambda=\alpha\in\Lambda_S^+\cap\check\Phi^+$.

Since $\theta$ acts on $S$ as inverse, 
it acts as $-1$ on $\Lambda_S$.
As $\alpha=\mu-\lambda\in\Lambda_S$,
we have $\theta(\alpha)=-\alpha$.
The Levi subgroup $G_\alpha$ corresponding to $\alpha$
is preserved by $\theta$.
Under the isomorphism 
$G_\alpha\simeq\SL_2\times\bGm^r,\GL_2\times\bGm^r$, or $\PGL_2\times\bGm^r$,
consider the action of $\theta$ on the first factor,
we can see it has the same action on roots 
as transpose inverse $\theta_0(g)=(g^t)^{-1}$.
Thus $\theta$ and $\theta_0$ different by 
an inner automorphism by some $t\in T$.
Let $r\in T$ be such that $r\theta_0(r)^{-1}=r^2=t$,
then $\theta=\Ad_t\circ\theta_0=\Ad_r\circ\theta_0\circ\Ad_{r^{-1}}$.
Since $x_\lambda,x_\mu\in L(G_\alpha/(G_\alpha)^\theta)\hookrightarrow LX$
and $\Ad_r$ fixes $t^\lambda,t^\mu$,
it suffices to prove for $G=\SL_2,\GL_2$, or $\PGL_2$ 
and $\theta=\theta_0:g\mapsto(g^t)^{-1}$.
Thus the closure relation $x_\lambda\in\bLX_\mu$ follows from 
Corollary \ref{c:GL2/O2 orbits closure}.
Note that the corollary applies to $\SL_2$ because 
in Proposition \ref{p:GL2/O2 closure curve} 
we only used action of $\SL_2(\cO)$.

Case 2: $\mu-\lambda=\alpha-\theta(\alpha)$ is not a multiple of coroots, and $\alpha+\theta(\alpha)$ is also not a coroot.

In this case $\alpha\neq\pm\theta(\alpha)$. 
Let $H$ be the subgroup generated by 
Levi subgroups $G_\alpha$ and $G_{\theta(\alpha)}$.
The root system of $H$ is of type $A_1\times A_1$.
Fix isomorphisms of root subgroups
$U_{\pm\alpha}\simeq\bGa,U_{\pm\theta(\alpha)}\simeq\bGa$.
Consider the following subgroups of $H$:
\[
U_+=\mathrm{Im}(\bGa\xrightarrow{\Delta}\bGa\times\bGa\simeq U_\alpha\times U_{-\theta(\alpha)}),\quad
U_-=\mathrm{Im}(\bGa\xrightarrow{\Delta}\bGa\times\bGa\simeq U_{-\alpha}\times U_{\theta(\alpha)}).
\]
Then $\theta(U_+)=U_-$.
Let $A$ be the subgroup generated by $U_+,U_-$,
then $A\simeq \SL_2$ or $\mathrm{PSL}_2$ with unique coroots $\pm(\alpha-\theta(\alpha))$.
Similar as in Case 1, 
we can assume $\theta$ acts as transpose inverse on $A$.
Consider the action of $A(\cO)$ on $L(H/H^\theta)\hookrightarrow LX$,
then again by Proposition \ref{p:GL2/O2 closure curve} 
we conclude $x_\lambda\in\bLX_\mu$.

Case 3: $\mu-\lambda=\alpha-\theta(\alpha)$ is not a multiple of coroots, but $\alpha+\theta(\alpha)$ is a coroot.

Let $H$ be the subgroup generated by 
Levi subgroups $G_\alpha,G_{\theta(\alpha)}$.
By \cite[Lemma 2.5]{Springer85}, the root system of $H$ is of type $A_2$,
with coroots $\pm\alpha,\pm\theta(\alpha),\pm(\alpha+\theta(\alpha))$.
Observe that $\theta$ exchanges $\alpha,\theta(\alpha)$, 
and fixes $\alpha+\theta(\alpha)$.
Therefore the statement is reduced to the case of
$G=\GL_3, \SL_3$, or other isogeny types, $T$ is the diagonal, 
$J=\begin{pmatrix}
	& &1\\
	&1& \\
   1& &	
\end{pmatrix}$,
and $\theta$ has the same action on roots as involution
$\theta_0(g)=\Ad_J(g^t)^{-1}$.
Similar as in Case 1, 
we may assume $\theta=\theta_0$.
Note
\[
S=T^{\mathrm{inv}\circ\theta}
=\{\begin{pmatrix}
	a&&\\
	&b&\\
	&&a
\end{pmatrix}\,|\,
a,b\in k^\times\}.
\]
It suffices to prove for $\mu=(p,q,p),\lambda=(p-1,q+2,p-1)$, where $p>q+2$.

Let $\theta_0(g)=(g^t)^{-1}$,
then $\theta=\Ad_J\circ\theta_0=\theta_0\circ\Ad_J$.
Let $i=\sqrt{-1}\in k$.
Consider matrix
\[
h=\begin{pmatrix}
	-i&0&1\\
	0&1&0\\
	\frac{i}{2}&0&\frac{1}{2}
\end{pmatrix}.
\]
Then $J=hh^t=h\theta_0(h)^{-1}$, 
$\theta=\Ad_{\theta_0(h)}\circ\theta_0\circ\Ad_{\theta_0(h)^{-1}}$.
Observe that $\Ad_{\theta_0(h)}$ acts trivially on $S$.
Thus in order to show that the $t^\lambda$ 
lies in the orbit closure of $t^\mu$ with respect to the $\theta$-twisted conjugation of $\SL_3(\cO)$, 
it is equivalent to show that $\Ad_{\theta_0(h)^{-1}}t^\lambda=t^\lambda$ 
is in the orbit closure of $\Ad_{\theta_0(h)^{-1}}t^\mu=t^\mu$ with respect to the $\theta_0$-twisted conjugation of $\SL_3(\cO)$.
The later follows from applying Proposition \ref{p:GL2/O2 closure curve} twice.
This completes the proof.

\section{Iwahori orbits}\label{S:Iwahori orbits}

We will assume $K$ is connected  from Section \ref{S:Iwahori orbits}-Section \ref{Formality}.
In the last Section \ref{disconnected cases}, we 
discuss the case when $K$ is disconnected.

\subsection{Iwahori orbits parametrization}\label{s:Iwahori orbits}
Let $I_0\subset LG$ be an Iwahori subgroup 
assoicated to the $\theta$-stable Borel 
$B_0$.
We consider $I_0$-orbits on $LX$. 
Many properties of $B_0$-orbits on $X$ have been studied in \cite{Springer85, RSorder,MS}.
Our goal is to generalize these finite dimensional results to the affine setting.

In the following, we denote the pro-unipotent radical of an Iwahori subgroup $I$ by $I^+$. 
For a split group $H$, we denote the pro-unipotent radical of $L^+H=H(\cO)$ by $L^{++}H$.
Recall $T_0\subset B_0$ 
is a $\theta$-stable maximal torus 
contained in a $\theta$-stable Borel subgroup.
Then $L^+T_0,LT_0$ are $\theta$-stable.
The preimage of $B_0$ under the evaluation map 
$L^+G\rightarrow G$
defines a $\theta$-stable Iwahori subgroup $I_0$
that contains $L^+T_0$.

For every split maximal torus $T\subset LG$,
via the canonical isomorphism $X_*(T)\otimes_\bZ F^\times\simeq T$,
we can define a canonical integral model
$L^+T:=X_*(T)\otimes_\bZ\cO^\times\hookrightarrow T$.
Denote $N_0=N_{LG}(I_0)$.
Note that $N_0\backslash LG$ parametrizes Iwahori subgroups via
$g\mapsto \Ad_{g^{-1}}I_0$.
Consider sets
\[
\cC=\{(T,I)\ |\ I\in N_0\backslash LG,\ L^+T\subset I\},\quad
\cC^\theta=\{(T,I)\in\cC\ |\ \theta(T)=T\},
\]
where $T$ is any split maximal torus of $LG$. 
Note that $K$ acts on $\cC^\theta$ via adjoint action.
We can define a map
\[
\phi:LG\rightarrow\cC,\quad g\mapsto(\Ad_{g^{-1}} LT_0,\Ad_{g^{-1}} I_0).
\]
Note that
\[
\cV
:=\phi^{-1}(\cC^\theta)
=\{g\in LG\ |\ \tau(g)=g\theta(g)^{-1}\in N_{LG}(LT_0)\},
\]
and $\phi|_\cV$ factors through the quotient $L^+T_0\backslash\cV$ 
by left multiplication of $L^+T_0$.

We will need the following presumably well-known fact:
\begin{lem}\label{l:tori in Iwahori}
	For any Iwahori subgroup $I$, 
	all the maximal tori in $I$ can be conjugated to each other
	by an element of pro-unipotent radical $I^+$.
\end{lem}
\begin{proof}
	By conjugation, we can assume $I=I_0$ is the standard Iwahori,
	where $I_0=T_0I_0^+$.
	It suffices to show that any maximal torus
	$H\subset I_0$ can be conjugated to $T_0$ by $I_0^+$. 
	
	First for any $m>0$, 
	images of $H,T_0$ in $I_0/L^mG$ are conjugated by some $u_m\in I_0^+$,
	i.e. $\Ad_{u_m}H\equiv T_0\mod L^mG$.
	Next the images of $\Ad_{u_m}H,T_0$ in $I_0/L^{m+1}G$ are conjugated by some $g_m\in I_0^+$,
	i.e. $\Ad_{g_m}\Ad_{u_m}H\equiv T_0\mod L^{m+1}G$.
	Priori, $\Ad_{g_m}\Ad_{u_m}H\equiv \Ad_{g_m}T_0\equiv T_0\mod L^mG$.
	Let $\bar{g}_m$ be the image of $g_m$ in $I_0/L^mG$,
	then $\bar{g}_m\in N_{I_0/L^mG}(T_0)$. 
	As $T_0$ is the maximal torus of solvable group $I_0/L^mG$,
	we have $\bar{g}_m\in N_{I_0/L^mG}(T_0)=C_{I_0/L^mG}(T_0)$.
	Observe that 
	$C_{I_0/L^{m+1}G}(T_0)=T_0(\cO/t^{m+1})\twoheadrightarrow C_{I_0/L^mG}(T_0)=T_0(\cO/t^m)$.
	Thus we can write $g_m=a_mb_m$, $a_m\in I_0^+$ has image in $C_{I_0/L^{m+1}G}(T_0)$, $b_m\in L^{m+1}G$.
	Let $u_{m+1}=b_mu_m\in I_0^+$, then 
	$\Ad_{u_{m+1}}H\equiv T_0\mod L^{m+1}G$.
	Repeat the above, we get a convergent sequence $u_m\in I_0^+$, $m>0$, with limit $u\in I_0^+$, such that
	$\Ad_u H\equiv \Ad_{u_m}H\equiv T_0\mod L^mG$, $\forall m$.
	Thus $\Ad_u H=T_0$.
\end{proof}

Now we give the parametrization of orbits.

\begin{lem}\label{l:Iwahori orbirs on LX}
	\begin{itemize}
		\item [(i)] We have bijection 
		$\bar{\phi}:(LT_0\cap N_0)\backslash\cV\simeq\cC^\theta$.
		
		\item [(ii)] The projection 
		$p:\cC^\theta\rightarrow N_0\backslash LG, (T,I)\mapsto I$ 
		induces a bijection 
        $\cC^\theta/LK\simeq N_0\backslash LG/LK$
        on $LK$-orbits. 
		
		\item [(iii)] 
        The inclusion $\mathcal V\to LG$
        induces bijections
		$\sV:=L^+T_0\backslash\cV/LK\simeq 
        I_0\backslash LG/LK\simeq
        I_0\backslash LX$.
	\end{itemize}
\end{lem}
\begin{proof}
	(i): Clearly $\phi$ is surjective, so is its restriction $\bar{\phi}$. 
	To see it is injective, suppose $\phi(g_1)=\phi(g_2)$. 
	Then $g_2g_1^{-1}$ normalizes $LT_0$ and $I_0$,
        thus belongs to $LT_0\cap N_0$.
	
	(ii): 
        First we show the surjectivity. 
	For any Iwahori $I$, 
        we need to find a $\theta$-stable maximal torus $T\subset LG$ 
	  with $L^+T\subset I$. 
	Note that $\theta$ acts on $I\cap\theta(I)$
        and that the intersection of two Iwahori subgroups $I,\theta(I)$
        contains a torus $A$ of maximal rank in $I$.
        Denote by $I_m$ the $m$th Moy-Prasad subgroup of $I_m$.
        Observe that $I_m\cap\theta(I_m)$ is a 
        $\theta$-stable normal subgroup of
        $I\cap\theta(I)$ of finite codimension.
        Thus $\theta$ acts on the finite dimensional solvable group 
	$I\cap\theta(I)/I_m\cap\theta(I_m)$.
        Let $m$ be large enough such that
        $A$ is isomorphic to its image in
        $I\cap\theta(I)/I_m\cap\theta(I_m)$,
        so that the quotient group is of the same rank as $I$.
	By \cite[Theorem 7.5]{Steinberg}, there exists 
	a $\theta$-stable maximal split torus 
        $T_m\subset I\cap\theta(I)/I_m\cap\theta(I_m)$.
	We can lift $T_m$ to $LG$ as follows.
	Let $H_m$ be the preimage of $T_m$ in $I\cap\theta(I)$.
	In $H_m/I_{m+1}\cap\theta(I_{m+1})$, 
        we can similarly find maximal torus $T_{m+1}$,
	and take its preimage $H_{m+1}$ in $H_m$.
	Repeat this and let $T'=\cap_m H_m$, 
	then $T'\simeq T_m$ is a $\theta$-stable torus in 
        $I\cap\theta(I)$. 
	Moreover, $T=C_{LG}(T')$ is a 
        split $\theta$-stable maximal torus of $LG$.
	In fact, note that this is true if $T'\subset G\subset LG$.
	Since $T'$ can be conjugated into the standard Iwahori $I_0$,
	and by Lemma \ref{l:tori in Iwahori} we know
	all the tori in $I_0$ can be conjugated into $G$,
	we obtain that $T$ is split.
	Thus $T$ satisfies the requirements.

	For the injectivity, suppose $(I,T),(I,T')\in\cC^\theta$.
	We need to find $g\in LK$ such that $\Ad_gI=I$, $\Ad_g T=T'$.
	Here $L^+T,L^+T'\subset I\cap\theta(I)$.
	For the split torus $T\simeq X_*(T)\otimes_\bZ F^\times$,
	consider the constant subtorus 
        $\mathring{T}:=X_*(T)\otimes_\bZ k^\times\subset T$.
	Then $T\simeq L\mathring{T}$, $L^+T\simeq L^+\mathring{T}$.
	Similarly define $\mathring{T}'\subset T'$.
	By the same proof of Lemma \ref{l:tori in Iwahori}, 
	$\mathring{T},\mathring{T}'$ can be conjugated by 
	an element of the pro-unipotent radical 
	$(I\cap\theta(I))_u$ of $I\cap\theta(I)$.
	Let $g_1\in(I\cap\theta(I))_u$ be
	such that $\Ad_{g_1}\mathring{T}=\mathring{T}'$.
	Then $\Ad_{g_1}L^+T=L^+T'$, $\Ad_{g_1} T=\Ad_{g_1} T'$.
	Since $T,T'$ are $\theta$-stable, we get 
	$g_1^{-1}\theta(g_1)\in(N(T)\cap I^+)^{\on{inv}\circ\theta}= 
        (L^{++}T)^{\on{inv}\circ\theta}$.
	Since $(L^{++}T)^{\on{inv}\circ\theta}$ is pro-unipotent, 
	we can find $t_1\in (L^{++}T)^{\on{inv}\circ\theta}$
	such that $t_1\theta(t_1^{-1})=t_1^2=g_1^{-1}\theta(g_1)$.
	Thus $g:=g_1t_1\in LK\cap I$ maps $(I,T)$ to $(I,T')$.
	
	(iii): From (i), (ii)
        and Proposition \ref{p:sym var orbits} (ii),
        we obtain bijection
        \[
        (LT_0\cap N_0)\backslash\cV/LK\simeq N_0\backslash LG/LK.
        \]
        We have commutative diagram
        \[\xymatrix{
        L^+T_0\backslash\cV/LK\ar[d]\ar[r]^{\iota}& 
        I_0\backslash LG/LK\ar[d]\\ 
        (LT_0\cap N_0)\backslash\cV/LK\ar[r]^{\iota'}&N_0
        \backslash LG/LK.}
        \]
        
        Let $Z_G=Z(G)$.
        Observe that $LT_0\cap N_0=L^+T_0 LZ_G$.
        In fact, we can write $LT_0=t^{X_*(T_0)}L^+T_0$.
        Then $t^{X_*(T_0)}\cap N_0$ consists of $t^\lambda$
        for $\langle\pm\lambda,\alpha\rangle\geq0$, 
        $\alpha$ any simple root.
        Thus $\langle\lambda,\alpha\rangle=0$, $t^\lambda\in LZ_G$.

        We first show that $\iota$ is injective.
        Given $u,v\in\cV$ such that $\iota(u)=\iota(v)$.
        By the injectivity of $\iota'$,
        we have $u=tzvk$
        for some $t\in L^+T_0$, $z\in LZ_G$, $k\in LK$.
        Then 
        $\iota(u)=I_0zvLK=I_0vLK=\iota(v)$.
        Taking $\tau$, we obtain that
        $\tau(u)\tau(z)\equiv \tau(u)\in\widetilde{\mathrm W}=N_{LG}(LT_0)/L^+T_0$.
        We deduce that $\tau(z)\in L^+T_0\cap LZ_G=L^+Z_G$,
        $z\in L^+Z_G L(Z_G\cap K)$.
        We obtain $L^+T_0uLK=L^+T_0vLK$.

        Next we show $\iota$ is surjective.
        From the surjectivity of $\iota'$,
        any $I_0$-orbit $I_0gLK$ satisfies 
        $g=nvk$ for some $v\in\cV$, $n\in N_0$, $k\in LK$.
        Since $N_0/I_0\simeq (N_{LG}(LT_0)\cap N_{LG}(I_0))/L^+T_0$,
        we can write $n=he$, $h\in I_0$, $e\in N_{LG}(I_0)\cap N_{LG}(LT_0)$.
        Note $\tau(ev)\in N_{LG}(LT_0)$.
        Thus $I_0gLK=\iota(ev)$
        for $ev\in\cV$.
\end{proof}

\begin{defe}
    For any $v\in\sV:=L^+T_0\backslash\cV/LK$, we let $\cO_v\subset LX$
    be the corresponding $I_0$-orbit on $LX$
    under the bijection in Lemma \ref{l:Iwahori orbirs on LX} (iii).
\end{defe}

\subsection{Twisted involutions in extended affine Weyl groups }
We denote by $\Delta_\aff$, $\Phi_\aff$, $\Phi_\aff^+$, $\Phi_\aff^-$ the set of 
simple affine roots,
 affine roots,
 positive affine roots, and negative affine roots with respect to $T_0$ and $I_0$.
Since $T_0, I_0$ are $\theta$-stable,
$\theta$ acts on $\Phi_\aff$ and $\Phi_\aff^+$.

Let $\widetilde{\mathrm W}=N_{LG}(LT_0)/L^+T_0$ be the extened 
affine Weyl group. 
We denote by $\mathrm W_\aff\subset\widetilde{\mathrm W}$
the affine Weyl group generated by the simple reflections $\{s_\alpha\}_{\alpha\in\Delta_\aff}$
and 
$\Omega=\widetilde{\mathrm W}/{\mathrm W_\aff}$ the quotient.
We denote by $l:\widetilde{\mathrm W}\to\mathbb Z_{\geq0}$ the length function.

The natural map $N_{LG}(LT_0)\to LG$
induces a bijection 
$\widetilde{\mathrm W}=
N_{LG}(LT_0)/L^+T_0\cong I_0\backslash LG/I_0$.
The composition 
 $N_{LG}(I_0)/I_0\subset I_0\backslash LG/I_0\cong\widetilde{\mathrm W}\to\Omega$ induces an
isomorphism of groups
 $N_{LG}(I_0)/I_0\cong\Omega$, which gives rise to an isomorphism 
 $\widetilde{\mathrm W}\cong\mathrm W_\aff\rtimes\Omega$
 where $\Omega=\{w\in\widetilde{\mathrm W}|l(w)=0\}$ consists of elements of length zero.

Since  $L^+T_0$ and $N_{LG}(LT_0)$ are $\theta$-stable, we get an induced action of $\theta$ on the 
extended affine Weyl group  
$\widetilde{\mathrm W}$.
Let $\mathscr I=\{w\in\widetilde{\mathrm W}|\theta(w)=w^{-1}\}$.
We say that the elements in $\mathscr I$ are twisted involutions.

Note that the map 
$\tau:G\to G$, $\tau(g)=g\theta(g)^{-1}$
induces a map
$p:\sV=L^+T_0\backslash\cV/LK\to\mathscr I$
from the set of $I_0$-orbits on $LX$ to the set of twisted
involutions in the extended affine Weyl group.
We study some basic properties of it.

We first need a lemma.
\begin{lem}\label{l:Iwahori orbit rep}
	Any coset $v\in\sV$ 
        contains a representative
	$\dv\in\cV$ such that 
	$\tau(\dv)=n_v=\bar{n}_vt^\lambda$,
	where $\bar{n}_v\in N_{G}(T_0)$ and $\lambda\in X_*(T_0)$
	satisfy $\theta(\bar{n}_v)=\bar{n}_v^{-1},\
	\Ad_{\bar{n}_v}t^{-\lambda}=t^{\theta(\lambda)}$.
\end{lem}
\begin{proof}
	We can uniquely write 
        $\tau(\dv)=\dv\theta(\dv)^{-1}\in N_{LG}(LT_0)$ as
	$\tau(\dv)=\bar{n}_vt^\lambda\gamma$, 
        where $\bar{n}_v\in N_{G}(T_0),\lambda\in X_*(T_0),\gamma\in L^{++}T_0$.
	Since $\tau(\dv)\in(LG)^{\on{inv}\circ\theta}$, we have
	\[
	\theta(\tau(\dv))=\theta(\bar{n}_v)t^{\theta(\lambda)}\theta(\gamma)
	=(\tau(\dv))^{-1}=\gamma^{-1}t^{-\lambda}(\bar{n}_v)^{-1}
	=(\bar{n}_v)^{-1}t^{-\Ad_{\bar{n}_v}\lambda}\Ad_{\bar{n}_v}\gamma^{-1}.
	\]
	By uniqueness, we get 
	$\theta(\bar{n}_v)=\bar{n}_v^{-1}$,
	$t^{\theta(\lambda)}=t^{-\Ad_{\bar{n}_v}\lambda}$,
	$\theta(\gamma)=\Ad_{\bar{n}_v}\gamma^{-1}$.
	Denote $\theta_v:=\Ad_{\bar{n}_v^{-1}}\circ\theta$,
	then $\theta_v(\gamma)=\gamma^{-1}$.
	Since $\theta_v$ acts on the pro-unipotent group
	$L^{++}T_0$ as an involution, 
	there exists 
	$\beta\in (L^{++}T_0)^{\on{inv}\circ\theta_{\bar{v}}}$ 
	such that $\beta^2=\gamma$, i.e. 
	$\beta\theta_{\bar{v}}(\beta)^{-1}=\gamma$.
	Then $\dv'=\theta(\beta)\dv$ satisfies 
	$\tau(\dv')=\bar{n}_vt^\lambda$, giving the desired representative.
\end{proof}

\begin{lem}\label{l:finite fiber}
The map  
$p:\sV\to \mathscr{I}$
has finite fibers.
\end{lem}
\begin{proof}
(i): 
For $u,v\in \sV=L^+T_0\backslash\cV/LK$
with representatives $\du,\dv$ as in Lemma \ref{l:Iwahori orbit rep},
suppose $\tau(\du)=\bar{n}_ut^\lambda,\tau(\dv)=\bar{n}_vt^\mu$ 
have the same images in $\widetilde{\mathrm W}$.
Then $\lambda=\mu$ and $\bar{n}_u=c\bar{n}_v$ for some $c\in T_0$.
Since $\tau(\du),\tau(\dv)\in (LG)^{\on{inv}\circ\theta}$, we have
\[
\theta(c\tau(\dv))=\theta(c)\tau(\dv)^{-1}=\theta(c)t^{-\mu}\bar{n}_v^{-1}
=\tau(\dv)^{-1}c^{-1}=t^{-\mu}\bar{n}_v^{-1}c^{-1}.
\]
Thus $\psi_{\bar v}(c)=c^{-1}$ 
where $\psi_{\bar v}:=\Ad_{\bar{n}_v}\circ\theta$ is an involution.
Denote $\tau_{\bar v}(g)=g\psi_{\bar v}(g)^{-1}$.
Observe that $T_0$ is $\psi_{\bar v}$-stable.
Write $T_0^{\on{inv}\circ\psi_{\bar v}}
=\bigsqcup_i t_iT_0^{\on{inv}\circ\psi_{\bar v},\circ}$,
where $T_0^{\on{inv}\circ\psi_{\bar v},\circ}=\tau_{\bar v}(T_0)$.
There exists $a\in T_0$ such that $c=\tau_{\bar v}(a)t_i$ for some $t_i$.
Thus $\tau(a\du)=t_i\tau(\dv)$.
The preimage of $p$ over the image of $\tau(\dv)$ is at most
the number of connected components of $T_0^{\on{inv}\circ\psi_{\bar v}}$,
which is finite.
\end{proof}

\begin{cor}\label{c:Iwahori closure and G(O) into Iwahori}
	\begin{itemize}
		\item [(i)]
		The closure of each $I_0$-orbit contains only finitely many
		$I_0$-orbits.
		
		\item [(ii)]
		Each $L^+G$-orbit on $LX$ is a union of finitely many $I_0$-orbits.
	\end{itemize}
\end{cor}
\begin{proof}
	(i):
	For an $I_0$-orbit $\cO_v$,
	$\tau(\cO_v)\subset I_0n_vI_0$,
	$\tau(\overline{\cO}_v)\subset\overline{I_0n_vI_0}$.
	We know $\overline{I_0n_vI_0}$ is a finite union of double cosets.
	By Lemma \ref{l:finite fiber},
	the preimage under $\tau$ of each double coset consists of finitely many
	$I_0$-orbit.
	Part (i) follows.
	
	(ii):
	Note that $\tau(L^+Gx_v)\subset L^+Gn_vL^+G$
	where $L^+Gn_vL^+G$ is a finite union of $I_0$-double cosets.
	Then part (ii) follows from Lemma \ref{l:finite fiber}.
\end{proof}

\subsection{Simple parahoric orbits}
In this section we collect some basic facts on Iwahori orbits.

For a simple affine root $\alpha\in\Delta_\aff$,
let $P_\alpha$ be the standard parahoric subgroup associated to $\alpha$.
Let $P_\alpha^+$ and $L_\alpha$ 
be the pro-unipotent radical and Levi subgroup of $P_\alpha$.
Denote by $s_\alpha\in W_\aff$ the simple reflection with respect to 
the hyperplane defined by $\alpha$, then
\[
P_\alpha=I_0\sqcup I_0s_\alpha I_0.
\]

Let $\cO_v=I_0\dv LK/LK=I_0x_v$ be an $I_0$-orbit on $LX$, 
where $\dv\in\cV$ is as in Lemma \ref{l:Iwahori orbit rep}.
Denote by $w_v=p(v)\in\mathscr I$ the associated twisted involution 
in the extended affine Weyl group.
Recall we associated to $v$ an involution $\psi_v=\on{Ad}_{n_v}\circ\theta$
so that $\Stab_{LG}(x_v)=(LG)^{\psi_v}$.
The stabilizer of $P_\alpha$ on $\dv$ is
$P_{\alpha,v}=P_\alpha^{\psi_v}$.
For $\alpha\in\Delta_\aff$,
observe
\[
I_0\backslash P_\alpha \cO_v\simeq 
I_0\backslash P_\alpha/P_{\alpha,v}\simeq
\bP^1/H_{\alpha,v}
\]
where $H_{\alpha,v}\subset\Aut(\bP^1)\simeq\mathrm{PGL}_2$ 
is the image of the right action of $P_{\alpha,v}$ 
on $I_0\backslash P_\alpha\simeq\bP^1$.

Denote by $\phi:P_\alpha\rightarrow\Aut(\bP^1)$ the action.
Let $U_{\pm\alpha}\subset P_\alpha$ 
be the affine root subgroups for $\pm\alpha$.
Choose isomorphisms $u_{\pm\alpha}:k\xrightarrow{\sim}U_{\pm\alpha}$
as in \cite[\S4.1.1]{MS},
so that $n_\alpha=u_\alpha(1)u_{-\alpha}(-1)u_\alpha(1)$ 
is a representative of $s_\alpha$.
Recall from \cite[\S4.1.3, \S4.1.4]{MS} the following classification of 
infinite $H_{\alpha,v}$:
\begin{lem}\label{l:orbit type}
	Up to conjugation by $\phi(I_0)$,
	any infinite subgroup $H=H_{\alpha,v}=\phi(P_\alpha)\subset\mathrm{PGL}_2$
	is one of the following classes:
	\begin{itemize}
		\item [(I)] $H=\mathrm{PGL}_2$,\hspace{5.2cm}
		$P_\alpha\cO_v=\cO_v$.
		
		\item [(IIa)] $\phi(U_\alpha)\subset H\subset \phi(I_0)$,\hspace{3.7cm}
		$P_\alpha\cO_v=\cO_v\sqcup\cO_{v'}$
		in which $\cO_{v'}$ is open.
		
		\item [(IIb)] $\phi(U_{-\alpha})\subset H\subset \phi(\Ad_{s_\alpha}I_0)$,\hspace{2.6cm}
		$P_\alpha\cO_v=\cO_v\sqcup\cO_{v'}$
		in which $\cO_v$ is open.
		
		\item [(IIIa)] $H=\phi(T_0)$,\hspace{5.3cm}
		$P_\alpha\cO_v=\cO_v\sqcup\cO_{v'}\sqcup
		\cO_{v''}$
		in which $\cO_{v''}$ is open.
		
		\item [(IIIb)] $H=\phi(\Ad_{n_\alpha u_{\alpha(-1)}}T_0)$,\hspace{3.5cm}
		$P_\alpha\cO_v=\cO_v\sqcup\cO_{v'}\sqcup
		\cO_{v''}$
		in which $\cO_v$ is open.
		
		\item [(IVa)] $H$ is the normalizer of $\phi(T)$,\hspace{2.1cm}
		$P_\alpha\cO_v=\cO_v\sqcup\cO_{v'}$
		in which $\cO_{v'}$ is open.
		
		\item [(IVb)] $H$ is the normalizer of 
		$\phi(\Ad_{n_\alpha u_{\alpha(-1)}}T_0)$,\hspace{0.2cm}
		$P_\alpha\cO_v=\cO_v\sqcup\cO_{v'}$
		in which $\cO_v$ is open.
	\end{itemize}
\end{lem}

We say an orbit $\cO_v$ is of type I, IIa,... for $s_\alpha$
if $H_{\alpha,v}$ is of type I, IIa,... as in the above.
Observe that $\cO_v$ is open in $P_\alpha\cO_v$
if and only if $\cO_v$ is of type I, IIb, IIIb, or IVb for $s_\alpha$.

Note that the involution $\psi_v=\Ad_{n_v}\circ\theta$ preserves $T_0$,
thus acts on the set of affine roots.
Since $\theta(n_v)=n_v^{-1}$,
the action of $\psi_v$ on affine roots is given by
$w_v\circ\theta=\theta\circ w_v^{-1}$.
By the same discussion as in \cite[\S6.5]{MS},
the type of $\cO_v$ for $s_\alpha$ can be read off
from this action as follows:
\begin{lem}\mbox{}\label{l:orbit type by psi}
	\begin{itemize}
		\item [(i)] $\cO_v$ is of type I for $s_\alpha$
		if and only if $\psi_v\alpha=\alpha$ and $\psi_v|_{U_{\pm\alpha}}=\mathrm{Id}$.
		
		\item [(ii)] $\cO_v$ is of type IIa or IIb for $s_\alpha$
		if and only if $\psi_v\alpha\not=\pm\alpha$.
		
		\item [(iii)] $\cO_v$ is of type IIIa or IVa for $s_\alpha$
		if and only if $\psi_v\alpha=\alpha$ and $\psi_v(g)=g^{-1}$
		on $U_{\pm\alpha}$.
		
		\item [(iv)] $\cO_v$ is of type IIIb or IVb for $s_\alpha$
		if and only if $\psi_v\alpha=-\alpha$.
		
		\item [(v)] If $\cO_v$ is of type I, IIa, IIIa, or IVa for $s_\alpha$, then $\psi_v\alpha>0$, i.e. $w_v^{-1}\alpha>0$.
		If $\cO_v$ is of type b, i.e. IIb, IIIb, or IVb 
		for $s_\alpha$, then $\psi_v\alpha<0$, i.e. $w_v^{-1}\alpha<0$.
	\end{itemize}
\end{lem}
In particular, the stabilizer $H_{\alpha,v}$ is always infinite.

\subsection{Characterization of closed orbits}
 We give a characterization of closed orbits.
To this end, we need the following useful lemma:
\begin{lem}\label{l:tau(P) image}
    Let $\psi$ be an involution on a parahoric subgroup $P$
    that preserves its unipotent radical $P^+$ and Levi subgroup $L_P$.
    Then the embedding $\tau(g)=g\psi(g)^{-1}$ has image
    $\tau(P)=P^{\on{inv}\circ\psi,\circ}$.
    In particular, the image is closed in $P$ and in $LG$.
\end{lem}
\begin{proof}
    It suffices to show the surjectivity.    
    For any $\ell u\in P^{\on{inv}\circ\psi,\circ}$,
    \[
    \psi(\ell u)=\psi(\ell)\psi(u)=(\ell u)^{-1}=\ell^{-1}\Ad_{\ell}u^{-1}.
    \]
    Denote $\psi_{\ell^{-1}}=\Ad_{\ell^{-1}}\circ\psi$.
    From the decomposition $P=L_PP^+$, we obtain
    $\psi(\ell)=\ell^{-1}$ and $\psi_{\ell^{-1}}(u)=u^{-1}$.
    This shows that the composition 
    $P^{\on{inv}\circ\psi}\hookrightarrow P\twoheadrightarrow L_P$
    lands inside $L_P^{\on{inv}\circ\psi}$.
    In particular, 
    the neutral component $P^{\on{inv}\circ\psi,\circ}$
    maps into $L_P^{\on{inv}\circ\psi,\circ}$,
    so that $\ell\in L_P^{\on{inv}\circ\psi,\circ}$.

    Since $(L_P,\psi)$ is a symmetric pair of finite type,
    we know $\tau(L_P)=L_P^{\on{inv}\circ\psi,\circ}$,
    so that $\ell=a\psi(a)^{-1}$ for some $a\in L_P$.
    Also, since $P^+$ is a pro-unipotent subgroup with 
    a sequence of Moy-Prasad subgroups stable under $\psi_{\ell^{-1}}$
    by the assumptions on $\psi$,
    we know the square map on $P^{+,\on{inv}\circ\psi_{\ell^{-1}}}$ is onto.
    Thus there exists 
    $b\in P^{+,\on{inv}\circ\psi_{\ell^{-1}}}$ with $u=b^2$.
    We have
    \[
    \tau(a\Ad_{\psi(a)^{-1}}(b))
    =a\psi(a)^{-1}b\psi(a)a^{-1}\psi(b)^{-1}a\psi(a)^{-1}
    =\ell b\psi_{\ell^{-1}}(b)^{-1}
    =\ell b^2=\ell u.
    \]
\end{proof}
As in the finite dimensional situation \cite[\S6.6]{MS},
we have the following important characterization of closed orbits:

\begin{lem}\label{l:orbit type b}
	For an $I_0$-orbit $\cO_v$, the following are  equivalent:
	\begin{itemize}
		\item [(i)] $\cO_v$ is a closed orbit,

		\item [(ii)] $n_v$ normalizes $I_0$,

        \item [(iii)]
        $w_v=p(v)\in\Omega$, that is,
        the twisted involution $w_v\in\mathscr{I}$
        has length zero,
		
		\item [(iv)] $\cO_v$ is not of type IIb, IIIb, or IVb 
		for any simple reflection $s_\alpha$.
	\end{itemize}
\end{lem}
\begin{proof}
$(ii)$ and $(iii)$
are equivalent since $\Omega\cong N_{LG}(I_0)/I_0$.
	We show implications $(i)\Rightarrow(iv)\Rightarrow(ii)\Rightarrow(i)$.
	
	Let $\cO_v$ be a closed orbit.
	If it is of type b for a simple reflection $s_\alpha$,
	then from the list in Lemma \ref{l:orbit type},
	the closure of $\cO_v$ contains at least one distinct orbit,
	so that it cannot be closed. 
	Thus (iv) must be satisfied.
	
	If $\cO_v$ satisfies (iv), 
	then by Lemma \ref{l:orbit type by psi}.(5),
	$w_v^{-1}\alpha>0$ for any simple affine root $\alpha$.
	Thus $w_v^{-1}$ and $w_v$ fix the fundamental alcove 
        and it follows that $n_v$ normalizes $I_0$.
	
	  If $n_v$ normalizes $I_0$, 
        we show the orbit $\cO_v$ is closed.
        The involution $\psi_v=\Ad_{n_v}\circ\theta$ acts on $I_0$
        and preserves $T_0\subset I_0$.
        Let $\tau_v(g)=g\psi_v(g)^{-1}$.
        Observe that
        $\tau(I_0x_v)=\tau_v(I_0)n_v$.
        By Lemma \ref{l:tau(P) image},
        $\tau_v(I_0)=I_0^{\on{inv}\circ\psi_v,\circ}$,
        which is closed in $I_0^{\on{inv}\circ\psi_v}$
        and further closed in $I_0$.
        Thus $\tau(I_0x_v)=\tau_v(I_0)n_v$ is closed in 
        $I_0n_v$, thus closed in $LG$.
        We obtain that $\tau(I_0x_v)$ is closed in $\tau(LX)$,
        so that $I_0x_v$ is a closed orbit in $LX$.

\end{proof}

\begin{rem}\label{closed orbits}
The equivalence of (i) and (ii) in the above 
is an affine generalization of 
\cite[Proposition 1.4.2]{RSbook}.
Moreover, when $G$ is simply connected (hence $\Omega=0$),
for a closed $\cO_v$ we have $\tau(x_v)\in I_0\subset L^+G$.
Thus $\cO_v\subset L^+X$ is the base change 
of a $B_0$-orbit $B_0\bar{x}_v$ on $X$
along the evaluation map $L^+X\rightarrow X$
such that $\tau(\bar{x}_v)\in N_G(T_0)\cap B_0=T_0$.
By \cite[Proposition 1.4.2]{RSbook}, 
$B_0\bar{x}_v$ is a closed orbit on $X$.
This gives a bijection between
closed $I_0$-orbits on $LX$ and
closed $B_0$-orbits on $X$.
\end{rem}

\begin{rem}\label{closed orbits remark}
As mentioned in \cite[\S6.6]{MS}, 
although part (i) and (iv) of the above lemma make sense 
for general spherical varieties,
they are not equivalent in general.
\end{rem}

\section{Placidness of orbit closures}\label{placid}
 In this section, 
we show that spherical and Iwahori obits on the loop space $LX$ are placid schemes. 
In particular, although those orbit closures are infinite dimensional,
their singularities are in fact finite dimensional.
We will use some basic definitions  and constructions of placid schemes 
collected 
in Appendix \ref{dimension theory}.

\subsection{Affine degenerations}\label{affine degeneration}
We recall the construction of affine degenerations of symmetric varieties
and its relationship with spherical orbits 
in $LX$. The main reference is 
\cite[Section 2.2]{SW}.\footnote{
Note that in \emph{loc. cit.}, the authors made some assumptions on the characteristic of $k$. Thanks to the recent work
\cite{BS}, those assumptions can be removed.}

Let $X^\circ=BK/K\subset X$ be the open $B$-orbit.
The action of $B/U$ on $
X^\circ//U=\on{Spec}(k[X^\circ]^U)$ factors through a quotient torus 
$T\ra T/T\cap K\stackrel{\tau}\cong S$.
Let $\check\Lambda_S\subset X^*(T)$ be the characters that
factor through $T\xrightarrow{\tau}S$.
Consider the filtration $k[X]_\chi,\chi\in\check\Lambda_S$ of $k[X]$ 
where $k[X]_\chi$ consists of $f\in k[X]$ such that the representation of $G$ generated by $f$
has highest weights $\chi'\in\check\Lambda_S$ satisfying $\langle\chi-\chi',\Lambda_S^+\rangle\leq 0$.\footnote{Since $k[X]^U\subset k[X^\circ]^U$, the action of $T$ on highest weight 
vectors factors through the quotient torus $S$
and hence the highest weights are contained in $\Lambda_S$}
The \emph{affine degeneration} $\mathcal X$ of $X$ is the affine scheme defined by
\[
k[\mathcal X ]=\bigoplus_{\chi\in\check\Lambda_S} k[X]_\chi\otimes e^\chi\subset k[X\times S],
\]
where $e^\chi\subset k[S]$ is the function 
corresponding to the character $\chi$. 
Consider the following affine toric scheme
 $A=\on{Spec}(\bigoplus_{\chi\in\check\Lambda_S,\langle\chi,\Lambda^+_S\rangle\leq 0} ke^\chi)$.
The natural inclusion $\bigoplus_{\chi\in\check\Lambda_S,\langle\chi,\Lambda^+_S\rangle\leq 0} ke^\chi\to k[\mathcal X]$ gives rise to 
$G\times S$-equivariant morphism 
\[\pi:\mathcal X\to A.\]
For any $a\in A(k)$, the fiber $\mathcal X_a=\pi^{-1}(a)$ is an affine spherical variety for 
$G$ (not necessary smooth and homogeneous). 
We let $\mathcal X^\bullet\subset\mathcal X$ be the open subscheme
such that the fiber $\mathcal X_a^\bullet\subset\mathcal X_a$ over $a$ is the open $G$-orbit.

Let $\lambda\in\Lambda_S^+$. 
Then the composition 
$\mathbb G_m\xrightarrow{\lambda^{-1}}S\rightarrow A$ 
extends to a map 
$\iota_\lambda:\mathbb A^1\to A$. 
We let 
$\mathcal X_\lambda=\mathcal X\times_{A,\iota_\lambda}\mathbb A^1$
and $\mathcal X^\bullet_{\lambda}=\mathcal X^\bullet\times_{A,\iota_\lambda}\mathbb A^1$.
Then $\pi_\lambda:\calX_\lambda\to\mathbb A^1$ is a
$G\times\mathbb G_m$-equivariant affine family and $\pi^\bullet_\lambda:\calX^\bullet_\lambda\to\mathbb A^1$ 
is an open subfamily.
The preimages of both $\pi_\lambda$ and $\pi^\bullet_\lambda$ over $\mathbb G_m$ are isomorphic to $X\times\mathbb G_m$.

The open inclusion $X\times\mathbb G_m\cong\pi_{\lambda}^{-1}(\mathbb G_m)\subset \calX_\lambda$ defines an inclusion
 $LX\times  L\mathbb G_m\subset  L\calX_\lambda$.
Consider the embedding 
\begin{equation}
\label{f}
f_\lambda: LX\to LX\times  L\mathbb G_m\subset  L\calX_\lambda, \ \ \quad\gamma\to (\gamma,t).
\end{equation}
Here $t\in k((t))=L\mathbb G_m(k)$.
We have an fp-open embedding $L^+\calX^\bullet_\lambda\subset L^+\calX_\lambda$
and a closed embedding $ L^+\calX_\lambda\to  L\calX_\lambda$.
We have the following description of $LX_\lambda$
in terms of the arc space $L^+\calX_\lambda^\bullet$:

\begin{lem}\label{fiber product}
The map $f_\lambda$ induces an isomorphism 
\[L X_\lambda\cong  LX\times_{f_\lambda, L\calX_\lambda} L^+\calX_\lambda^\bullet.\]
\end{lem}
\begin{proof}
According to \cite[Section 6.3]{MS}, the 
stabilizer of $B$ on $X^\circ$ is a smooth subgroup and the lemma follows from \cite[Lemma 2.3.10]{SW}.
\end{proof}

\subsection{Placidness of spherical and Iwahori orbit closures}

We have the following basic placidness results for loop spaces due to Drinfeld:

\begin{thm}\label{placid of loop spaces}
Let $Y$ be a smooth affine scheme of finite type
over $k$. The loop space 
$LY$ of $Y$ is a placid ind-scheme.
More precisely,
for any affine embedding $Y\subset\mathbb A^n$
let $L^iY=LY\cap t^{-i}L^+\mathbb A^n$.
Then each $LY^i$ is 
a placid scheme and 
$LY\cong\on{colim}_{i} LY^i$ is an ind-placid presentation of $LY$.

\end{thm}
\begin{proof}
    This is \cite[Theorem 6.3]{D}.
\end{proof}

We shall deduce the placidness of orbit closures 
from 
Lemma \ref{fiber product} and 
Theorem \ref{placid of loop spaces}.

\begin{thm}\label{Placidness of orbits}
\begin{itemize}
\item [(i)]
 The orbit closures 
 $\overline\cO_v$ (resp. 
 $\overline {LX}_\lambda$) 
 are irreducible placid schemes
 and the closed embeddings $\overline\cO_v\to LX$ (resp. $\overline {LX}_\lambda\to LX$)
 are finitely presented.
 \item [(ii)]
 The open embedding $\cO_v\to\overline\cO_v$ (resp. $LX_\lambda\to\overline{LX}_\lambda$) and the closed embedding
 $\overline\cO_v\to \overline\cO_{v'}$, $v\leq v'$
 (resp. $\overline{LX}_\lambda\to\overline{LX}_{\lambda'}$, $\lambda\leq\lambda'$)
 are finitely presented.
\end{itemize}

 \end{thm}
\begin{proof}
We first deal with $L^+G$-orbit closures.
Recall the $G\times\mathbb G_m$-equivariant affine family 
$\pi_\lambda:\calX_\lambda\to\mathbb A^1$ and 
the subschemes $X\cong\pi_{\lambda}^{-1}(1)\subset X\times\mathbb G_m\cong\pi_{\lambda}^{-1}(\mathbb G_m)\subset\calX^\bullet_\lambda\subset\calX_\lambda$ in Section \ref{affine degeneration}.
Since $\calX_\lambda$ is an affine $G\times\mathbb G_m$-scheme of finite type, we can find 
 a $G\times\mathbb G_m$-equivariant 
closed embedding $\calX_\lambda\subset V$ into 
a finite dimensional  representation $V$ of $G\times\mathbb G_m$.

We have a $LG\times L\mathbb G_m$-equivariant
closed embedding 
$ L\calX_\lambda\subset LV$ and 
for any $i\in\mathbb Z_{\leq 0}$  we can form the  $ L^+G\times L^+\mathbb G_m$-invariant closed subschemes  
$ L^i\calX_\lambda=  L\calX_\lambda\cap t^{i} L^+V\subset  L\calX_\lambda$ and 
$ L^+G$-invariant closed subscheme
$ L^iX= LX\cap t^{i} L^+V\subset LX$.
Since $X$ is smooth, Theorem \ref{placid of loop spaces} implies 
$ L^iX$ is a placid scheme
and $LX\cong\on{colim}_{i} L^iX$ is a placid presentation.

We claim that for any $i\in\bZ_{\leq0}$, there exists  $i'\in\mathbb Z_{\leq 0}$ such that the 
embedding 
$f_\lambda:  LX\to   LX\times  L\mathbb G_m\subset  L\calX_\lambda, f_{\lambda}(\gamma)=(\gamma,t)$ in~\eqref{f} satisfies 
\begin{equation}\label{image}
    f_\lambda( L^iX)\subset  L^{i'}\calX_\lambda.
\end{equation}
Indeed, let $V=\bigoplus_{j\in\bZ} V_j$ be the weight spaces decomposition with respect to the
$\mathbb G_m$-action, that is, $s\cdot v_j=s^jv_j$ for $v_j\in V_j$.
Then for any $w_j\in  LV_j$ we have 
$t\cdot w_j=t^j w_j$ and it follows that 
$f_\lambda( L^iX)\subset t\cdot ( LX\cap t^i L^+V)\subset t^{i'} L^+V$
for $i':=\on{min}\{i+j| j\in\bZ, V_j\neq 0\}$. The claim follows.

For any $\lambda\in\Lambda_S^+$,
choose $i,i'\in\bZ_{\leq 0}$ satisfying~\eqref{image}
and $LX_\lambda\subset  L^iX$.
Then Lemma \ref{fiber product} implies 
that there is a Cartesian diagram
\[
\xymatrix{LX_\lambda\cong L^iX\times_{L^{i'}\calX_\lambda} L^+\calX_\lambda^\bullet\ar[d]^{\eta_\lambda}\ar[r]& L^+\calX_\lambda^\bullet\ar[d]^{\eta}\\ L^iX\ar[r]^{f_\lambda}&L^{i'}\calX_\lambda}.
\]
Note that $\eta: L^+\calX_\lambda^\bullet\to  L^+\calX_\lambda\to  L^i\calX_\lambda$ is a fp-locally closed embedding
and hence $\eta_\lambda$ is also a fp-locally closed embedding.
Since $ L^iX$ is placid, the Noetherian descent for fp-morphisms 
implies that 
 there is a unique decomposition 
\[\eta_\lambda: LX_\lambda\stackrel{j_\lambda}\to\overline{ LX_\lambda}\stackrel{i_\lambda}\to  L^iX\]
where $j_\lambda$ is a placid fp-open embedding and $i_\lambda$ is a placid fp-closed emebdding (see, e.g., \cite[Lemma 1.3.6 and Lemma 1.3.11]{BKV}).
Now the placidness of $\overline{ LX_\lambda}$
follows from \ref{Placid schemes} (2).

For the case of Iwahori orbits, 
it suffices to show that 
the embedding $\cO_v\to L^+G x_v$ into the 
spherical orbit through $x_v\in\cO_v$ is finitely presented. 
Let $H\subset L^+G$ and $H'\subset I_0$ be the stabilizers of 
$x_v$ in $L^+G$ and $I_0$ respectively.
For any $j\in\mathbb Z_{\geq0}$, let  $H_j\subset L^+_jG$ and $H'_j\subset (I_0)_j$) be the images of $H$ and $H'$ 
along the map $L^+G\to L^+_jG$
and  $I_0\to (I_0)_j$. Here $L^+_jG$ is the $j$-th
arc group of $G$ and $(I_0)_j=\on{Im}(I_0\to L^+G\to L^+_jG)$.
We have a Cartesian diagram
\[\xymatrix{\cO_v\cong\lim_j(I_0)_j/H'_j\ar[r]\ar[d]&L^+Gx_v\cong\lim_j L^+_jG/H_j\ar[d]\\
(I_0)_j/H'_j\ar[r]^{f}&L^+_jG/H_j}\]
where the bottom horizontal arrow is induced by
the natural inclusion 
\[
f:(I_0)_j/H'_j\cong (I_0)_j/H_j\cap (I_0)_j\to L^+_jG/H^j.
\]
Since $f$ is finitely presented it follows that the 
top inclusion
$\cO_v\to L^+G x_v$ is finitely presented. 
The desired claim follows.
\end{proof}

\begin{rem}
Since \cite[Lemma 2.3.10]{SW}
and Theorem \ref{placid of loop spaces}
are valid for general 
spherical varieites, 
Theorem \ref{Placidness of orbits} remains true for general spherical varieties.
\end{rem}

\subsection{Filtered structures on orbits}
We shall show that, on each connected component of $LX$,
the orbits closure relation 
defines filtered partial orders on the sets of $I_0$-orbits and $L^+G$-orbits.

We have  isomorphisms of 
exact sequences 
\begin{equation}\label{component}
    \xymatrix{ \pi_0(LG)\ar[r]\ar[d]^\simeq&\pi_0(LX)\ar[d]^\simeq\ar[r]&0\ar[d]^\simeq\\
\pi_1(G)\ar[r]\ar[d]^\simeq&\pi_1(X)\ar[r]\ar[d]^\simeq&\pi_0(K)\ar[d]\\
\Lambda_T/Q\ar[r]^{\tau}&\Lambda_S/\Lambda_S\cap Q\ar[r]&0}
\end{equation}
induced from the
fibration $K\to G\to G/K=X$ (note that we assume $K$ is connected), here  $Q\subset \Lambda_T$ is the coroot lattice and the map $\tau$
is given by $\tau(\lambda)=-\theta(\lambda)+\lambda$ (see, e.g., \cite[Section 3.9]{NadlerRealGr}).

Let $LX=\bigsqcup_{\beta\in\pi_0(LX)} LX_\beta$ 
be the decomposition into connected components. 
We denote by $\sV_\beta=\{v\in\sV|\cO_v\subset LX_\beta\}$ and $\Lambda^+_{S,\beta}=\{\lambda\in\Lambda_S^+|LX_{\lambda }\subset LX_\beta\}$
and consider the  set 
$\underline\sV=\{\underline v=\{v_\beta\}_{\beta\in\pi_0(LX)}|v_\beta\in\sV_\beta\}$ and $\underline\Lambda^+_S=\{\underline\lambda=\{\lambda_\beta\}_{\beta\in\pi_0(LX)}|\lambda_{\beta}\in\Lambda_{S,\beta}^+\}$.

\begin{lem}\label{direct}
\begin{itemize}
\item [(i)]
    The set $\Lambda_{S,\beta}^+$
    together with the partial order 
$\lambda\leq\lambda'$
if $LX_\lambda\subset\overline{LX}_{\lambda'}$ is a filtered set.
    
\item [(ii)]    The set $\sV_\beta$ 
 together with the partial order $\bar v\leq\bar v'$
if $\cO_v\subset\overline\cO_{v'}$
is a filtered set
\end{itemize}
\end{lem}
\begin{proof}
Proof of (i). Let $\lambda,\lambda'\in\Lambda_{S,\beta}^+$.
According to~\eqref{component}, we have $\lambda-\lambda'\in\Lambda_S\cap Q$.
Let $M=Z_G(S)=P\cap\theta(P)$ be the Levi subgroup of a minimal 
$\theta$-split parabolic $P$. Let $\Delta$ (resp. $\Delta_M$) be the set of simple roots of $G$ (resp.
$M$) and let $\mu=2\check\rho-2\check\rho_M$
where $2\check\rho$ (resp. $2\check\rho_M$) is the sum of positive coroots of 
$G$ (resp. $M$). We have the following properties of $\mu$:

(1)
Note that $\mu=2\sum_{\alpha\in\Delta\setminus\Delta_M}\check\omega_{\alpha}$,
where $\check\omega_\alpha$ is the fundamental coweight corresponding to the simple root $\alpha$.

(2) By \cite[Lemma 8.1]{T}, we have $\theta(\mu)=-\mu$, and hence $\mu\in\Lambda_S\cap Q$.

(3) For any $\alpha\in\Delta$, it follows from (1) that 
$\langle\mu,\alpha\rangle=0$ if $\alpha\in\Delta_M$
and $\langle\mu,\alpha\rangle=2$ if $\alpha\in\Delta\setminus\Delta_M$.

(4) $\mu$ is a sum of simple coroots 
that has positive integer coefficient in each simple coroot.
This follows from the fact 
the coefficients of fundamental coweights $\check\omega_\alpha$
in simple roots are entries of the inverse of Cartan matrix, 
which are all positive (checked case-by-case). 

Using the 
above properties of $\mu$, it is straightforward to check that that for large enough integer $n$,
we have 
$\gamma=\lambda+n\mu\in\Lambda_S^+$, $\gamma\geq\lambda$,
and $\gamma\geq\lambda'$. Part (i) follows.

Proof of (ii).
Let $v,v'\in\sV_\beta$.
Let $LX_\lambda\supset\cO_v$ and $LX_{\lambda'}\supset\cO_{v'}$
be the $L^+G$-orbits through $x_v$ and $x_{v'}$
and let $\gamma\geq\lambda,\lambda'$ be as in (1).
By Corolalry \ref{c:Iwahori closure and G(O) into Iwahori}, there is a unique open $I_0$-orbit $\cO_{v''}\subset LX_\gamma$,
and it follows that $\overline{\cO}_{v''}=\overline{LX_\gamma}\supset LX_\lambda\cup LX_{\lambda'}\supset\cO_v\cup\cO_{v'}$.
Thus $v''\in\sV_\beta$ satisfies 
$v''\geq v$ and $v''\geq v'$. Part (ii) follows.

\end{proof}

We denote $\sV_\beta=\{v\in\cV|\cO_v\subset LX_\beta\}$ and $\Lambda^+_{S,\beta}=\{\lambda\in\Lambda_S^+|LX_{\lambda }\subset LX_\beta\}$
and consider the set 
$\underline\sV=\{\underline v=\{v_\beta\}_{\beta\in\pi_0(LX)}|v_\beta\in\sV_\beta\}$ and $\underline\Lambda^+_S=\{\underline\lambda=\{\lambda_\beta\}_{\beta\in\pi_0(LX)}|\lambda_{\beta}\in\Lambda_{S,\beta}^+\}$.
It follows from Lemma \ref
{direct} that the set $\underline\Lambda_S^+$
(resp. $\underline\sV$)
    together with the partial order 
$\underline\lambda\leq\underline\lambda'$
if $\lambda_\beta\leq\lambda_\beta'$ 
(resp. $\underline v\leq\underline v'$
if $v_\beta\leq v_\beta'$)
is a filtered set.

 The following proposition following immediately from Theorem \ref{Placidness of orbits}.

\begin{prop}\label{placid of LX}
\begin{itemize}
\item [(i)]
  
  The union $\overline\cO_{\underline v}=\bigsqcup_{\beta} \overline\cO_{\bar v_\beta}$
  is a locally equidimensional placid 
  scheme and 
  $LX\cong\on{colim}_{\underline v}\overline\cO_{\underline v}$
  is a locally equidimensional 
  $I_0$-placid presentation.
 
\item [(ii)] The disjoint union
  $\overline{LX}_{\underline \lambda}=\bigsqcup_{\beta} \overline{LX}_{\lambda_\beta}$
  is a locally equidimensional placid scheme and 
  $LX\cong \on{colim}_{\underline\lambda}\overline {LX}_{\underline\lambda}$
  is a locally equidimensional $L^+G$-placid presentation.
\end{itemize}
\end{prop}

\subsection{Dimension theory for $LX$}\label{Dimension theory}
Following \ref{dimension theory} (6),
a dimension theory for the locally equidimensional placid ind-scheme  $LX$
is an assignment  
to  each locally equidimensional placid subscheme 
$S\subset LX$
a locally constant function
$\delta_{S}:S\to\mathbb Z$
such that for any finitely presented locally closed 
embedding $S\subset S'\subset LX$ we have 
\begin{equation}\label{dim_S/S'}
\delta_{S}-\delta_{S'}|_{S}=\on{dim}_{S/S'}:S\to\mathbb Z,
\end{equation}
here
$\dim_{S/S'}:S\to\mathbb Z$ 
is the dimension function associated to 
the inclusion $S\subset S'$
in
Appendix \ref{dimension theory}(4).

\begin{lem}\label{l:closed I_0-orbits equal dim}
For any  closed $I_0$-orbits $\cO_{u},\cO_{u'}$
contained in $\overline\cO_v$, both 
$\dim_{\cO_u/\overline\cO_v}$ and $\dim_{\cO_{u'}/\overline\cO_v}$ are constant functions with the same value.
    
\end{lem}
\begin{proof}
By the additive property of  dimension functions in \ref{dimension theory}(6)
and the filtered property of $I_0$-orbits
in Proposition \ref{direct}, it suffices to verify the lemma for a particular $\overline\cO_v$ containing $\cO_u,\cO_{u'}$.

Claim: $N_{LG}(I_0)$ has transitive action on connected components.

Since $K$ is assumed to be connected, 
the natural map $LG\rightarrow LX$ induces a surjection $\pi_0(LG)\to\pi_0(LX)$
on connected components (see~\eqref{component}).
On the other hand, a well-know result of 
Kottwitz \cite{K} implies that 
the inclusion $ N_{LG}(I_0)\subset LG$
induces a bijection 
$\Omega\cong N_{LG}(I_0)/I_0\cong\pi_0(LG)$.
The claim follows.

Since the action of $\Omega$ maps $I_0$-orbit to $I_0$-orbit,
by the claim we can translate $\cO_v$ to the neutral component.
Then $\cO_u,\cO_{u'}$ are closed $I_0$-orbits in the neutral component,
so that $\tau(\cO_u)\subset I_0n_uI_0\subset N_{LG}(I_0)$
lands inside the neutral component of $LG$,
same for $u'$.
It follows that 
$\tau(\cO_u),\tau(\cO_{u'})\subset I_0$ 
and hence
$\cO_u,\cO_{u'}\subset L^+X$ (as $\tau^{-1}( L^+G)=L^{+}X$).
Since the closed $I_0$-orbits on $L^+X$ 
are base change from closed $B_0$-orbits on $X$
via the pro-smooth morphism $L^+X\rightarrow X$,
it suffices to show that closed orbits on $X$
are equal-dimensional.
As discussed in \cite[\S1.4]{RSbook},
closed $B_0$-orbits on $X$ all have the same dimension
as the flag variety of $K$.
This completes the proof.    
\end{proof}

\begin{prop}\label{dim for LX}
  There is a dimension theory $\delta$ on $LX$.
\end{prop}
\begin{proof}
By \ref{dimension theory} (9)
and 
Proposition \ref{direct}, it suffices to 
assign $\delta_{\overline\cO_v}:\overline\cO_v\to\mathbb Z$
for each $v\in\sV$ satisfying~\eqref{dim_S/S'}.
To this end, we set $\delta_{\overline\cO_v}=\dim_{\cO_u/\overline\cO_v}:\overline\cO_v\to\mathbb Z$
where $\cO_u\subset \overline\cO_v$
is a closed $I_0$-orbit in $\overline\cO_v$.
Lemma \ref{l:closed I_0-orbits equal dim} 
and additive property of dimension functions \ref{dimension theory}(6)
imply $\delta_{\overline\cO_v}$ is well-defined and satisfies \eqref{dim_S/S'}. The proposition follows.
\end{proof}


\section{Transversal slices}\label{Transversal slices}
We construct transversal slices of 
spherical and Iwahori obits on the loop space $LX$ 
and establish some basic properties of them.
The construction generalizes the one due to 
Mars-Springer \cite[\S6]{MS} in the case of $B$-orbits on $X$.

\subsection{Slices for spherical orbits}
\subsubsection{Involution on Affine Grassmannian slices}
Let $L^-G=G[t^{-1}]$, $L^{<0}G=\ker(L^-G\rightarrow G)$, and $L^-\fg,L^{<0}\fg$ their Lie algebras. 
For any $\lambda\in\Lambda_T^+$, 
we write $LG_\lambda=L^+Gt^\lambda L^+G$
for the $L^+G\times L^+G$-orbit of $t^\lambda$ in $LG$,
and $\overline {LG_\lambda}$ its closure.
The affine Grassmannian slice of $LG_\lambda$ in $LG$ is defined as 
$W^\lambda=(L^{<0}G\cap \Ad_{t^\lambda}L^{<0}G)t^\lambda$.
For any pair $\lambda,\mu\in\Lambda_T^+$, 
we define
$W^\lambda_\mu=W^\lambda\cap LG_\mu$ and $W^\lambda_{\leq\mu}=W^\lambda\cap \overline{LG_\mu}$.
We have the following well-known facts, see for example \cite{ZhuIntroGr}:
\begin{lem}\label{p:W^lambda_mu}
\begin{itemize}
\item [(i)] $W^\lambda$ is an ind-scheme of ind-finite type and formally smooth.
\item [(ii)] $W^\lambda_\mu$ is non-empty if and only if $\lambda\leq\mu$.
\item [(iii)]  $W^\lambda_\mu$ is smooth and connected (if non-empty).
\item [(iv)]  $W^\lambda_{\leq\mu}$ is of finite type, affine, and normal.
\item [(v)]  $\dim W^\lambda_\mu=W^\lambda_{\leq\mu}=\langle2\rho,\mu-\lambda\rangle$.
\end{itemize}
\end{lem}

Recall the involution 
$\on{inv}\circ\theta:LG\to LG, g\to \theta(g)^{-1}$.
\begin{lem}\label{inv on Aff slices}
Let $\lambda\leq\mu\in\Lambda_S^+$.
The 
map $\on{inv}\circ\theta:LG\to LG$
restricts to involutions on $W^\lambda$, $W^\lambda_\mu$, and $W^\lambda_{\leq\mu}$.
Moreover, the identification 
$W^\lambda\cong L^{<0}G\cap \Ad_{t^\lambda}L^{<0}G$, sending $\gamma\cdot t^\lambda\to \gamma $, intertwines the involution $\on{inv}\circ\theta|_{W^\lambda}$
with the involution $\on{inv}\circ\psi_\lambda$
on $L^{<0}G\cap \Ad_{t^\lambda}L^{<0}G$
where $\psi_\lambda=\on{Ad}_{t^\lambda}\circ\theta$.

\end{lem}
\begin{proof}
Note that 
$LG_\lambda$ is  stable under $\on{inv}\circ\theta$
and $L^{<0}G\cap \Ad_{t^\lambda}L^{<0}G$ is stable under $\on{inv}\circ\psi_\lambda$.
Thus 
for any $\gamma t^\lambda\in W^\lambda=(L^{<0}G\cap \Ad_{t^\lambda}L^{<0}G)t^\lambda$
(resp. $W^\lambda_\mu$, $W^\lambda_{\leq\mu}$), we have 
\[
\on{inv}\circ\theta(\gamma t^\lambda)=\theta(t^\lambda)^{-1}\theta(\gamma)^{-1}=(t^\lambda\theta(\gamma)^{-1}t^{-\lambda})t^\lambda=\on{inv}\circ\psi_\lambda(\gamma) t^\lambda\in
 W^\lambda\ \    (resp.\ \  W^\lambda_\mu, W^\lambda_{\leq\mu}).
\]
The lemma follows.

\end{proof}

For any $\lambda\in\Lambda_S^+$, consider the 
 maps $\rho_{\lambda}(s):LG\to LG$ sending 
$\gamma(t)\to \rho_\lambda(s)(\gamma(t))=s^\lambda\gamma(s^{-2}t)s^\lambda$.

\begin{lem}\label{G_m action on Aff slices}
The two maps $\rho_\lambda(s)$
and $\on{inv}\circ\theta$ commute with each other.
Moreover,
the  map $\rho_{\lambda}(s)$ restricts to a $\mathbb G_m$-action $\rho_s$ on $W^\lambda,W^{\lambda}_\mu, W^{\lambda}_{\leq\mu}$.
The $\mathbb G_m$-action $\rho_s$ on $W^\lambda$ and $W^\lambda_{\leq\mu}$
is contrating 
with unique fixed point $t^\lambda$.
\end{lem}
\begin{proof}
The first claim is a direct computation.
For any $\gamma(t) t^\lambda\in W^\lambda$,
we have 
\[
\rho_\lambda(s)(\gamma(t)t^\lambda)
=s^\lambda\gamma(s^{-2}t)(s^{-2}t)^\lambda s^\lambda
=s^\lambda\gamma(s^{-2}t)s^{-\lambda}t^\lambda\in W^\lambda.
\]
Thus $\rho_\lambda(s)$ preserves $W^\lambda$.
Since $\gamma\in L^{<0}G\cap\on{Ad}_{t^\lambda}L^{<0}G$,
we have 
$s^\lambda\gamma(s^{-2}t)s^{-\lambda}$ goes to the unit 
$e\in LG$ as $s$ goes to $0$, and it follows that the $\mathbb G_m$-action 
on $W^\lambda$
is contracting with unique fixed point $t^\lambda$.
Since $LG_\mu$ is obviously stable under 
$\rho_\lambda(s)$, the $\mathbb G_m$-action 
preserves $W^\lambda_\mu$ and $W^\lambda_{\leq\mu}$.
\end{proof}

Consider the fixed points 
\begin{equation}\label{L}
L^\lambda=(W^\lambda)^{\on{inv}\circ\theta},\ \ \ L^\lambda_\mu=(W_\mu^\lambda)^{\on{inv}\circ\theta},\ \ \ \ 
L^\lambda_{\leq\mu}=(W^\lambda_{\leq\mu})^{\on{inv}\circ\theta}.
\end{equation}\label{identification}
Lemma \ref{inv on Aff slices} (2) implies there is an isomorphism
\begin{equation}
L^\lambda\cong (L^{<0}G\cap \Ad_{t^\lambda}L^{<0}G)^{\on{inv}\circ\psi_\lambda}.
\end{equation}

\begin{lem}\label{p:L^lambda_mu}
\begin{itemize}
\item [(i)] $L^\lambda $ is a connected  ind-scheme of ind-finite type and formally smooth.
\item [(ii)]  $L^\lambda_\mu$ is smooth.
\item [(iii)]  $L^\lambda_{\leq\mu}$ is of finite type, affine, and connected.
\item [(iv)]
The  map $\rho_{\lambda}(s)$
restricts to a $\mathbb G_m$-action $\rho_s$ on $L^\lambda,L^{\lambda}_\mu, L^{\lambda}_{\leq\mu}$.
The $\mathbb G_m$-action $\rho_s$ on $L^\lambda$ and $L^\lambda_{\leq\mu}$
is contrating 
with unique fixed point $t^\lambda$.
\end{itemize}
\end{lem}
\begin{proof}
Part (i) and (ii)  follow from Lemma \ref{p:W^lambda_mu}
and Lemma \ref{fixed points}.
Since  $W^\lambda_{\leq\mu}$ is of finite type and affine, the fixed points 
$L^\lambda_{\leq\mu}$ is a closed subscheme and hence of finite type and affine.
Lemma \ref{G_m action on Aff slices} and \ref{inv on Aff slices} imply that 
those fixed points subschemes $L^\lambda,L^{\lambda}_\mu, L^{\lambda}_{\leq\mu}$ are stable under the 
$\mathbb G_m$-action $\rho_s$.
Moreover since $t^\lambda\in L^\lambda$, 
the $\mathbb G_m$-action on 
$L^\lambda$ and $L^\lambda_{\leq\mu}$ 
is contracting with unique fixed point $t^\lambda$.
In particular, we see that 
$L^\lambda_{\leq\mu}$ is  connected.
Part (iii) and (iv) follow.
\end{proof}

\subsubsection{The contracting slices $S^\lambda$
and $S^\lambda_{\leq\mu}$}
Recall the  embedding $\tau:LX \to LG$ induced from $\tau(\gamma)=\gamma\theta(\gamma)^{-1}:X\rightarrow G$.

 \begin{defe}
 (i) For any $\lambda\in\Lambda_S^+$, we define
$S^\lambda\subset LX$ as the following fiber product
\[\xymatrix{S^\lambda\ar[r]\ar[d]&LX\ar[d]^{\tau}\\
W^\lambda\ar[r]&LG.} \]
(ii) For any $\lambda,\mu\in\Lambda_S^+$,
let $S^\lambda_\mu=S^\lambda\cap LX_\mu$
and $S^\lambda_{\leq\mu}=S^\lambda\cap\bLX_\mu$.
\end{defe}

The following proposition implies that 
$S^\lambda$ defines a
transversal contracting slice 
to the orbit $LX_\lambda$ inside $LX$.

\begin{prop}\label{slice spherical}
	\begin{itemize}
	\item [(i)] The embedding $\tau$ induces an isomorphism $S^\lambda\cong L^\lambda$. In particular,
	$S^\lambda$ is a connected ind-scheme of ind-finite type and formally smooth.
				\item [(ii)] Let $r_s:t\mapsto s^{-2}t$ be rotation automorphism of $LG$, $s\in\bGm$. Then $\phi_s^\lambda=\phi_s:(x\mapsto s^\lambda r_s(x))$ defines a $\bGm$-action on $S^\lambda$, which fixes $x_\lambda$, and contracts $S^\lambda$ to $x_\lambda$.
		
			\item [(iii)] The multiplication map $m:L^+G\times S^\lambda\rightarrow LX$ is formally smooth.	
		\item [(iv)] 
		We have $S^\lambda\cap LX_\lambda=\{x_\lambda\}$
		and 
		$S^\lambda\subset LX$ is transversal to $LX_\lambda=L^+G\cdot x_\lambda$ at $x_\lambda$.

		\item [(v)] All the $\bGm$-actions $\phi_s^\lambda$ 
		preserve all the $L^+G$-orbits $LX_\mu$. 
		Any $L^+G$-equivariant local system on  $LX_\mu$ is $\bGm$-equivariant with respect to the action $\phi_s^\lambda$ for any $\mu\in\Lambda_S^+$. 
	\end{itemize}
\end{prop}

\begin{proof}
Part (i).
Note that the image $\tau(LX)\cong L(G^{\on{inv}\circ\theta,\circ})\subset L(G^{\on{inv}\circ\theta})=(LG)^{\on{inv}\circ\theta}$
is an union of components of the fixed point $(LG)^{\on{inv}\circ\theta}$.
Thus the image  $\tau(S^\lambda)=W^\lambda\cap\tau(LX)\subset L^\lambda=(W^\lambda)^{\on{inv}\circ\theta}$
is an union of component of $L^\lambda$.
Now the desired claim follows from Lemma \ref{p:L^lambda_mu} (i).

Part (ii). A direct computation shows that the isomorphism $\tau:S^\lambda\cong L^\lambda$
is compatible the  $\rho_s(\mathbb G_m)$-action.
The claims follow from Lemma \ref{p:L^lambda_mu} (iv).

Part (iii). 
It is equivalent to show that the twisted conjugation map
$f:L^+G\times L^\lambda\to LG^{\on{inv}\circ\theta}, (g,h t^\lambda),\to g h t^\lambda\theta(g)^{-1}$
is formally smooth.
Since $f$ is a map between  placid ind-schemes, we can apply the formally smooth criterion in Lemma \ref{formally smooth criterion}
and reduce to check that 
the differential
\[d_{(g,h t^\lambda)}f:
T_{g}L^+G\times T_{h t^\lambda}L^\lambda\to T_{a}{LX}\] is surjective
for any 
$R$-point
$(g,ht^\lambda)\in L^+G(R)\times L^\lambda(R)$ and
$a=ght^\lambda\theta(g)^{-1}\in (LG)^{\on{inv}\circ\theta}(R)$.
Denote $\psi_a=\Ad_a\circ\theta:LG(R)\to LG(R)$
and $L\fg^{-\psi_a}(R)$ the $-1$ eigensubspace of $\psi_a$ in $L\fg^{}(R)$.
Similarly define $\psi_{ht^\lambda}$.
Note that 
\[T_gL^+G\cong L^+\fg(R),\ \ 
T_{ht^\lambda}L^\lambda\simeq(L^{<0}\fg\cap\Ad_{t^\lambda}L^{<0}\fg)^{-\psi_{ht^\lambda}}(R),
\ \ 
T_a(LG)^{\on{inv}\circ\theta,\circ}\simeq(L\fg)^{-\psi_a}(R).
\]
We need to show that 
the tangent map 
\[
L^+\fg(R)\times(L^{<0}\fg\cap\Ad_{t^\lambda}L^{<0}\fg)^{-\psi_{ht^\lambda}}(R)
\rightarrow
(L\fg)^{-\psi_a}(R),
\quad
(X,Y)\mapsto X-\Ad_a\theta(X)+\Ad_g(Y)
\]
is surjective.
For simplicity,
in the proof below we will write $L^+\fg=L^+\fg(R)$, etc.

Applying $\Ad_{g^{-1}}$,
we equivalently need to show that map
\[
L^+\fg\times(L^{<0}\fg\cap\Ad_{t^\lambda}L^{<0}\fg)^{-\psi_{ht^\lambda}}
\rightarrow
(L\fg)^{-\psi_{ht^\lambda}},
\quad
(X,Y)\mapsto X-\Ad_{ht^\lambda}\theta(X)+Y
\]
is surjective, i.e.
\begin{equation}\label{eq:tau_htlambda}
\tau_{ht^\lambda}:
L^+\fg\rightarrow(L\fg)^{-\psi_{ht^\lambda}}/(L^{<0}\fg
\cap\Ad_{t^\lambda}L^{<0}\fg)^{-\psi_{ht^\lambda}},
\quad
X\mapsto X-\Ad_{ht^\lambda}\theta(X)
\end{equation}
is surjective.
Consider subspaces
\begin{align*}
&A_1=L^+\fg\cap\Ad_{ht^\lambda}L^+\fg,\\
&A_2=L^{<0}\fg\cap\Ad_{ht^\lambda}L^{<0}\fg,\\
&A_3=L^+\fg\cap\Ad_{ht^\lambda}L^{<0}\fg\oplus
L^{<0}\fg\cap\Ad_{ht^\lambda}L^+\fg,\\
&L\fg=L^+\fg\oplus L^{<0}\fg=A_1\oplus A_2\oplus A_3.
\end{align*}

Each of $A_i$ is stable under $\psi_{ht^\lambda}$.
Thus we have decompositions
\[
A_i=A_i^{\psi_{ht^\lambda}}\oplus A_i^{-\psi_{ht^\lambda}},
\]
which induces
\[
L\fg^{-\psi_{ht^\lambda}}=A_1^{-\psi_{ht^\lambda}}\oplus A_2^{-\psi_{ht^\lambda}}\oplus A_3^{-\psi_{ht^\lambda}}.
\]

The map \eqref{eq:tau_htlambda} is equivalent to map
\begin{equation}
    \tau': L^+\fg\mapsto A_1^{-\psi_{ht^\lambda}}\oplus A_3^{-\psi_{ht^\lambda}},\quad
    X\mapsto X-\psi_{ht^\lambda}(X).
\end{equation}

Note that for any $X\in A_1^{-\psi_{ht^\lambda}}\subset L^+\fg$,
$\tau'(\frac{1}{2}X)=X$.
Thus $\tau'$ restricts to an isomorphism on $A_1^{-\psi_{ht^\lambda}}$.
We are left to show that
\[
\tau'':L^+\fg\cap\Ad_{ht^\lambda}L^{<0}\fg\rightarrow
A_3^{-\psi_{ht^\lambda}},\quad
X\mapsto X-\psi_{ht^\lambda}(X)
\]
is surjective.

Denote $A=L^+\fg\cap\Ad_{ht^\lambda}L^{<0}\fg\subset L^+\fg$
and $\psi:=\psi_{ht^\lambda}$.
Note that $\psi(A)=L^{<0}\fg\cap\psi(L^+\fg)\subset L^{<0}\fg$,
so that
$A_3=A\oplus\psi(A)$.
For any $a+\psi(b)\in A_3$, $a,b\in A$,
\[
a+\psi(b)\in A_3^{-\psi}\Leftrightarrow
\psi(a+\psi(b))=\psi(a)+b=-a-\psi(b)\Leftrightarrow
b=-a.
\]
Thus $a+\psi(b)=a-\psi(a)=\tau''(a)$, proving the surjectivity.
\\

Part (iv).
$S^\lambda\cap LX_\lambda=\{x_\lambda\}$ follows from
$W^\lambda\cap LG_\lambda=\{t^\lambda\}$.
 To show    $L^\lambda$ is transversal to $LX_\lambda$ at $x_\lambda$, we need to check
    \[
    T_{t^\lambda}\tau(LX_\lambda)\oplus T_{t^\lambda}L_\lambda=T_{t^\lambda}(LG)^{-\psi_\lambda}.
    \]
    Explicitly, 
    \[
    T_{t^\lambda}\tau(LX_\lambda)=\{X-\psi_\lambda(X)|X\in L^+\fg\},\quad
    T_{t^\lambda}L^\lambda=(L^{<0}\fg\cap\Ad_{t^\lambda}L^{<0}\fg)^{-\psi_\lambda}.
    \]
    It follows from part (iii) that the LHS spans RHS.
    It only remains to show the LHS is a direct sum,
    which follows from the decomposition
    \[
    L\fg=L^+\fg\oplus(\Ad_{t^\lambda}L^+\fg\cap L^{<0}\fg)\oplus
    (L^{<0}\fg\cap\Ad_{t^\lambda}L^{<0}\fg).
    \]
\\

Part (v). For a $L^+G$-orbit $LX_\mu=L^+Gx_\mu$ 
	and action $\phi_s^\lambda$, for $g\in G(\cO)$, we have
	\[
	\phi_s^\lambda(gx_\mu)
	=s^\lambda r_s(g)r_s(x_\mu)
	=(s^\lambda r_s(g)s^{-\mu})\phi_s^\mu(x_\mu)
	=(s^\lambda r_s(g)s^{-\mu})x_\mu.
	\]
	As $s^\lambda r_s(g)s^{-\mu}\in G(\cO)$, 
	$\phi_s^\lambda$ preserves $LX_\mu$.
	
	Now we show any local system $\cL$ on $LX_\mu$ 
	is $\phi_s^\lambda$-equivariant. 
	Consider $\bGm$-action on $L^+G$:
	$\gamma_s(g)=\Ad_{s^\lambda}r_s(g)$.
	Then the action of $L^+G$ and $\bGm$ on $LX_\mu$ at $x_\mu$
	can be integrated into an action of the semidirect product
	$L^+G\rtimes\bGm$ with respect to $\gamma_s$.
	Consider exact sequence
	\[
	1\rightarrow
	\mathrm{Stab}_{L^+G}(x_\mu)\rightarrow \mathrm{Stab}_{L^+G\rtimes\bGm}(x_\mu)\rightarrow
	\bGm,\quad
	g\mapsto (g,1),\quad (g,s)\mapsto s.
	\]
	Consider section $p:\bGm\rightarrow\mathrm{Stab}_{L^+G\rtimes\bGm}(x_\mu), 
	p(s)=(s^{\mu-\lambda},s)$.
	This is clearly a group homomorphism. 
	Thus the above exact sequence is a split short exact sequence.
	We get semidirect product
	\[
	\mathrm{Stab}_{L^+G\rtimes\bGm}(x_\mu)\simeq
	\mathrm{Stab}_{L^+G}(x_\mu)\rtimes\bGm.
	\]
	Thus the component groups are 
	\[\pi_0(\mathrm{Stab}_{L^+G\rtimes\bGm}(x_\mu))\simeq
	\pi_0(\mathrm{Stab}_{L^+G}(x_\mu)\rtimes\bGm)\simeq
	\pi_0(\mathrm{Stab}_{L^+G}(x_\mu)).
	\]
	Therefore $L^+G$-equivariant local systems on $LX_\mu$ are $\phi_s^\lambda$-equivariant for any $\mu$.
\end{proof}

We also give an alternative proof of the above 
Proposition \ref{slice spherical}.(iii).
\begin{lem}\label{l:2nd pf spherical submersion}
    The multiplication map $m:L^+G\times S^\lambda\rightarrow LX$ is formally smooth.
\end{lem}
\begin{proof}
    As before, 
    it sufficies to check the surjectivity of the differential map.
    From standard results on affine Grassmanian slices, we know that
    the multiplication map
    \[
    m: L^+G\times W^\lambda\times L^+G\rightarrow LG,\quad
    (g_1,x,g_2)\mapsto g_1xg_2
    \]
    is a submersion.
    
    Assume $\lambda\in\Lambda_S^+$.
    Denote $\sigma=\on{inv}\circ\theta$, 
    which is an involution of space on $LG$.
    Also denote by $\gamma(g_1,x,g_2)=(\sigma(g_2),\sigma(x),\sigma(g_1)$,
    which is an involution on $L^+G\times W^\lambda\times L^+G$.
    Its fixed subspace is 
    \[
    \{(g_1,x,\sigma(g_1))\mid g_1\in L^+G,x\in(W^\lambda)^\sigma\}
    \simeq L^+G\times S^\lambda.
    \]
    
    Observe that $m$ intertwines $\gamma$ and $\sigma$.
    The induced map on fixed subspaces is isomorphic to 
    $m:L^+G\times S^\lambda\rightarrow LX$.
    Now for any $a=(g,x)\in L^+G\times S^\lambda\hookrightarrow L^+G\times W^\lambda\times L^+G$, the surjective differential map
    \[
    \td m: T_a(L^+G\times W^\lambda\times L^+G)\rightarrow T_{m(a)}LG
    \]
    intertwines $\gamma$ and $\sigma$.
    Since both $\gamma$ and $\sigma$ are involutions, 
    the induced map on fixed subspaces,
    which is $\td m:T_a(L^+G\times S^\lambda)\rightarrow T_{m(a)}LX$,
    is also surjective as desired.
\end{proof}

\begin{cor}\label{c:LX equi-singular}
	Let $\lambda,\mu\in\Lambda_S^+$.
	\begin{itemize}

		\item [(i)] The multiplication map $m:L^+G\times S^\lambda_{\leq\mu}\rightarrow\bLX_\mu$ is 
		formally smooth.
		\item [(ii)] There is a $\bGm$-action on $S^\lambda_{\leq\mu}$, which fixes $x_\lambda$, and contracts $S^\lambda_{\leq\mu}$ to $x_\lambda$.

    \item [(iii)]
$S^\lambda_\mu\neq\emptyset$ if and only if $\lambda\leq\mu$.
		
	\end{itemize}
\end{cor}
\begin{proof}
The map in (i) is the base-change of the formally smooth map
$L^+G\times S^\lambda\to LX$ in Proposition \ref{slice spherical} (iii) along the embedding $\overline{LX}_{\mu}$ and hence is formally smooth.
Part (ii) follows from Proposition \ref{slice spherical} (ii).

Part (iii).
Assume $S^\lambda_\mu=S^\lambda\cap LX_\mu\neq\varnothing$. 
	Recall that we have $\bGm$-action $\phi_s$(depended on $\lambda$) 
	on $S^\lambda$. 
	This action also acts on all the $L^+G$-orbits $LX_\mu$. 
	The contraction of $S^\lambda\cap LX_\mu$, 
	which gives $x_\lambda$, 
	must be contained in $\bLX_\mu$. 
	Thus $\lambda\leq\mu$.
	
	Conversely, by Theorem \ref{t:sym var orbit closure} 
	we have $x_\lambda\in S^\lambda_{\leq\mu}=S^\lambda\cap\bLX_\mu\neq\emptyset$.
    Since  $m:L^+G\times (S^\lambda_{\leq\mu}\cap\bLX_\mu)\to \overline{LX}_\mu$ is 
    formally smooth and $L^+G\times S^\lambda_{\leq\mu}$ admits a placid 
    presentation $L^+G\times S^\lambda_{\leq\mu}\cong\lim_i ((L^+G)_i\times S^\lambda_{\leq\mu})$
	with surjctive transition maps, 
    Lemma \ref{formally smooth} and 
    Lemma \ref{quasi-flatness}
    imply that the 
    image of $m$
    has nonempty intersection with the open dense subset $LX_\mu$. 
    It follows that 
    $m^{-1}(LX_\mu)=L^+G\times S^\lambda_\mu\neq\emptyset$.
    Part (iii) follows.
\end{proof}

The corollary  implies that  $S^\lambda_{\leq\mu}$
defines a contracting slice for $LX_{\lambda}$ inside $\overline{LX}_{\mu}$, see Definition \ref{def of slice}).

\begin{rem}
Unlike the case of affine Grassmannian slices $W^\lambda_{\leq\mu}$
and $W^\lambda_\mu$,
the slices $S^\lambda_{\leq\mu}$ 
might be reducible and non-normal, 
and $S^\lambda_\mu$ 
might be disconnected 
(see Example \ref{non-normal}).

\end{rem}

\subsection{Lusztig embedding}\label{Lusztig}

Let $G=\GL_n$.
Consider the classical groups 
$K=\mathrm O_n$ and $\Sp_{n}$(when $n$ is even) of orthogonal group and symplectic group 
corresponding to the involution 
$\theta(g)=(g^t)^{-1}$ and $\theta(g)=J^{-1}(g^t)^{-1}J$ respectively.
Here
$J=\begin{pmatrix}0&I_m\\
-I_m&0\end{pmatrix}$ 
with $n=2m$.
We have 
$\fg=\mathfrak{gl}_{n}$, 
$d\theta(g)=-g^t$ (resp. $d\theta(g)=-J^{-1}g^tJ$), 
$\fk=\fg^{d\theta}=\mathfrak{o}_n$ (resp. $\fk=\fsp_{n}$),
$\fp=\fg^{-d\theta}=\{g\in\mathfrak{gl}_n\ |\ g^t=g\}$ 
(resp. $\fp=\{g\in\mathfrak{gl}_n\ |\ (gJ)^t=-gJ\}$).

Let $\sN\subset\mathfrak g$ be the nilpotent cone
and $\sN_{\fp}=\sN\cap\mathfrak p$ be the corresponding symmetric nilpotent cone.
The $G$-orbits on $\sN$, called the nilpotent orbits, are parametrized by partitions 
$\cP_n=\{(\lambda_1\geq \lambda_2\geq\cdot\cdot\cdot\geq \lambda_n\geq0|\sum \lambda_i=n\}$.
Let $K^\circ\subset K$ be the identity component of $K$.
The $K^\circ$-orbits on $\sN_\fp$, called the symmetric nilpotent orbits, are parametrized by 
$\cP_n$ in the case of $K^\circ=\SO_n$ and 
$\cP_{m}$ in the case of $K^\circ=K=\Sp_{2m}$. 
We will write $\cO_\lambda$ and $\cO_{\fp,\lambda}$  for the corresponding nilpotent orbits
and symmetric nilpotent orbits
and $\overline\cO_\lambda$ and 
$\overline\cO_{\fp,\lambda}$ the orbit closures.
 
We will identify the set $\Lambda_T^+$ with 
$\Lambda_T^+=\{(\lambda_1\geq \lambda_2\geq\cdot\cdot\cdot\geq \lambda_n)\}$
so that 
\[LG_\lambda=L^+G\on{diag}(t^{\lambda_1-1},...,t^{\lambda_n-1})L^+G\]
for $\lambda\in\Lambda_T^+$.
Recall the Lusztig embedding 
\begin{equation}\label{eq:GLn Lusztig embed iota_n}
	\iota:\sN\hookrightarrow W^0=L^{<0}G,\quad x\mapsto(I_n-t^{-1}x).
\end{equation}
It induces an isomorphism 
\[\sN\cong W^0_{\leq\omega_n}=L^{<0}G\cap \overline{LG}_{\omega_n}\]
where $\omega_n=(n\geq0\geq\cdots\geq0)$. In addition, for $\lambda\in\cP_n$, the above isomorphism 
restricts to isomorphisms
\[\cO_\lambda\cong W^0_\lambda,\ \ \ \ \ \ \ 
\overline{\cO}_\lambda\cong W^0_{\leq\lambda}.\]
 
Note that the Lusztig embedding intertwines the 
involution 
$-d\theta$ on $\sN$ with $\on{inv}\circ\theta$ on $W^0$
and hence induces isomorphisms
\[\sN_\fp\cong (W^0_{\leq\omega_n})^{\on{inv}\circ\theta}\cong S^0_{\leq\omega_n},\]
\[\cO_{\fp,\lambda}\cong (W^0_{\lambda})^{\on{inv}\circ\theta}\cong S^0_{\lambda},\ \ \ \ \
\overline\cO_{\fp,\lambda}\cong (W^0_{\leq\lambda})^{\on{inv}\circ\theta}\cong S^0_{\leq\lambda}\]
on the fixed points.
In view of Proposition \ref{slice spherical}, we obtain
\begin{prop}\label{slice=nilp}\mbox{}
There is a $\mathbb G_m$-equivariant isomorphism
		$\overline\cO_{\fp,\lambda}\cong S^0_{\leq\lambda}$
		between symmetric nilpotent orbit closures and transversal slices 
		in loop spaces of classical symmetric varieties. For any $\lambda\leq\mu\leq\omega_n\in\cP_n$, the 
		variety $S^\lambda_{\leq\mu}$
		defines a Slodowy slice to the 
		symmetric nilpotent orbit 
		$\cO_{\fp,\lambda}$ inside $\overline\cO_{\fp,\mu}$.
	
\end{prop}

\begin{exam}\label{non-normal}
For $X=\GL_2/\mathrm O_2$, the slice $S^0_{\leq\omega_2}$
is isomorphic to 
$S^0_{\leq\omega_2}\cong\sN_\fp\cong\{xy=0\}\subset k^2$, which is reducible and non-normal.

\end{exam}

\subsection{Slices for Iwahori orbits}

\subsubsection{Involutions on affine flag slices}
Let $T_0\subset B_0\subset G$ and $I_0=B_0(L^{++}G)$ be
$\theta$-stable 
 maximal torus, Borel subgroup, and Iwahori of $LG$ as before.
Denote the opposite Iwahori by $I_0^-:=B_0^-L^{<0}G$, 
where $B_0^-$ is the Borel opposite to $B_0$.
Let $I_0^{--}=U_0^-L^{<0}G$, $U_0^-$ is the unipotent radical of $B_0^-$.
For any  $w\in\widetilde{\mathrm W}=N_{LG}(LT_0)/L^+T_0$ 
 we denote by $\Fl_w\subset\Fl=LG/I_0$ the corresponding $I_0$-orbits on the affine flag variety 
 and $\overline{\Fl}_w$ its closure.
 We denote by 
$LG_w,\overline{LG}_w\subset LG$ the pre-images of $\Fl,\overline{\Fl}_w$
 in $LG$.

Let $n\in N_{LG}(LT_0)$ be a lifting of $w$.
The affine flag slice of $LG_w$ in $LG$ at $n$ is defined as 
\[
W^{n}=(I_0^{--}\cap\on{Ad}_{n} I_0^{--})n\subset LG.
\]
For any pair $n,n'\in N_{LG}(LT_0)$ with images $w,w'\in\widetilde W$, 
we define
$W^{n}_{n'}=W^n\cap LG_{w'}$ and $W^n_{\leq n'}=W^n\cap \overline{LG}_{w'}$.
We have the following well-known facts:
\begin{lem}\label{p:W^lambda_mu}
\begin{itemize}
\item [(i)] $W^n$ is an ind-scheme of ind-finite type and formally smooth.
\item [(ii)] $W^n_{n'}$ is non-empty if and only if $w\leq w'$.
\item [(iii)]  $W^n_{n'}$ is smooth and connected (if non-empty).
\item [(iv)]  $W^n_{\leq n'}$ is of finite type, affine, and normal.
\item [(v)]  $\dim W^n_{n'}=\dim W^n_{\leq n'}=l(w')-l(w)$.
\end{itemize}
\end{lem}
\begin{proof}
  This follows from 
  \cite[Theorem 7 and 8]{F} and 
  \cite[Proposition 1.3.1 and 1.3.2]{KT}.
\end{proof}

Since $\theta(T_0)=T_0$
 the normalizer $N_{LG}(LT_0)$
is stable under the involution $\on{inv}\circ\theta$ of $LG$.

\begin{lem}\label{inv of W^n}
Let
$n,n'\in (N_{LG}(LT_0))^{\on{inv}\circ\theta}$.
Then $W^n$, $W^n_{n'}$, and $W^n_{\leq n'}$
are stable under 
the involution $\on{inv}\circ\theta:LG\to LG$.
Moreover, the identification 
$W^n\cong I_0^{--}\cap \Ad_{n} I_0^{--}$, sending $\gamma\cdot n\to \gamma $, intertwines the involution $\on{inv}\circ\theta|_{W^\lambda}$
with the involution $\on{inv}\circ\psi_n$
on $I_0^{--}\cap \Ad_{n} I_0^{--}$,
where 
$\psi_n=\on{Ad}_n\circ\theta:LG\to LG$.

\end{lem}
\begin{proof}
A direct computation shows that
 that $LG_{w'}$ and $\overline{LG}_{w'}$
 are stable under $\on{inv}\circ\theta$ and 
 $I_0^{--}\cap \Ad_{n} I_0^{--}$ is stable under 
 $\on{inv}\circ\psi_n$. Thus
for any $\gamma n\in W^n=(I_0^{--}\cap \Ad_{n} I_0^{--})n$
(resp. $W^n_{n'}$, $W^n_{\leq n'}$), we have 
$\on{inv}\circ\theta(\gamma n)=n\theta(\gamma)^{-1}=\on{inv}\circ\psi_n(\gamma) n\in
 W^\lambda$ (resp.\ \  $W^\lambda_\mu,\  W^\lambda_{\leq\mu}$).
 The lemme follows.

\end{proof}

Assume further that $n\in N_{LG}(LT_0)^{\on{inv}\circ\theta}$
has the form 
$n=\bar nt^\lambda$
where $\bar{n}\in N_{G}(T_0)$ and $\lambda\in X_*(T_0)$
	satisfy $\theta(\bar{n})=\bar{n}^{-1},\
	\Ad_{\bar{n}}t^{-\lambda}=t^{\theta(\lambda)}$.
 Let $\psi_{\bar n}=\on{Ad}_{\bar n}\circ\theta:T_0\to T_0$.
Choose a cocharacter $\mu\in X_*(T_0^{\psi_{\bar{n}}})$ 
as in \cite[\S6.4]{MS}, so that
$\langle\mu,\alpha\rangle<0$ 
for all positive roots $\alpha>0$ with $\psi_{\bar{n}}\alpha>0$.
Let $m=\max\{1,\langle\mu,\alpha\rangle+1,\forall\alpha\in\Phi\}$ 
where $\Phi$ is the set of roots of $G$. 
Consider the following map on $LG$:
\begin{equation}
	\phi_s:\gamma(t)\mapsto s^{\mu-m\theta(\lambda)}\gamma(s^{-2m}t)s^{-\theta(\mu)+m\lambda}. 
\end{equation}

\begin{lem}\label{G_m action on AF slices}
Let $n,n'\in N_{LG}(LT_0)^{\on{inv}\circ\theta}$
and assume $n=\bar n t^\lambda$ as above.
\begin{itemize}
\item [(i)]
The two maps 
$\phi_s$ and $\on{inv}\circ\theta$ commute with each other.
\item [(ii)]
$\phi_s$
restricts to a $\mathbb G_m$-action $\rho_s$ on $W^n,W^{n}_{n'}, W^{n}_{\leq n'}$.
The $\mathbb G_m$-action $\rho_s$ on $W^n$ and $W^n_{\leq n'}$
is contrating 
with unique fixed point $n$.
\end{itemize}
\end{lem}
\begin{proof}
Part (i) is a direct computation.
Proof of (ii).
First we check that
\[
\phi_s(n)
=s^{\mu-m\theta(\lambda)}\bar{n}t^\lambda s^{-2m\lambda}s^{-\theta(\mu)+m\lambda}
=\bar{n}t^\lambda s^{\Ad_{\bar{n}^{-1}}\mu-m\Ad_{\bar{n}^{-1}}\theta(\lambda)-2m\lambda-\theta(\mu)+m\lambda}
=n.
\]
Here we used that on $X_*(T_0^{\psi_{\bar{n}}})\subset X_*(T_0)$,
$\Ad_{\bar{n}^{-1}}$ acts as $\theta$.
Next we  show that $\phi_s$ acts on $W^n$ and contracts it to $n$.
Write $\gamma n\in W^n$.
Then
\[
\phi_s(\gamma n)=(\Ad_{s^{\mu-m\theta(\lambda)}}\gamma(s^{-2m}t))n.
\]
It suffices to show that the action 
$\Ad_{s^{\mu-m\theta(\lambda)}}\circ r_{s^{m}}$
acts on $I_0^{--}\cap\psi_n I_0^{--}$ and contracts it to $1$,
where $r_s:\gamma(t)\mapsto\gamma(s^{-2}t)$ is the rotation on $LG$.
Since $I_0^{--}\cap\psi_v I_0^{--}$ is stable under the adjoint action of $T_0$
     and $\phi_s$ contracts elements in $L^{--}T_0$ to $1$,
     it suffices to show the weight of $s$ 
     in each affine root subgroup contained in $I_0^{--}\cap\psi_n I_0^{--}$ is positive.
     To this end, we compute the weights of 
     $\Ad_{s^{\mu}}$ and $\Ad_{s^{-\theta(\lambda)}}\circ r_{s^2}$ separately.
     
     Claim: The weight of $s$ with respect to the action 
     $\Ad_{s^{-\theta(\lambda)}}\circ r_{s}$
     on any affine root subgroup 
     $U_{\tilde{\alpha}}\subset I_0^{--}\cap\psi_n I_0^{--}$
     is positive, unless when $\tilde{\alpha}=\alpha\in\Phi$ 
     is a negative finite root such that $\psi_{\bar{v}}(\alpha)<0$, 
     where the weight is still non-negative.
     
     The positivity of weights of $\Ad_{s^{\mu-m\theta(\lambda)}}\circ r_{s^{m}}$
     then follows immediately from the claim and the definitions of $m$ and $\mu$.
     
     To see the claim, let $\tilde{\alpha}=\alpha+m_\alpha$ be an affine root, 
     where $\alpha\in\Phi$ is the finite part, $m_\alpha$ is the imaginary part.
     The weight of $\Ad_{s^{-\theta(\lambda)}}\circ r_{s}$ on $U_{\tilde{\alpha}}$ is
     \[
     \langle-\theta(\lambda),\alpha\rangle-2m_\alpha
     =-\langle\lambda,\theta(\alpha)\rangle-2m_\alpha.
     \]
     
     When
     $U_{\tilde{\alpha}}\subset I_0^{--}\cap\psi_n I_0^{--}$, 
     $\tilde{\alpha}$ satisfies
     $\tilde{\alpha}=\alpha+m_\alpha\in\Phi_{\aff}(I_0^{--})$
     and 
     \[
     \psi_n(\tilde{\alpha})
     =\psi_{\bar{n}}(\alpha)+(m_\alpha+\langle\lambda,\theta(\alpha)\rangle)\in\Phi_{\aff}(I_0^{--}),
     \]
     from which we easily deduce that the weight
     $-\langle\lambda,\theta(\alpha)\rangle-2m_\alpha$
     is as in the claim.
     
     We can conclude the proof by noticing that the orbit 
     $LG_{w'}$ is stable under $\phi_s$. 
\end{proof}

We preserve the set ups in Lemma \ref{G_m action on AF slices}.
Consider the fixed points 
\begin{equation}\label{L}
L^n=(W^n)^{\on{inv}\circ\theta},\ \ \ L^n_{n'}=(W_{n'}^n)^{\on{inv}\circ\theta},\ \ \ \ 
L^n_{\leq n'}=(W^n_{\leq n'})^{\on{inv}\circ\theta}.
\end{equation}\label{identification}
Lemma \ref{inv on Aff slices} (2) implies there is an isomorphism
\begin{equation}
L^n\cong (I_0^{--}\cap \Ad_{n}I_0^{--})^{\on{inv}\circ\psi_n}.
\end{equation}

Using Lemma \ref{G_m action on AF slices}, 
the same proof as in the spherical case (see Lemma \ref{p:L^lambda_mu}) shows:
\begin{lem}\label{p:L^n_n'}
\begin{itemize}
\item [(i)] $L^n $ is a connected  ind-scheme of ind-finite type and formally smooth.
\item [(ii)]  $L^{n}_{n'}$ is smooth.
\item [(iii)]  $L^n_{\leq n'}$ is of finite type, affine, and connected.
\item [(iv)]
The  $\mathbb G_m$-action $\rho_s$ on $W^n$ restricts to a
$\mathbb G_m$-action 
on $L^\lambda,L^{n}_\mu, L^{n}_{\leq n'}$.
The $\mathbb G_m$-action $\rho_s$ on $L^n$ and $L^n_{\leq n'}$
is contrating 
with unique fixed point $n$.
\end{itemize}
\end{lem}

\subsubsection{The contracting slices $S^v$ and $S^v_{\leq u}$}
\label{construction of S^v}

Denote by $\cO_v$ the $I_0$-orbit on $LX$ 
parametrized by $v\in \sV=L^+T_0\backslash\cV/LK$ 
in Lemma \ref{l:Iwahori orbirs on LX}.
According to Lemma \ref{l:Iwahori orbit rep},
we can find a lifting $\dv\in\cV$  of $v$ 
such that   $x_v=\dv LK\in\cO_v$
and 
$n_v=\tau(x_v)=\tau(\dv)=\bar{n}_vt^\lambda$
where $\bar{n}_v\in N_{G}(T_0)$ and $\lambda\in X_*(T_0)$
	satisfy $\theta(\bar{n}_v)=\bar{n}_v^{-1},\
	\Ad_{\bar{n}_v}t^{-\lambda}=t^{\theta(\lambda)}$.

Thus we are in the setting of Lemma \ref{inv of W^n} and Lemma \ref{G_m action on AF slices}  and we denote by
$\psi_v=\Ad_{n_v}\circ\theta$ and
$\psi_{\bar{v}}=\Ad_{\bar{n}_v}\circ\theta$ the involutions in \emph{loc. cit.}.

\begin{defe}
\begin{itemize}
   \item [(i)]
   For any $v\in \sV=L^+T_0\backslash\cV/LK$ and a lifting $\dv\in\cV$ as above, we define
$S^v\subset LX$ as the following fiber product
\[\xymatrix{S^v\ar[r]\ar[d]&LX\ar[d]^{\tau}\\
W^{v}\ar[r]&LG.} \]
Here $W^v=W^{n_v}=(I_0^{--}\cap\on{Ad}_{n_v}I_0^{--})n_v$ is the slice in Lemma \ref{p:W^lambda_mu}.
   \item [(ii)]
   For any $v,u\in L^+T_0\backslash\cV/LK$,
let $S^v_u=S^v\cap\cO_u$
and $S^v_{\leq u}=S^v\cap\overline\cO_u$.
\end{itemize}
\end{defe}

It follows from Lemma \ref{inv of W^n} that
there is an isomorphism
\[S^v\cong (I_0^{--}\cap\on{Ad}_nI_0^{--})^{\on{inv}\circ\psi_v}.\]

The following proposition implies  that 
$S^v$ defines a
transversal contracting slice 
to the orbit $\cO_v$ inside $LX$ at $x_v$, parallel to Proposition \ref{slice spherical} in the spherical case:
\begin{prop}\label{Iwahori Slice}
	\begin{itemize}
	\item [(i)] The embedding $\tau$ induces an isomorphism $S^v\cong L^{n_v}$. In particular,
	$S^v$ is a connected ind-scheme of ind-finite type and formally smooth.
				\item [(ii)] 
				Choose a cocharacter $\mu\in X_*(T_0^{\psi_{\bar{v}}})$ 
so that
$\langle\mu,\alpha\rangle<0$ 
for all positive root $\alpha>0$ with $\psi_{\bar{v}}\alpha>0$.
Let $m=\max\{1,\langle\mu,\alpha\rangle+1,\forall\alpha\in\Phi\}$ 
where $\Phi$ is the set of roots of $G$. 
				Then $\phi^v_s:\gamma(t)\mapsto s^{\mu-m\theta(\lambda)}\gamma(s^{-2m}t)$ defines a $\bGm$-action on $S^v$, which fixes $x_v$, and contracts $S^v$ to $x_v$.
		
			\item [(iii)] The multiplication map $m:I_0\times S^v\rightarrow LX$ is formally smooth.	
		\item [(iv)] 
		We have $S^v\cap \cO_v=\{x_v\}$
		and 
		$S^v\subset LX$ is transversal to $\cO_v$ at $x_v$.

		\item [(v)] All the $\bGm$-actions $\phi^v_s$ 
		preserve all the $I_0$-orbits $\cO_u$. 
		Any $I_0$-equivariant local system on  $\cO_u$ is $\bGm$-equivariant with respect to the action $\phi_s^v$. 
	\end{itemize}
\end{prop}

\begin{proof}
In 
view of Lemma \ref{G_m action on AF slices} and \ref{p:L^n_n'}, the  same arguments as in Proposition \ref{slice spherical} in the 
 spherical case apply to the 
 Iwahori case,
replacing $L^+\fg, L^{<0}\fg$ with $\on{Lie}(I_0), \on{Lie}(I_0^{--})$.
Part (iii) can also be deduced from the well-known formally smoothness of the multiplication map
\[
m:I_0\times W^n\times I_0\rightarrow LG
\]
by taking $(\on{inv}\circ\theta)$-fixed subspaces as in Lemma \ref{l:2nd pf spherical submersion}.
\end{proof}

\begin{cor}\label{c:bar O equi-singular}
Let $v,u\in\sV$. 
	\begin{itemize}
		\item [(i)] 
            The multiplication map 
            $m:I_0\times S^v_{\leq u}\rightarrow\cO_u$ 
            is formally smooth.
        
		\item [(ii)] 
            There is a $\bGm$-action on $S^v_{\leq u}$, 
            which fixes $x_v$, and contracts $S^v_{\leq u}$ to $x_v$.

            \item [(iii)]
            $S^v_u\neq\emptyset$ if and only if $v\leq u$.
	\end{itemize}
\end{cor}
\begin{proof}
Same proof as in the spherical case 
Corollary \ref{c:LX equi-singular} using Proposition \ref{Iwahori Slice}.
\end{proof}

\subsection{Mars-Springer slices}\label{MS slices}
Note that 
the evaluation map 
$ev:L^+X\to X$ induces a 
 a bijection 
 between 
$I_0$-orbits on $L^+X$
and $B_0$-orbits on $X$.
Let $\cO_u\subset L^+X$ be the unique open $I_0$-orbit such that $\overline\cO_u=L^+X$.
For any $I_0$-orbit $\cO_v\subset L^+X$
let $\cO_{v,X}\subset X$
be the corresponding $B_0$-orbit.
We have $\tau(x_v)=n_v=\bar n_v=\tau(\bar x_v)\in N_G(T_0)^{\on{inv}\circ\theta}$
and let 
\[S_{\bar x_v}:=(U_0^-\cap\on{Ad}_{\bar n_v}U_0^-)\bar x_v\subset X\]
be the Mars-Springer slice 
of $\cO_{v,X}$ at $\bar x_v:=ev(x_v)\in\cO_{v,X}$
introduced in \cite[Section 6.4]{MS}.

\begin{lem}\label{comparison with MS}
The evaluation map $ev:L^+X\to X$ induces an isomorphism 
$S^v_{\leq u}\cong S_{\bar x_v}$
\end{lem}
\begin{proof}
It suffices to show that 
the evaluation map $ev:L^+G\to G$
indcues an isomorphism 
$\tau(S^{v}_{\leq u})\cong\tau(S_{\bar x_v})$
where  $\tau:X\to G$ is the embedding.
We have $ev:I_0^{--}\cap L^+G\cong U_0^-$ and hence  isomorphisms 
\[\tau(S^v_{\leq u})\cong 
\tau(S^v\cap L^+X)
\cong 
(I_0^{--}\cap\on{Ad}_{\bar n_v}I_0^{--})\bar n_v)^{\on{inv}\circ\theta}\cap L^+G
\stackrel{ev}\cong
 (U_0^-\cap\on{Ad}_{\bar n_v}U_0^-)\bar n_v)^{\on{inv}\circ\theta}\cong\tau(S_{\bar x_v})\]
 where the last isomorphism 
 follows from the fact that 
 $(U_0^-\cap\on{Ad}_{\bar n_v}U_0^-)\bar n_v)^{\on{inv}\circ\theta}\cong (U_0^-\cap\on{Ad}_{\bar n_v}U_0^-)^{\on{inv}\circ\psi_{\bar v}}$ is a single 
 $(U_0^-\cap\on{Ad}_{\bar n_v}U_0^-)$-orbit under the 
 twsited conjugation action 
 $u\cdot v=uv\psi_{\bar v}(u)^{-1}$
 (see, e.g., \cite[Proposition 9.1]{Ri}).
The lemma follows.

\end{proof}

\section{Lagrangians}\label{Lagrangian slices}
We show that the  slices 
$S^0_{\leq\lambda}$ 
are Lagrangians inside 
the affine Grassmannian slices $W^0_{\leq\lambda}$.
Along the way, we prove a combinatorial formula 
for the codimensions of spherical orbits.
The proof uses real groups and quasi-maps \cite{CN2} 
and a spreading out argument.
The codimension formula obtained in this section 
provides a simple way to check the semisimplicity criterion 
of relative Satake category in Section \ref{Semi cri}.

\subsection{Quasi-maps}
Consider the stack 
$Z=LG^+\backslash LG/K[t^{-1}]$
of quasi-maps
classifying 
a $G$-bundle $\cE$ on $\mathbb P^1$
and a section $\phi:\mathbb P^1-\{0\}\to\cE\times^G X$, 
i.e., a $K$-reduction $\cE_K$ of 
$\cE|_{\mathbb P^1-0}$
 (here $t$ is the local coordinate at $0$).
We have a natural map 
\begin{equation}\label{Z}
    f:Z=L^+G\backslash LG/K[t^{-1}]\to  L^+G\backslash LX
    \end{equation}
between ind-stacks.
For any $\lambda\in\Lambda_S^+$
 we write
\begin{equation}\label{Z_v}
Z_{\lambda}=(f)^{-1}(L^+G\backslash{LX}_{\lambda}),\ \  
\overline{Z}_{\lambda}=(f)^{-1}(L^+G\backslash\overline{LX}_{\lambda})
\end{equation}
for the pre-images of the corresponding orbits and orbit closures 
and 
\[
f_\lambda:=f|_{\overline Z_\lambda}:
\overline Z_\lambda\to L^+G\backslash\overline{LX}_{\lambda}
\]
for the induced map.

 \begin{lem}\label{finiteness of Z}
$Z$  is an ind-stack ind-locally of finite type and 
$Z_{\lambda}$, $\overline{Z}_{\lambda}$
 are stacks locally of finite type.
 \end{lem}
\begin{proof}
This is \cite[Section 3.4]{GN} and 
\cite[Lemma 3.5.2]{BFT}.
    
\end{proof}
Recall the notion of placid morphisms 
in  \ref{placid morphism}.

 \begin{lem}\label{placid covering}
 The morphisim
 $f_\lambda:\overline Z_\lambda\to L^+G\backslash\overline{LX_\lambda}$ 
 is placid.
\end{lem}
\begin{proof}
 Consier the following Cartesian diagrams
 \[\xymatrix{\overline Z_\lambda
 \ar[d]^{f_\lambda}&\overline Z_\lambda\times_{L^+G\backslash\overline{LX_\lambda}}\overline{LX_\lambda}\ar[r]^{\ \ i}\ar[l]_{\pi\ \ \ \ }\ar[d]^a&LG/K[t^{-1}]\ar[d]^{b}\\
 L^+G\backslash\overline{LX_\lambda}&\overline{LX}_\lambda\ar[r]^i\ar[l]_\pi&LX
 }\]
where $\pi$ are pro-smooth coverings,
$i$ are the natural inclusions, and $b$
is the natural map 
\[b:LG/K[t^{-1}]\to LG/LK\to LX.\]
Note that  $\overline Z_\lambda\times_{L^+G\backslash\overline{LX_\lambda}}\overline{LX_\lambda}$ (resp. $LG/K[t^{-1}]$)
is a $L^+G$-torsor over 
$\overline Z_\lambda$ (resp. $Z$) and the finiteness results in Lemma \ref{finiteness of Z} implies it is a 
placid scheme (resp. a placid ind-scheme).
Since $\pi$ is a pro-smooth covering, to show the placidness of $f_\lambda$, it 
suffices to show that the map $a$ is placid. 
From Lemma \ref{formally smooth}, it suffices to show that 
$a$ is formally smooth. 
Since $a$ is the base change of $b$, we reduce to show that 
$b$ is formally smooth. Since 
$K[t^{-1}]$ and $LG$ are formally smooth, it follows that 
$LG/K[t^{-1}]$ is a formally smooth
placid ind-scheme and the formally smoothness of $b$ follows from
Lemma \ref{formally smooth criterion} and the  fact that the tangent morphism 
$db:T_{LG/K[t^{-1}]}\to b^*T_{LX}$ is surjective.
 \end{proof}

\begin{prop}\label{dim for LX}
For any $\lambda\leq\mu\in\Lambda_S^+$,
    the codimension function 
    $\on{codim}_{\overline{LX}_\lambda/\overline{LX}_\mu}:\overline{LX}_\lambda\to\mathbb Z$ is constant and is equal to $\on{dim}(Z_\mu)-\dim(Z_\lambda)$.

\end{prop}
\begin{proof}
Consider the Cartesian diagram
\[\xymatrix{\overline Z_{ \lambda}\ar[r]\ar[d]^{f_{\lambda}}&\overline Z_{\mu}\ar[d]^{f_{\mu}}\\
L^+G\backslash\overline{LX}_{\lambda}\ar[r]&L^+G\backslash\overline{LX}_{\mu}}.\]
By Lemma \ref{placid covering} the vertical map $f_\mu$
is placid and  \ref{dimension theory} (7) 
 implies 
\begin{equation}\label{equality}
\dim_{\overline Z_{\lambda}/\overline Z_{\mu}}=f_{ \lambda}^*\dim_{(L^+G\backslash\overline{LX}_{\lambda})/(L^+G\backslash\overline{LX}_{\mu})}.
\end{equation}
Since $\overline{LX}_{\lambda}, \overline Z_\lambda, \lambda\in\Lambda_S^+$ are equidimensional placid, \ref{dimension theory} (6) implies that 
 both dimension fucntions $\dim_{{\overline{LX}_{\lambda}/\overline{LX}_{\mu}}}$ 
and $\dim_{(L^+G\backslash\overline{LX}_{\lambda})/(L^+G\backslash\overline{LX}_{\mu})}$ are constant functions with the same value  
and~\eqref{equality} implies the common value is given by
\[
\dim_{{\overline{LX}_{\lambda}/\overline{LX}_{\mu}}}=\dim_{(L^+G\backslash\overline{LX}_{\lambda})/(L^+G\backslash\overline{LX}_{\mu})}=\dim_{\overline Z_{\lambda}/\overline Z_{\mu}}=\dim( Z_{\lambda})-\dim(Z_{\mu}).
\]
Since $\on{codim}_{\overline{LX}_\lambda/\overline{LX}_\mu}=-\on{dim}_{\overline{LX}_\lambda/\overline{LX}_\mu}$, the porposition follows.
\end{proof}

\subsection{Codimension for spherical orbits}\label{ss:codim formula}
In this section we assume 
$LX$ is connected (e.g., when
$G$ is a simply connected semisimple group).

\begin{prop}\label{codim formula}
$\on{codim}_{\overline{LX}_\lambda/\overline{LX}_\mu}:
\overline{LX}_\lambda\to\mathbb Z$
is constant with value equal to 
$\langle\rho,\mu-\lambda\rangle$.
\end{prop}
\begin{proof}
Since $\pi_0(LX)\cong\Lambda_S/\Lambda_S\cap Q=0$ by assumption, Theorem  \ref{t:sym var orbit closure} implies
$LX_0\subset\overline{LX}_\lambda$
for all $\lambda\in\Lambda_S^+$.
We claim that 
$\on{codim}_{\overline{LX}_0/\overline{LX}_\mu}:\overline{LX}_0\to\mathbb Z$
is equal to $\langle\rho,\mu\rangle$.
Then $\on{codim}_{\overline{LX}_\lambda/\overline{LX}_\mu}=
\on{codim}_{\overline{LX}_0/\overline{LX}_\mu}-\on{codim}_{\overline{LX}_0/\overline{LX}_\lambda}$ is the constant function with value $\langle\rho,\mu-\lambda\rangle$.

Proof of the claim. 
By Proposition~\eqref{dim for LX}, we need to show that 
$\dim(Z_\mu)-\dim(Z_0)=\langle\rho,\mu \rangle$.
We have $Z_0=\overline Z_0\cong L^+K\backslash LK/K[t^{-1}]\cong\Bun_K(\mathbb P^1)$
and hence $\dim(Z_0)=-\dim(K)$.
On the other hand,
the open orbit $L^{<0}G\subset L^+G\backslash LG$ gives rise to an open dense embedding $L^{<0}X/K\cong (L^{<0}G/L^{<0}K)/K\subset 
L^+G\backslash LG/K[t^{-1}]$
where $K$ acts on $L^{<0}X$ be the conjugation action. 
Note that we have the following 
Cartesian diagrams 
\[\xymatrix{(L^{<0}X\cap\overline {LX}_\mu)/K\ar[r]^{j}\ar[d]&\overline Z_\mu\ar[r]\ar[d]&L^+G\backslash\overline {LX}_\mu\ar[d]\ar[r]^{\tau}&\overline{LG}_\mu/\on{Ad}_\theta(L^+G)\ar[d]\\
L^{<0}X/K\ar[r]^{j}&L^+G\backslash LG/K[t^{-1}]\ar[r]&L^+G\backslash LX\ar[r]^{\tau}&LG/\on{Ad}_\theta(L^+G)}\]
where the vertical maps are natural inclusions,
the maps $j$ are open dense embeddings, 
and 
$LG/\on{Ad}_\theta(L^+G)$ is the quotient 
of $LG$ by the $\theta$-twisted conjugation action. Since $L^{<0}X\cap\overline{LX}_\mu=L^0_{\leq\mu}$
and $(L^{<0}X\cap\overline{LX}_\mu)/K\subset \overline{Z}_\mu$ is open dense, 
Lemma \ref{dim} below implies
\[\dim({Z}_\mu)=
\dim((L^{<0}X\cap\overline{LX}_\mu)/K)=
\dim(L^0_\mu/K)=
\langle\rho,\mu\rangle-\dim(K).\]
All together, we obtain 
\[\dim( Z_\mu)-\dim( Z_0)=(\langle\rho,\mu\rangle-\dim(K))-(-\dim(K))=\langle\rho,\mu \rangle.\]
The claim follows.
\end{proof}

\begin{lem}\label{dim}
We have $\on{dim}(L_\mu^0)=\langle\rho,\mu\rangle$.
\end{lem}
\begin{proof}
Pick a
 mixed characteristic DVR $A$ between $\mathbb Z[1/2]$
and $\mathbb C$ with residue field $k$ and a smooth model 
$G_A$ of $G$ over $A$.
According to the Isomorphism Theorem in \cite[Theorem 6.1.17]{BC}, one can find an involution $\theta_A:G_A\to G_A$ such that 
$\theta_A\otimes k=s\theta s^{-1}$ for some $s\in T$. 

Note that  $W^0_\mu=L^{<0}G\cap LG_{\mu}$
admits a smooth model 
$W^0_{\mu,A}=L^{<0}G_A\cap LG_{\mu,A}$ over $A$ and, by \cite[Lemma 9.4]{Gro2}, the fixed points $L_{\mu,A}^0=(W^0_{\mu,A})^{\on{inv}\circ\theta_A}$ is smooth over $A$ with special fiber $(L^0_{\mu,A})\otimes_Ak\cong L^0_\mu$.
Thus we have $\dim(L^0_\mu)=
\dim((L^0_{\mu,A})\otimes_Ak)=\dim((L^0_{\mu,A})\otimes_A\mathbb C)$
and we reduce to show 
$\dim((L^0_{\mu,A})\otimes_A\mathbb C)=\langle\rho,\mu\rangle$.

Let $G_\bR$ be the real form 
associated to the involution $\theta_\bC=\theta_A\otimes\bC$
on the complex reductive group $G_\bC=G_A\otimes\bC$.
Let $\Gr_\mathbb R=LG_\mathbb R/L^+G_\mathbb R$
be the real affine Grassmannian for $G_\bR$.
For any real dominant weight $\mu\in\Lambda_S^+$, we denote by
$\Gr_{\mathbb R,\mu}=LG_\mathbb R\cdot[t^\mu]$ the corresponding 
real spherical orbit and its intersection
$\Gr^0_{\mathbb R,\mu}:=\Gr_{\mathbb R,\mu}\cap L^{<0}G_\mathbb R\cdot[t^0]$
with the open $L^{<0}G_\mathbb R$-orbit $L^{<0}G_\mathbb R\cdot[t^0]\subset\Gr_\mathbb R$, see \cite[Section 3.6]{NadlerRealGr}.
According to \cite[Theorem 6.9]{CN2}, there is a real analytic isomorphism 
    \[(L_{\mu,A}^0)
    \otimes\bC\cong 
    (L^{<0}G_\bC\cap LG_{\mu,\bC})^{\on{inv}\circ\theta_\bC}\cong
    \Gr_{\mathbb R,\mu}^0\] and hence by
   \cite[Proposition 3.6.1]{NadlerRealGr} we have 
    $\dim((L_{\mu,A}^0)
    \otimes\bC)=\frac{1}{2}\dim_{\mathbb R}(\Gr_{\mathbb R,\mu}^0)=\langle\rho,\mu\rangle$. The lemma follows.

\end{proof}

\subsection{Lagrangian slices}
Recall the affine Grassmannian slice $W^\lambda_{\leq\mu}$
has a natural Poisson structure and $W^\lambda_\nu$ for $\lambda\leq\nu\leq\mu$
are symplectic leaves of dimension 
$\dim W^\lambda_\nu=2\langle\rho,\nu-\lambda\rangle$
(see, e.g., \cite{KWWY}). 
As an application of the co-dimension formula,
 we show that, when $\lambda=0$, the fixed points subvariety $L^0_\mu=(W_\mu^0)^{\on{inv}\circ\theta}\subset W^0_\mu$ 
is a Lagrangian subvariety of dimension $\dim(L_\mu^0)=\langle\rho,\mu\rangle$.

We preserve the setup in Section \ref{ss:codim formula}.
We further assume $p$ is large enough so that 
there is a non-degenerate symmetric bilinear invariant form 
$(\ ,\ )$ on $\fg$
that is also $\theta$-invariant.

\subsubsection{Poisson structures on loop groups and affine Grassmannians}
It is well-known that $(L\fg,L^{<0}\fg,L^+\fg)$ form
a \emph{Manin triple},
which we review. 
Fix a non-degenerate symmetric bilinear invariant form $(\ ,\ )$ on $\fg$
that is also $\theta$-invariant.
Extend it $F$-linearly to $L\fg=\fg(F)$. 
Composing with residue map, we get a $k$-valued inner product on $L\fg$,
which we still denote by $(\ ,\ )$.
Then $L^{<0}\fg,L^+\fg$ are both Lagrangian subspaces of $L\fg$
with respect to $(\ ,\ )$.
Fix a basis $e_i^*$ of $L^{<0}\fg$
and a basis $e_i$ of $L^+\fg$
such that $(e_i^*,e_j)=\delta_{ij}$. 
Define $r$-matrix
\[
r=\sum_i e_i^*\wedge e_i\in\wedge^2L\fg,
\]
where the formal sum is understood as an element of 
$\fg(\!(t^{\pm})\!)\otimes\fg(\!(t^\pm)\!)$.

The definition of $r$ is independent of the choice of orthonormal basis. Since $\theta$ acts continuously on $L^{<0}\fg$ and $L^+\fg$, we have
\[
\theta(r)=r.
\]

Define a bi-vector field $\pi\in\Gamma(LG,\wedge^2TLG)$ by
\[
\pi(g):=R_{g*} r-L_{g*} r,
\]
where $g\in LG$, $L_g,R_g$ are left and right translations by $g$.
Then $\pi$ gives a Poisson bi-vector and defines a Poisson structure on $LG$. Moreover, $L^{<0}G,L^+G$ are Poisson subgroups of $LG$.

There are two ways to define a bi-vector field on $\Gr$ from $r$. 
Let $p:LG\rightarrow\Gr$ be the projection map.
Define $\bar{\pi}_1\in\Gamma(\Gr,\wedge^2T\Gr)$ by
$\bar{\pi}_1(g)=d_gp(\pi_g)$.

On the other hand, the action of $LG$ on $\Gr$
defines a map $\kappa:L\fg\rightarrow\Gamma(\Gr,T\Gr)$.
Let $\bar{\pi}_2=\kappa(r)\in\Gamma(\Gr,\wedge^2T\Gr)$.

\begin{lem}\label{l:two Poisson}
	We have $\bar{\pi}_1=\bar{\pi}_2$.
\end{lem}
\begin{proof}
	At any $g\in LG$, let $\bar{g}=p(g)$.
	Since $e_i\in L^+\fg$, 
	we have $p(ge^{te_i})=\bar{g}$.
	Thus $d_gp(L_{g*}e_i)=0$, $d_gp(L_{g*}r)=0$, 
	$\bar{\pi}_1(g)=dp(R_{g*}r)$.
	Also, for $a\in L\fg$, 
	$\kappa(a)(\bar{g})$ is represented by curve $p(e^{ta}g)$,
	i.e. $\kappa(a)(\bar{g})=d_gp(R_{g*}a)$.
	Thus $\bar{\pi}_1(\bar{g})=\bar{\pi}_2(\bar{g})=d_gp(R_{g*}r)$.
\end{proof}

Denote $\bar{\pi}=\bar{\pi}_1=\bar{\pi}_2$.
From \cite[Theorem 2.3]{LYPoisson},
we know $\bar{\pi}$ defines a Poisson structure on $\Gr$,
and both $L^+G$-orbits $p(LG_\lambda)=L^+G\cdot [t^\lambda]$
and $L^{<0}G$-orbits $p(W^\lambda)=L^{<0}G\cdot[t^\lambda]$
are Poisson subvarieties.
The intersections of $L^+G$-orbits and $L^{<0}G$-orbits, 
when nonempty, 
are symplectic leaves,
c.f. \cite[Corollary 2.9]{LYPoisson}, \cite[Theorem 2.5]{KWWY}.
 
\subsubsection{Lagrangian slices}\label{LS}
The projection $p:LG\rightarrow\Gr=LG/L^+G$ 
restricts to an embedding on $W^\lambda$, thus
$p:W^\lambda_\mu\xrightarrow{\sim}L^{<0}G\cdot[t^\lambda]\cap L^+G\cdot [t^\mu]$.
The pullback induces a symplectic structure on $W^\lambda_\mu$.

\begin{lem}
$L^\lambda_\mu\subset W^\lambda_\mu$ is a coisotropic subvariety.
\end{lem}
\begin{proof}
This is equivalent to that 
the ideal of $L^\lambda_\mu$ in $W^\lambda_\mu$ is closed under Poisson bracket.
For this, it suffices to prove for the embedding 
$L^\lambda\subset W^\lambda$,
where the Poisson structure on $W^\lambda$ 
is pullback from $p:W^\lambda\xrightarrow{\sim}L^{<0}G\cdot[t^\lambda]\subset\Gr$.

Denote $\sigma=\on{inv}\circ\theta$, $L^\lambda=(W^\lambda)^\sigma$.
For any $g\in L^\lambda$, 
$g\in\tau(LX)=\tau(LG/LK)$,
so we can write $g=h\theta(h)^{-1}$ for some $h\in LG$.
For $x\in L\fg$,
if $e^{tx}g\in(LG)^\sigma$,
then $x=-\Ad_g(d\theta(x))$,
$d\theta(\Ad_{h^{-1}}x)=-\Ad_{h^{-1}}x$,
$x\in\Ad_h L\fp$.
Thus 
\[
T_g(L^\lambda)=R_{g*}(L^{<0}\fg\cap\Ad_{t^\lambda}L^{<0}\fg\cap\Ad_hL\fp).
\]
Under the inner product $(\ ,\ )$, 
the conormal space
$(N^*_{L^\lambda}W^\lambda)(g)=\{\phi\in T^*_gW^\lambda\ |\ \phi(T_gL^\lambda)=0\}$ 
can be identified with $N_g:=R_{g*}(L^+\fg\cap\Ad_{t^\lambda}L^+\fg\cap\Ad_hL\fk)$.

In view of Lemma \ref{l:two Poisson}, to show the coisotropicity, 
we need to show that the contraction of conormal bundle with bivector $\pi$
\[
N^*_{L^\lambda}W^\lambda\hookrightarrow T^*W^\lambda|_{L^\lambda}\xrightarrow{\pi}TW^\lambda|_{L^\lambda}
\]
lands inside $TL^\lambda$. 
Since 
$T_gL^\lambda=\{x\in T_gW^\lambda\subset T_gLG\ |\ (x,N_g)=0\}$,
it remains to show that for any 
$g=h\theta(h)^{-1}\in L^\lambda$
and $x,y\in N_g$,
$(\pi_g,x\wedge y)=0$. 
Write $x=R_{g*}\Ad_h a$, $y=R_{g*}\Ad_h b$, 
where $a,b\in L\fk$. 
Since $(\ ,\ )$ is $LG$-invariant, we have
\[
(\pi_g,x\wedge y)
=(R_{g*}r-L_{g*}r,R_{g*}\Ad_h(a\wedge b))
=(\Ad_{h^{-1}}r-\Ad_{\theta(h^{-1})}r,a\wedge b).
\]

Recall $\theta(r)=r$. 
Note that $\theta$ acts trivially on $a,b\in L\fk$ 
and $(\ ,\ )$ is $\theta$-invariant. 
Thus
\[
(\Ad_{\theta(h^{-1})}r,a\wedge b)=(\Ad_{h^{-1}}r,a\wedge b),
\]
which completes the proof.
\end{proof}

\begin{prop}\label{t:Lagrangian}
$L^0_\mu\subset W^0_{\mu}$
is a Lagrangian subvariety.
\end{prop}
\begin{proof}
    From Lemma \ref{dim} and Lemmra \ref{t:Lagrangian}, we see that 
    $L^0_\mu\subset W^0_{\mu}$ is coisotropic subvarities of 
    dimension $\dim(L^0_\mu)=\frac{1}{2}\dim(W^\lambda_\mu)=\langle\rho,\mu\rangle$, and hence is a
    Lagrangian.
\end{proof}
	
\begin{rem}	
Specializing Proposition \ref{t:Lagrangian} to 
the case in Proposition \ref{slice=nilp}, 
we obtain (special case of) a well-known result 
of Kostant-Rallis saying that 
the symmetric nilpotent orbits $\cO_{\fp,\lambda}$
are Lagrangian subvarieties of the nilpotent orbits 
$\cO_\lambda$ with respect to the Kostant-Kirillov symplectic forms.
\end{rem}

\section{Affine Hecke modules}\label{Affine Hecke modules}
Let $\sD$ be the set of pairs 
$(\xi,v)$ where $v\in\sV$ and 
$\xi$ is an $I_0$-equivariant local system on 
the  $I_0$-orbit $\cO_v\subset LX$.
Let $M$ be the free $\mathbb Z[q^{\frac{1}{2}},q^{\frac{-1}{2}}]$-module with basis 
indexed by $\sD$.
In this section we endow 
$M$ with a module structure over the affine Hecke algebra of $G$
and an anti-linear involution $D_\delta:M\to M$ 
compatibile with the module structure. 
The construction is geometric 
and relies on the $\ell$-adic sheaf theory
on placid (ind)schemes,
where $\ell$ is a prime number $p\neq\ell$ 
(see Appendix \ref{dimension theory} for a review).
 
For the rest of the paper, we will fix 
a dimension theory $\delta$ on $LX$ (see Proposition \ref{dim for LX}).
Note that, 
for any $v\in\sV$ (resp. $\lambda\in\Lambda_S^+$), 
the function $\delta_{\overline\cO_v}:\overline\cO_v\to\mathbb Z$
(resp. $\delta_{\overline{LX}_\lambda}:
\overline{LX}_\lambda\to\mathbb Z$)
is constant and 
we will write
$\delta(v)$ (resp. $\delta(\lambda)$)
for its value 
to be called the dimension of 
$\cO_v$ (resp. $LX_\lambda$).


\subsection{Affine Hecke module categories}
\subsubsection{
Equivariant sheaves on $LG$ and $LX$}\label{category of sheaves}
Recall the Iwahori decomposition
$LG=\bigsqcup_{w\in\widetilde{\mathrm{W}}} I_0wI_0$ of the loop group into $I_0\times I_0$-orbits. Here $\widetilde{\mathrm{W}}\simeq N_{LG}(LT_0)/L^+T_0$ is the extended Weyl group.
Write $LG_w=I_0wI_0$ and $\overline{LG}_w$ be the 
orbit closure.
We denote by 
$D(I_0\backslash LG/I_0)$ the dg category of  
$I_0\times I_0$-equivariant  constructible complexes on $LG$, known as the  affine Hecke category.
It has a Verdier duality functor $\mathbb D:D(I_0\backslash LG/I_0)\cong D(I_0\backslash LG/I_0)^{op}$
and a monoidal structure given by the convolution product:
consider the maps 
\[LG/I_0\times  LG/I_0\stackrel{p\times\on{id}}\leftarrow  LG\times LG/I_0\stackrel{q}\to  LG\times^{I_0}LG/I_0\stackrel{m}\to  LG/I_0\]
where $p$ and $q$ are the natural quotient maps
and $m$ is the multiplication map.
Given $\mathcal M_1,\mathcal M_2\in D^{}(I_0\backslash LG/I_0)$  we have 
\begin{equation}\label{Hecke action}
	\mathcal M_1\star\mathcal M_2:=m_!(\mathcal M_1\tilde\boxtimes\mathcal M_2)
\end{equation}
where $\mathcal M_1\tilde\boxtimes\mathcal M_2\in D^{}(I_0\backslash LG\times^{I_0}LG/I_0)$ is the unique complex 
such that 
$({p\times\on{id}})^*(\mathcal M_1\boxtimes\mathcal M_2)\cong q^*(\mathcal M_1\tilde\boxtimes\mathcal M_2)$.

We denote by
\[D(I_0\backslash LX)\cong
\on{colim}_{\underline v,!}D(I_0\backslash\overline\cO_{\underline v})\] 
the dg category $I_0$-equivariant 
constructible complexes on the $I_0$-placid ind-scheme $LX\cong\on{colim}_{\underline v} \overline\cO_{\underline v}$.
We have the affine Hecke action, denoted by $\star$, of the  affine Hecke category  $D(I_0\backslash LG/I_0)$ on $D(I_0\backslash LX)$ defined as follows.
Consider the maps 
\[LG/I_0\times LX\stackrel{p\times\on{id}}\leftarrow LG\times LX\stackrel{q}\to LG\times^{I_0}LX\stackrel{a}\to LX\]
where $a$ is the action map.
Given $\mathcal M\in D^{}(I_0\backslash LG/I_0)$ and $\cF\in D^{}(I_0\backslash LX)$ we have 
\begin{equation}\label{Hecke action}
	\mathcal M\star\cF:=a_!(\mathcal M\tilde\boxtimes\cF)
\end{equation}
where $\mathcal M\tilde\boxtimes\cF\in D^{}(I_0\backslash LG\times^{I_0}LX)$ is the unique complex 
such that 
$({p\times\on{id}})^*(\mathcal M\boxtimes\cF)\cong q^*(\mathcal M\tilde\boxtimes\cF)$.

The affine Hecke category $D(I_0\backslash LG/ I_0)$
has a standard $t$-structure with abelian heart 
$\on{Shv}(I_0\backslash LG/ I_0)$ the abelian category of $I_0\times I_0$-equivariant
constructible sheaves on $LG$ and a perverse $t$-structure with abelian heart 
$\on{Perv}(I_0\backslash LG/ I_0)$ the abelian category $I_0\times I_0$-equivariant perverse sheaves
on $LG$. We denote by $\sH^n:D(I_0\backslash LG/ I_0)\to\on{Shv}(I_0\backslash LG/ I_0)$
and ${^p}\sH^n:D(I_0\backslash LG/ I_0)\to\on{Perv}(I_0\backslash LG/ I_0)$ the corresponding cohomological functors.

The category $D(I_0\backslash LX)$
also has a standard $t$-structure with abelian 
heart  $\on{Shv}(I_0\backslash LX)$
the $I_0$-equivariant constructible sheaves on $LX$.
According to \ref{t-structures}, the dimension theory $\delta$ on $LX$
gives rise to
a $*$-adapted perverse $t$-structure on $D(I_0\backslash LX)$
with abelian heart $\on{Perv}_\delta(I_0\backslash LX)$ called the abelian category of $I_0$-equivariant perverse on $LX$.
We denote by $\sH^n:D(I_0\backslash LX)\to\on{Shv}(I_0\backslash LX)$
and ${^p}\sH^n_\delta:D(I_0\backslash LG/ I_0)\to\on{Perv}_\delta(I_0\backslash LX)$ the corresponding cohomological functors.

\subsubsection{Verdier duality}
Denote by $\mathbb D:D(I_0\backslash LG/I_0)\cong D(I_0\backslash LG/I_0)^{op}$
the Verdier duality functor.
By Lemma \ref{Verdier duality}, the dimension theory $\delta$ gives rise to a Verdier duality functor
\[\mathbb D_\delta: D(I_0\backslash LX)\cong D(I_0\backslash LX)^{op}\]
which is $t$-exact with respect to the $*$-adapted
perverse $t$-structure above and 
satisfying $\mathbb D_\delta^2\cong\on{Id}$.

\begin{prop}\label{decomposition theory}
\begin{enumerate}
    \item [(i)]
    For any $\mathcal M\in D(I_0\backslash LG/I_0)$
and $\cF\in D(I_0\backslash LX)$, 
there is an isomorphism
	$\mathbb D_\delta(\mathcal M\star\cF)\cong\mathbb D(\mathcal M)\star\mathbb D_\delta(\cF)$
    \item [(ii)]
    Assume both $\mathcal M\in D(I_0\backslash LG/I_0)$ 
and $\cF\in D(I_0\backslash LX)$
are semi-simple complexes. 
Then the convolution $\mathcal M\star\cF$ is semi-simple.
\end{enumerate}
\end{prop}
\begin{proof}
There are  $w\in\widetilde{\mathrm{W}}$, $v\leq v'\in\sV$ such that 
$\mathcal M\in D(I_0\backslash\overline{LG}_w/I_0)$, $\cF\in D(I_0\backslash\overline{\cO}_v)$,
and $\mathcal M\star\cF\in D(I_0\backslash\overline\cO_{v'})$.
One can find \'etale covers
$U_{v'}\to\overline\cO_{v'}$ 
(resp. $W\to \overline{LG}_w\times\overline\cO_v$) 
admitting placid presentations, 
fitting into the 
following commutative diagram
\[
\xymatrix{\overline{LG}_w/I_0\times U_v\ar[d]&W\ar[r]^q\ar[d]\ar[l]_p&V\ar[r]^{a}\ar[d]&U_{v'}\ar[d]\\
\overline{LG}_w/I_0\times\overline{\cO}_v&\overline{LG}_w\times\overline{\cO}_v\ar[l]_p\ar[r]^q&\overline{LG}_w\times^{I_0}\overline{\cO}_v\ar[r]^a&\overline{\cO}_{v'}}.
\]
Here $V=U'\times_{\overline{\cO}_{v'}}(\overline{LG}_w\times^{I_0}\overline\cO_v)\to \overline{LG}_w\times^{I_0}\overline\cO_v$ and 
$U_v=U_{v'}\times_{\overline\cO_{v'}}\overline\cO_v\to\overline\cO_v$ are the \'etale covers obtained by the base changes.
Let $a^{[n]}:U^{[n]}_{v'}\to V^{[n]}$ be the map between the 
the \v Cech nerves associated to the smooth covers 
$U_{v'}\to I_0\backslash\overline\cO_{v'}$
and $V\to I_0\backslash \overline{LG}_w\times^{I_0}\overline\cO_v$.
Then by~\eqref{Placid schemes} (4) there are placid presentations $U^{[n]}_{v'}\cong\lim_j U^{[n]}_{v',j}$,
$U^{[n]}_{v}\cong\lim_j U^{[n]}_{v,j}$,
$V^{[n]}\cong\lim_j  V^{[n]}_j$, 
such that  $U_v^{[n]}\to U^{[n]}_{v'}$
and $a^{[n]}: V^{[n]}\to U_{v'}^{[n]}$
are base changes from $U^{[n]}_{v,j_0}\to U^{[n]}_{v',j_0}$
and $a^{[n]}_{j}:V^{[n]}_{j}\to U^{[n]}_{v',j}$ for some $j$.

Proof of (i).
We have 
\begin{equation}\label{boxtimes}
    \mathcal M\tilde\boxtimes\mathcal F\cong \on{lim}^*_{[n]}\on{colim}^*_j\cF_{V,j}^{[n]}\in
D(I_0\backslash\overline{LG}_w\times^{I_0}\overline\cO_v)\cong\on{lim}_{[n]}^*\on{colim}_j^* D(V^{[n]}_j)
\end{equation}
and it follows that 
   \begin{equation}\label{formula}
       \mathcal M\star\cF=a_!\mathcal M\tilde\boxtimes\mathcal F\cong a_!(\on{lim}^*_{[n]}\on{colim}^*_j\cF_{V,j}^{[n]})\cong
    \on{lim}_{[n]}^*\on{colim}_j^*a^{[n]}_{j,!}(\cF_{V,j}^{[n]}) 
    \end{equation}
    \[\in
D(I_0\backslash\overline\cO_v)\cong\on{lim}_{[n]}^*\on{colim}_j^* D(U^{[n]}_{v',j}).
\]
From the construction of Verdier duality in~\eqref{Verdier duality}, we have  
\[\mathbb D_\delta(\mathcal M\star\mathcal F)\stackrel{\eqref{formula}}\cong\mathbb D_\delta(\on{lim}_{[n]}^*\on{colim}_j^*a^{[n]}_{j,!}(\cF_{V,j}^{[n]}))
\cong\on{lim}_{[n]}^*\on{colim}_j^*(\mathbb Da^{[n]}_{j,!}((\cF_{V,j}^{[n]})\langle-2\dim(U_{v',j}^{[n]}+2\delta_{v'})\rangle))
\cong\]
\[\cong\on{lim}_{[n]}^*\on{colim}_j^*(a^{[n]}_{j,!}(\mathbb D(\cF_{V,j}^{[n]})\langle-2\dim(U_{v',j}^{[n]}+2\delta_{v'})\rangle))
\cong\]
\[\cong a_!(\on{lim}_{[n]}^*\on{colim}_j^*(\mathbb D(\cF_{V,j}^{[n]})\langle-2\dim(U_{v',j}^{[n]}+2\delta_{v'})\rangle))).\]
We claim that there is an isomorphism
\[\mathbb D\mathcal M\tilde\boxtimes\mathbb D_\delta\cF\cong
\on{lim}_{[n]}^*\on{colim}_j^*(\mathbb D(\cF_{V,j}^{[n]})\langle-2\dim(U_{v',j}^{[n]}+2\delta_{v'})\rangle))\in D(I_0\backslash\overline{LG}_w\times^{I_0}\overline\cO_v)\]
and it follows that 
\[\mathbb D_\delta(\mathcal M\star\mathcal F)\cong a_!(\mathbb D\mathcal M\tilde\boxtimes\mathbb D_\delta\cF)\cong \mathbb D\mathcal M\star\mathbb D_\delta\cF.\]
Part (i) follows.

Proof of the claim. 
By definition the complex $\mathcal M\tilde\boxtimes\cF$ is characterized by the isomorphism  $p^*(\mathcal M\boxtimes\cF)\cong q^*(\mathcal M\tilde\boxtimes\cF)$.
There are isomorphisms 
\begin{equation}\label{p=q Cech}
    (p^{[n]}_j)^*(\mathcal M\boxtimes\cF^{[n]}_j)\cong(q^{[n]}_j)^*(\cF_{V,j}^{[n]})
\end{equation}
inducing 
\[p^*(\mathcal M\boxtimes\cF)\cong\on{lim}^*_{[n]}\on{colim}^*_j (p^{[n]}_j)^*(\mathcal M\boxtimes\cF^{[n]}_j)\cong
\on{lim}^*_{[n]}\on{colim}^*_j (q^{[n]}_j)^*(\cF_{V,j}^{[n]})\cong q^*(\mathcal M\tilde\boxtimes\cF)\]
where 
$\mathcal F\cong\on{lim}_{[n]}^*\on{colim}^*_j\cF^{[n]}_j\in D(I_0\backslash\overline\cO_v)\cong \on{lim}_{[n]}^*\on{colim}^*_j D(U_{v,j}^{[n]})$
and 
\begin{align*}
p^{[n]}\cong \lim_jp^{[n]}_j&:W^{[n]}\cong W^{[n]}_j\to
\overline{LG}_w/I_0\times U_{v}^{[n]}\cong\lim_j
\overline{LG}_w/I_0\times U_{v,j}^{[n]},\\
q^{[n]}\cong\lim_jq^{[n]}_j&:W^{[n]}\cong \lim_jW^{[n]}_j\to V^{[n]}\cong\lim_jV^{[n]}_j
\end{align*}
are the maps between \v Cech nerves.
Since $f^*\mathbb D\langle-2\dim(X)\rangle\cong \mathbb D\langle-2\dim(Y)\rangle f^*$ for a smooth map $f:X\to Y$
between finite type schemes, we conclude from~\eqref{p=q Cech} that 
\begin{equation}\label{p=q}
    p^{[n]*}_j(\mathbb D\mathcal M\boxtimes\mathbb D\cF_{j}^{[n]}\langle-2\dim(U_{v,j}^{[n]})-2l(w)\rangle)\cong q^{[n],*}_j(\mathbb D\cF_{V,j}^{[n]})\langle-2\dim(V^{[n]}_j)\rangle.
\end{equation}
On the other hand,
the discussion of dimension functions in~\eqref{dimension theory}
implies that
\begin{equation}\label{dim_a}
  \dim_a=\dim(V_j^{[n]})-\dim(U^{[n]}_{v',j})=\delta_v+l(w)-\delta_{v'}. 
\end{equation}
All together we get
\begin{align*}
p^*(\mathbb D\mathcal M\boxtimes\mathbb D_\delta\cF)
&\ \cong
\on{lim}_{[n]}^*\on{colim}^*_j(p^{[n]*}(\mathbb D\mathcal M\boxtimes\mathbb D\cF_{j}^{[n]}\langle-2\dim(U_{v,j}^{[n]})+2\delta_v\rangle))\\
&\stackrel{~\eqref{p=q}}\cong 
\on{lim}_{[n]}^*\on{colim}^*_j(q^{[n],*}(\mathbb D\cF_{V,j}^{[n]})\langle-2\dim(V^{[n]}_j)+2\delta_v+2l(w)\rangle)\\
&\stackrel{~\eqref{dim_a}}\cong 
\on{lim}_{[n]}^*\on{colim}^*_j(q^{[n],*}(\mathbb D\cF_{V,j}^{[n]})\langle-2\dim(U^{[n]}_{v',j})+2\delta_{v'}\rangle)\\
&\ \cong 
q^*(\on{lim}_{[n]}^*\on{colim}^*_j(\mathbb D\cF_{V,j}^{[n]})\langle-2\dim(U^{[n]}_{v',j})+2\delta_{v'}\rangle).
\end{align*}
The desried claim follows.

Proof of (ii).
To show that $\mathcal M\star\mathcal F$ is semi-simple it suffices to check that each 
$a^{[n]}_{j,!}(\cF_{V,j}^{[n]})$
in~\eqref{formula}
is semi-simple.
Note that since $\mathcal M\tilde\boxtimes \cF$ is semi-simple it implies 
each $\cF^{[n]}_{V,j}\in D(V_j^{[n]})$
is semi-simple. 
On the other hand,
since $a$ proper it follows from 
\cite[Theorem 8.10.5]{EGA IV}
that each $a_j^{[n]}$ is proper, 
and the decomposition theorem implies each
$a^{[n]}_{j,!}(\cF_{V,j}^{[n]})\in D(U_{v',j}^{[n]})$ is semi-simple.
Part (ii) follows.

\end{proof}

\subsubsection{Equivariant local systems}
We shall give a description of 
equivariant local systems on 
the orbits $LG_w$ and $\cO_v$.
Denote by $\pi_{T_0}:I_0\rightarrow I_0/I_0^+\simeq T_0$ the quotient map.
For an orbit $\cO_v$,
note that $T_0$ is stable under involution $\psi_v$.
Thus $\pi_{T_0}(I_0\cap(LG)^{\psi_v})=T_0^{\psi_v}$.
Denote $T_v:=T_0^{\psi_v}$.
The $T_0$-equivariant rank one local systems on $T_0/T_v$ are given by
characters $X(\pi_0(T_v))$.
For $\xi\in X(\pi_0(T_v))$,
denote the corresponding local system on $T_0/T_v$ by $\cL_\xi$.

\begin{lem}\label{l:equiv local systems}
\begin{itemize}
\item [(i)]
 Any 
$I_0\times I_0$-equivariant local system $\cL_w$ on the $I_0\times I_0$-orbit $LG_w$ is isomorphic to the constant local system 
$\cL_w\cong\overline{\mathbb Q}_{\ell,LG_w}$ with the canonical equivariant structure.
    
  \item [(ii)] Any $I_0$-equivariant local system on an $I_0$-orbit $\cO_v$
	is a direct sum of rank one local systems of the form
	$\cL_{\xi,v}:=\pi_v^*\cL_\xi$, $\xi\in X(\pi_0(T_v))$,
	where 
	$\pi_v:\cO_v=I_0x_v\simeq I_0/I_0\cap(LG)^{\psi_v}
	\rightarrow T_0/T_v$.
        
\end{itemize}
\end{lem}
\begin{proof}
Part (i) follows from the  standard fact
that points in $LG_w$ have connected 
$I_0\times I_0$-stabilizers. 
Proof of (ii).
	Let $C_v=I_0\cap(LG)^{\psi_v}$, $D_v=\pi_{T_0}(C_v)=T^{\psi_v}$.
	Denote their neutral components by $C_v^0,D_v^0$.
	The map $C_v\rightarrow D_v$
	has fibers isomorphic to $I_0^+\cap(LG)^{\psi_v}$.
	Since $I_0^+$ is pro-unipotent, 
	$I_0^+\cap(LG)^{\psi_v}$ is connected.
	Thus the induced map $C_v/C_v^\circ\rightarrow D_v/D_v^\circ$
	is an isomorphism.
	As a consequence, $C_v/C_v^0$ is a commutative finite group.
	Its representations give the equivariant local systems.
\end{proof}

\begin{defe}\label{the set sD}
Let $\sD$ be the set of pairs 
$(\xi,v)$ where $v\in\sV$ and 
$\xi$ is an $I_0$-equivariant local system on 
the  $I_0$-orbit $\cO_v\subset LX$.
\end{defe}

\subsection{Affine Lusztig-Vogan modules}\label{Affine LV modules}
We assume $G$, $B_0$, $T_0$, $\theta$ are defined over 
a finite field $\mathbb F_q$ and $T_0$ is split 
over $\mathbb F_q$. 
There are Frobenius morphisms
$F:LG\to LG$, $F:LX\to LX$, etc.
Consider the group $A=\overline{\mathbb Q}_\ell^\times/\mu$ where
$\mu\subset\overline{\mathbb Q}_\ell^\times$
is the subgroup of roots of $1$.
Let $\mathbb Z[A]$ be the group ring of $A$.
We fix a square root $q^{\frac{1}{2}}$
of $q$ and 
identify $\mathbb Z[q^{\frac{1}{2}},q^{\frac{-1}{2}}]\subset \mathbb Z[A]$ sending $q^{\frac{1}{2}}$ to its image in $A$.

\subsubsection{The affine Hecke module $M$}
We follow closely the discussions in \cite{L,LV,MS}.
For any positive integer $n$,
let $\on{Shv}(I_0\backslash LX)^{F^n}$
be the abelian category consisting of pairs
$(\cF,\Phi_n)$
where $\cF$
is an $I_0$-equivariant sheaf on $LX$
and 
$\Phi_n:(F^n)^*\cF\cong\cF$ is an isomorphism.
For any positive integer $m$ divisible by $n$, we have a natural exact functor 
$\on{Shv}(I_0\backslash LX)^{F^n}\to\on{Shv}(I_0\backslash LX)^{F^m}$ sending 
$(\cF,\Phi_n)$ to $(\cF,(\Phi_n)^{m/n})$
and we define 
\[
\on{Shv}(I_0\backslash LX)^{\text{Weil}}
:=\on{colim}_{n}\on{Shv}(I_0\backslash LX)^{F^n}.
\]
Thus $\on{Shv}(I_0\backslash LX)^{\text{Weil}}$ is an abelian category with objects consisting of pairs
$(\cF,\Phi)$
where $\Phi$,
to be called a Weil structure on $\cF$,
is a system of isomorphism 
$\Phi_n:(F^n)^*\cF\cong\cF$
for sufficiently divisible $n$
such that $(\Phi_n)^{m/n}=\Phi_{m}$
and two systems 
$(\cF,\Phi)$ and $(\cF,\Phi')$
 are isomorphic if 
there exists a multiple $m$ of $n_0$
and $n_0'$ such that $\Phi_m=\Phi'_m$.
Given an object $(\cF,\Phi)\in\on{Shv}(I_0\backslash LX)^{\text{Weil}}$
and a point $\gamma\in LX$, we have 
$\gamma\in (LX)^{F^n}$ for sufficiently divisible $n$
and hence  an endomorphism $\Phi_n:\cF_x\cong\cF_x$
of the stalk $\cF_x$ of $\cF$ at $x$. We can define the 
Frobenius eigenvalues of $\Phi$ on $\cF_x$,
as elements in the group $A=\overline{\mathbb Q}_\ell^\times/\mu$,
as being the $n$-th roots of the eigenvalues of $\Phi_{n}$
for sufficiently divisible $n$.
We define $\cC_{LX}\subset \on{Shv}(I_0\backslash LX)^{\text{Weil}}$ to be the full subcategory 
consisting of $(\cF,\Phi)$ such that 
the eigenvalues of $\Phi$ lands in the subgroup $q^{\frac{\mathbb Z}{2}}\subset A$. We denote by $M=\K(\cC_{LX})$
the Grothendieck group of $\cC_{LX}$.

For any equivariant local system 
$\cL_{\xi,v}$ on the orbit $\cO_v$
we have 
the extension by zero sheaf (i.e.,  the standard object)
\[j_{v,!}\cL_{\xi,v}\in\on{Shv}(I_0\backslash LX)\]
where $j_v:\cO_v\to LX$
is the inclusion map (see \eqref{standard-costandard}).
Since $\xi$ is a finite order character,
for a sufficiently divisible $n$ we have an isomorphism  $(F^n)^*j_{v,!}\cL_{\xi,v}\cong j_{v,!}\cL_{\xi,v}$ and hence 
 a canonical Weil structure $\Phi_{}$
 on $j_{v,!}\cL_{\xi,v}$ such that the
Frobenius eigenvalue of $\Phi_{}$ on non-zero stalks is equal to $1$
(note that $\cL_{\xi,v}$ is of rank $1$). 
We denote by $[\cL_{\xi,v}]=[(j_{v,!}\cL_{\xi,v},\Phi)]\in M$
the class of $(j_{v,!}\cL_{\xi,v},\Phi)\in\cC_{LX}$.
Then $M$ is a free $\mathbb Z[q^{\frac{1}{2}},q^{\frac{-1}{2}}]$-module with basis 
$[\cL_{\xi,v}], (\xi,v)\in\sD$ where the multiplication of $q^{\frac{1}{2}}$ on $[\cL_{\xi,v}]$
is given by 
$q^{\frac{1}{2}}[\cL_{\xi,v}]=[(j_{v,!}\cL_{\xi,v},q^{\frac{1}{2}}\Phi)]$ where $q^{\frac{1}{2}}\Phi$
is the Weil structure on $j_{v,!}\cL_{\xi,v}$ with eigenvalue 
$q^{\frac{1}{2}}$ on the non-zero stalks of $j_{v,!}\cL_{\xi,v}$. 
Note that 
$(j_{v,!}\cL_{\xi,v},q^{-n}\Phi)$ can be identified with the Tate twist 
$\cL_{\xi,v}(n)$ and hence 
we have 
$q^{-n}[\cL_{\xi,v}]=[\cL_{\xi,v}(n)]$.

We similarly define abelian category 
$\on{Shv}(I_0\backslash LG/I_0)^{\on{Weil}}=
\on{colim}_n\on{Shv}(I_0\backslash LG/I_0)^{F^n}$ 
and the full subcategory 
$\cC_{LG}\subset \on{Shv}(I_0\backslash LG/I_0)^{\on{Weil}}$.
We denote by $H=\K(\cC_{LG})$ the Grothendieck group of $\cC_{LG}$.
For any orbit $LG_w$, we have the extension by zero 
$j_{w,!}\cL_w$ of the constant local system $\cL_w$ on $j_w:LG_w\to LG$
and $H$ is a free $\mathbb Z[q^{\frac{1}{2}},q^{\frac{-1}{2}}]$-module with basis 
$[\cL_w], w\in\widetilde{\mathrm{W}}$.

The convolution map makes $H$ into 
an algebra over $\mathbb Z[q^{\frac{1}{2}},q^{\frac{-1}{2}}]$, known as the affine Hecke algebra of $G$, and 
the convolution action 
makes $M$ into an $H$-module:
\begin{equation}
	[\cL_w][\cL_{\xi,v}]
	=\sum_i(-1)^i[\sH^i(\cL_w\star\cL_{\xi,v})].
\end{equation} 
Finally, the Verdier duality functors induce involutions 
$D$ and $D_\delta$ on $H$ and $M$
given by
\[
D[\mathcal H]=\sum_i(-1)^i[\sH^i\mathbb D(\mathcal H)],\ \ \ 
D_\delta[\mathcal M]=\sum_i(-1)^i[\sH^i\mathbb D_\delta(\mathcal M)],
\]
and it follows from Proposition \ref{decomposition theory} that 
\[D_\delta([\mathcal H][\mathcal M])=D([\mathcal H])D_{\delta}([\mathcal M]).\]

\subsubsection{The affine Hecke module $M'$}
For any positive integer $n$, 
we define $\on{Perv}(I_0\backslash LX)^{F^n}$
to be the abelian category consisting of pairs
$(\cP,\Phi_n)$
where $\cP\in\on{Perv}_\delta(I_0\backslash LX)$ and 
$\Phi_n:(F^n)^*\cF\cong\cF$ is an isomorphism.
For any positive integer $m$ divisible by $n$,
we have a natural exact functor 
$\cA_{LX,n}\to\cA_{LX,m}$ sending 
$(\cP,\Phi_n)$ to $(\cP,(\Phi_n)^{m/n})$
and we define \[\on{Perv}_\delta(I_0\backslash LX)^{\on{Weil}}:=\on{colim}_{n}\on{Perv}_\delta(I_0\backslash LX)^{F^n}.\]
Thus $\on{Perv}_\delta(I_0\backslash LX)^{\on{Weil}}$ is an abelian category with objects consisting of pairs
$(\cP,\Phi)$
where $\Phi$,
to be called a Weil structure on $\cP$,
is a system of isomorphism 
$\Phi_n:(F^n)^*\cP\cong\cP$
for sufficiently divisible $n$
such that $(\Phi_n)^{m/n}=\Phi_{m}$
and two systems 
$(\cP,\Phi)$ and $(\cP,\Phi')$
 are isomorphic if 
there exists a multiple $m$ of $n_0$
and $n_0'$ such that $\Phi_m=\Phi'_m$.
Given an object $(\cF,\Phi)\in\on{Perv}_\delta(I_0\backslash LX)^{\on{Weil}}$
and a point $\gamma\in LX$, we have 
$\gamma\in (LX)^{F^n}$ for sufficiently divisible $n$
and hence  an endomorphism $\sH^i\Phi_{n}:\sH^i\cP_x\cong\sH^i\cP_x$
of the stalk $\sH^i\cP_x$ of $\sH^i\cP$ at $x$. We can define the 
Frobenius eigenvalues of $\Phi$ on $\sH^i\cP_x$,
as elements in the group $A=\overline{\mathbb Q}_\ell^\times/\mu$,
as being the $n$-th roots of the eigenvalues of $\sH^i\Phi_{n,x}$
for sufficiently divisible $n$.
We define $\cA_{LX}\subset \on{Perv}_\delta(I_0\backslash LX)^{\on{Weil}}$ to be the full subcategory consisting of $(\cP,\Phi)$ such that the eigenvalues of $\Phi$ lands in $q^{\frac{\mathbb Z}{2}}\subset A$.
We denote by $M'=\K(\cA_{LX})$ the Grothendieck group of 
$\cA_{LX}$.

For any equivariant local system 
$\cL_{\xi,v}$ on the orbit $\cO_v$
we have the IC-extension in \eqref{shifted IC}:
\[\on{IC}_{\xi,v}=(j_{v})_{!,*}\cL_{\xi,v}[\delta(v)]\in\on{Perv}_\delta(I_0\backslash LX)\]
where $j_v:\cO_v\to LX$
is the inclusion map
and $\delta(v)$ is the dimension function of $\cO_v$ 
fixed at the beginning of \S\ref{Affine Hecke modules}.
There is a 
 canonical Weil structure $\Phi$
 on $\IC_{\xi,v}$ such that the
Frobenius eigenvalue of $\sH^{-\delta( v)}\Phi_{\xi,v}$ on the  stalks of $\sH^{-\delta(v)}(\IC_{\xi,v}|_x)
\cong\cL_{\xi,v}|_x, x\in\cO_v$ is equal to $1$. 
It follows from \cite[Section 5]{BBDG}
that $(\IC_{\xi,v},\Phi)\in\cA_{LX}$
and we 
denote by $[\IC_{\xi,v}]=[(\IC_{\xi,v},\Phi_{})]\in M'$.
Then $M'$ is a free $\mathbb Z[q^{\frac{1}{2}},q^{\frac{-1}{2}}]$-module with basis 
$[\IC_{\xi,v}]$ where the multiplication of $\mathbb Z[q^{\frac{1}{2}},q^{\frac{-1}{2}}]$ on $[\IC_{\xi,v}]$
is given by 
$q^{\frac{1}{2}}[\IC_{\xi,v}]=[(\IC_{\xi,v},q^{\frac{1}{2}}\Phi]$ where $q^{\frac{1}{2}}\Phi$
is the Weil structure on $\IC_{\xi,v}$ with eigenvalue 
$q^{\frac{1}{2}}$ on the stalks of $\sH^{-\delta(v)}(\IC_{\xi,v}|_x), x\in\cO_v$. 
Note that 
$(\IC_{\xi,v},q^{-n}\Phi_{})$ can be identify with the Tate twist 
$\IC_{\xi,v}(n)$ and hence 
we have 
$q^{-n}[\IC_{\xi,v}]=[\IC_{\xi,v}(n)]$.

We have similarly defined category 
$\cA_{LG}\subset\on{Perv}(I_0\backslash LG/I_0)$ 
and denote by $H'=\K(\cA_{LG})$ the Grothendieck group of $\cA_{LG}$.
For any orbit $LG_w$, we have the IC-extension
$\IC_w$ of the constant local system $\cL_w$ on $j_w:LG_w\to LG$
and $H'$ is a free $\mathbb Z[q^{\frac{1}{2}},q^{\frac{-1}{2}}]$-module with basis 
$[\IC_w]$.

The convolution map makes $H'$ into 
an algebra over $\mathbb Z[q^{\frac{1}{2}},q^{\frac{-1}{2}}]$ and 
the convolution action 
makes $M'$ an $H'$-module:
\begin{equation}
	[\IC_w][\IC_{\xi,v}]
	=\sum_i(-1)^i[{^p}\sH_\delta^i(\IC_w\star\IC_{\xi,v})].
\end{equation} 
Finally, the Verdier duality functors induce involutions 
$D$ and $D_\delta$ on $H'$ and $M'$
given by
\[
D[\mathcal H]=\sum_i(-1)^i[{^p}\sH^i\mathbb D(\mathcal H)],\ \ \ 
D_\delta[\cP]=\sum_i(-1)^i[{^p}\sH^i_\delta\mathbb D_\delta(\cP)],
\]
and it follows from Lemma \ref{Hecke action} that 
\[D_\delta([\mathcal H][\cP])=D([\mathcal H])D_{\delta}([\cP]).\]

\subsubsection{Two bases}
Consider the map 
$h:M'\to M$
sending $[\cA]\to \sum_{i\in\mathbb Z} (-1)^i[\sH^i(\cA)]$, 
where each $\sH^i(\cA)$ is equipped with the induced Weil structure $\sH^i\Phi_n:\sH^i(F^n)^*(\cA)\cong \sH^i(\cA)$).
Since  
\[h([\IC_{\xi,v}])\in(-1)^{\delta(v)}[\cL_{\xi,v}]+\sum_{\bar u<\bar v}\mathbb Z[q^{\frac{1}{2}},q^{\frac{-1}{2}}][\cL_{\eta,u}]\]
it follows that 
$h:M'\cong M$ is bijective.
Similarly, we have a bijection
$h_G:H'\cong H$
sending $h_G([\cA])=\sum_{i\in\mathbb Z} (-1)^i[\sH^i(\cA)]$.
\begin{lem}\label{compatibility of h and D}
(i)  For any $a,b\in H'$ 
we have 
$h_G(ab)=h_G(a)h_G(b)$.
(ii) For any $a\in H'$ and $ m\in M'$,
we have 
$h(am)=h_G(a)h(m)$.
(iii)
$D_\delta\circ h=h\circ D$.

\end{lem}
\begin{proof}
The proof is the same as in the finite dimensional case \cite[Proposition 3.2.8 and 3.3.2]{MS}.
    
\end{proof}

The collection $\{h_G[\IC_w]\}_{w\in\widetilde{\mathrm W}}$
forms a new basis of $H$, known as the Kazhdan-Lusztig basis of $H$, and the transition matrix between the new basis and 
the standard basis $\{[\cL_w]\}_{w\in\widetilde{\mathrm W}}$
is given by the Affine Kazhdan-Lusztig polynomials for the extended affine Weyl group $\widetilde{\mathrm W}$.

Similarly, the  collection $\{h[\IC_{\xi,v}]\}_{(\xi,v)\in\sD}$
forms a new basis of $M$
and 
the
 transition matrix between the new basis 
and the standard basis $\{[\cL_{\xi,v}]\}_{(\xi,v)\in\sD}$ is given by 
a certain class of polynomials 
to be called the 
\emph{Affine Kazhdan-Lusztig-Vogan polynomials}.
In Section \ref{main results}, we will provide a characterization 
of   those polynomials 
and an algorithm to compute them.

\subsection{Simple reflection actions}
For each affine simple root $\alpha$, 
we provide an explicit formula for the action of the simple reflection
$[\cL_{s_\alpha}]\in H$ on $M$.

To state the results, we collect some notations from \cite{MS}.
First, note that the action of $s_\alpha$ preserves $T_0$.
For $\xi\in X(\pi_0(T_0^{\psi_0}))$,
$s_\alpha\xi\in X(\pi_0(s_\alpha\cdot T_0^{\psi_0}))$.
Next, we need a technical construction 
when $v$ is of type IIIb, IVb for $s_\alpha$.
For any torus $T$, define 
$\hat{X}(T)=X^*(T)\otimes_\bZ(\bZ_{(p)}/\bZ)$ where $\bZ_{(p)}$ is the localization at $(p)$.
Then $\hat{X}(T)$ is isomorphic to
the group of tame rank one local systems on $T$
via Kummer local systems.
For $\xi\in\hat{X}(T)$, denote the associated local system by $\cL_\xi$.
For a surjective map $\phi:T\rightarrow T'$ between tori, 
by Lemma 2.1.5 of \emph{loc. cit.} we have exact sequence
\[0\rightarrow X(\pi_0(\ker(\phi)))\rightarrow\hat{X}(T')\xrightarrow{\hat{\phi}}
\hat{X}(T)\rightarrow\hat{X}(\ker(\phi)^0)\rightarrow 0\]
where $\ker(\phi)^\circ$ is the neutral component of $\ker(\phi)$.
For examples, if $T=T_0$, $T'=T_0/T_v$,
$\xi\in X(\pi_0(T_v))$,
then $\hat{\phi}_v\xi=0$.
Thus since we restrict ourselves to 
$I_0$-equivariant local systems without weights,
we always have $\hat{\phi}_v\xi=0$.
In the case of $\cO_v$ being of type IIIb or IVb for $s_\alpha$,
let $\cO_{v'}\neq \cO_v$ be the closed orbit in $P_\alpha x_v$
as in Lemma \ref{l:orbit type}. 
From \S4.2.3, \S4.2.4 of \cite{MS} 
we see $T_0=T_{v'}\ker(\alpha)$ in type IIIb
(resp. $\mathrm{Im}(\check{\alpha})\subset T_v$ in type IVb),
so that $T_{v'}/T_{v'}\cap\ker(\alpha)\xrightarrow{\alpha}\bGm$
is an isomorphism.
In particular, $T_v\cap\ker(\alpha)$ is connected.
Also, in both two types 
$T_v\cap\ker(\alpha)=T_{v'}\cap\ker(\alpha)$.
Applying Lemma 2.1.5 of \emph{loc. cit.} to surjection
$T_0/T_{v'}\rightarrow T_0/T_{v'}\cap\ker(\alpha)$,
we get
\[
0\rightarrow\hat{X}(T_0/T_{v'})\rightarrow
\hat{X}(T_0/T_v\cap\ker(\alpha))\rightarrow
\hat{X}(\bGm)\rightarrow 0.
\]
As in \S4.2.7 of \emph{loc. cit.},
denote by $a_{v'}$ the composition
\[
a_{v'}:
\hat{X}(T_0/T_v)\rightarrow
\hat{X}(T_0/T_v\cap\ker(\alpha))\rightarrow
\hat{X}(\bGm).
\]
In type IIIb, similarly define $a_{v''}$ which will satisfy
$a_{v''}(\xi)=-a_{v'}(s\xi)$.
Since we are in the situation $\hat{\phi}_v\xi=0$,
from \S4.2.7 we always have
$a_{v'}(\xi)=-a_{v''}(\xi)$ in type IIIb
and $2a_{v'}(\xi)=0$ in type IVb.

By the same discussion as in \S4.1.3 of \emph{loc. cit.}, we have
\begin{prop}\label{construction of M}
	For $\alpha$ any simple affine root, the action of $[\cL_{s_\alpha}]\in H$
	on $M$ is given by the following formula.
\begin{equation}\label{eq:aff Hecke act}
	[\cL_{s_\alpha}][\cL_{\xi,v}]=
	\begin{cases}
		q[\cL_{\xi,v}],\hspace{5.9cm}
		\cO_v\text{ is of type I for }s_\alpha,\\
		[\cL_{s_\alpha\xi,v'}],\hspace{5.68cm}
		\cO_v\text{ is of type IIa for }s_\alpha,\\
		(q-1)[\cL_{\xi,v}]+q[\cL_{s_\alpha\xi,v'}],\hspace{2.8cm}
		\cO_v\text{ is of type IIb for }s_\alpha,\\
		[\cL_{s_\alpha\xi,v'}]+[\cL_{s_\alpha\xi,v''}],\hspace{3.73cm}
		\cO_v\text{ is of type IIIa for }s_\alpha,\\
		(q-2)[\cL_{\xi,v}]+(q-1)([\cL_{\xi,v'}]+[\cL_{\xi,v''}]),\
		\cO_v\text{ is of type IIIb for }s_\alpha,\ a_{v'}(\xi)=0,\\
		-[\cL_{\xi,v}],\hspace{5.8cm}
		\cO_v\text{ is of type IIIb for }s_\alpha,\ a_{v'}(\xi)\neq0,\\
		[\cL_{\xi,v}]+[\cL_{\xi_1,v'}]+[\cL_{\xi_2,v'}],\hspace{2.8cm}
		\cO_v\text{ is of type IVa for }s_\alpha,\\
		(q-1)[\cL_{\xi,v}]-[\cL_{\xi',v}]+(q-1)[\cL_{\bar{\xi},v'}],\hspace{0.6cm}
		\cO_v\text{ is of type IVb for }s_\alpha,\ a_{v'}(\xi)=0,\\
		-[\cL_{\xi,v}],\hspace{5.85cm}
		\cO_v\text{ is of type IVb for }s_\alpha,\ a_{v'}(\xi)\neq0.
	\end{cases}
\end{equation}
\end{prop}
In the above, 
in type IVa, $\xi_1,\xi_2$ are preimages of 
the image of $\xi$ in $\hat{X}(T_0/T_v\cap\ker(\alpha))$, 
c.f. \S4.2.6 of \emph{loc. cit.};
in type IVb where $a_{v'}(\xi)$, $\xi'$ is the element other than $\xi$
that has the same image as $\xi$ in $\hat{X}(T_0/T_v\cap\ker(\alpha))$,
and $\bar{\xi}\in\hat{X}(T_0/T_{v'})$ is the unique element
that maps to the image of $\xi$ in $\hat{X}(T_0/T_v\cap\ker(\alpha))$.

\begin{proof}
	It follows essentially from the results in \cite[\S4.1.3]{MS}.
    The key observation here is that the
    computation of simple Hecke action 
    in \emph{loc. cit.}
	  only uses information
	about the Levi subgroup of the parahoric subgroup
	and the stabilizer of the orbit.
	We demonstrate this by carrying it out 
	when $\cO_v$ is of type IIb for $s_\alpha$,
	following \S4.3.7 and \S4.3.9 of \emph{loc. cit.}.
	
	We chosen an isomorphism $u_{\pm\alpha}:k\simeq U_{\pm\alpha}$
	and defined element $n_\alpha=u_\alpha(1)u_{-\alpha}(-1)u_\alpha(1)$ as in \S4.1.1 of \emph{loc. cit.},
	 representing $s_\alpha$.
In the following, we choose orbit representatives $\tv\in\cO_v$
such that the stabilizers $H_{\alpha,\tv}$
are exactly as in Lemma \ref{l:orbit type}.
Note that they are not necessarily the same as 
the representatives $\dot{v}$ we used earlier.
In the case of $\cO_v$ being type IIb for $s_\alpha$,
	$P_\alpha\tv=\cO_v\sqcup\cO_{v'}=I_0\tv\sqcup I_0\tv'$
	in which $I_0\tv$ is open and $I_0\tv'$ has codimension one.
	Also, $n_\alpha \tv\in I_0\tv'$
	and $n_\alpha u_\alpha(x)\tv\in I_0\tv$ if $x\neq 0$.
	Here we also choose $\tv'=n_\alpha \tv$.
	Thus $I_0s_\alpha I_0\tv=P_\alpha \tv$.
	We have isomorphisms
	\begin{align*}
	I_0s_\alpha I_0\times^{I_0}I_0\tv
	&\simeq
	\{(pI_0,y)\in I_0s_\alpha I_0/I_0\times P_\alpha \tv\mid p^{-1}y\in I_0\tv\}\\
	&\simeq
	\{(x,y)\in k\times P_\alpha \tv\mid n_\alpha^{-1}u_\alpha(-x)y\in I_0\tv\},
	\end{align*}
    where the first isomorphism is given by $(p,a)\mapsto(pI_0,pa)$,
    the second is given by writing $pI_0=u_\alpha(x)n_\alpha I_0$
    and preserving $y$.
    The multiplication map $\mu$ becomes projection $(x,y)\mapsto y$.
    By \cite[Lemma 4.1.5]{MS}, 
    $U_{-\alpha}\tv\subset P_\alpha^+ \tv$.
    From the above discussions we can see 
    $n_\alpha^{-1}u_\alpha(-x)y\in I_0\tv$
    if and only if either $y\in I_0\tv'$
    or $y\in u_\alpha(k\backslash\{x\})P_\alpha^+ T_0\tv$.
    
    Now given $\cL_{\xi,v}$ on $I_0\tv$,
    denote by $\widetilde{\cL}$ 
    the local system on $I_0s_\alpha I_0\times^{I_0}I_0\tv$ 
    whose pull back to $I_0s_\alpha I_0\times I_0\tv$ is
    $\underline{\bQl}\boxtimes\cL_{\xi,v}$.
    Then
    \[
    [\cL_{s_\alpha}][\cL_{\xi,v}]=
    \sum_i(-1)^i[\sH^i(\mu_!\widetilde{\cL}|_{I_0v})]
    +\sum_i(-1)^i[\sH^i(\mu_!\widetilde{\cL}|_{I_0v'})].
    \]
    
    We first compute the first sum.
    Here $\mu^{-1}(I_0\tv)=\{(x,u_\alpha(x')P_\alpha^+T_0\tv)\mid x\neq x'\}$.
    By the definition of type IIb, we have
    $\Stab_{I_0}(\tv)\subset I_0\cap\Ad_{n_\alpha}I_0=T_0P_\alpha^+$.
    Thus we have isomorphism
    $k\times T_0P_\alpha^+\tv\simeq I_0\tv$
    given by $(x',y')\mapsto u_\alpha(x')y'$.
    Denote $z=x-x'$, then
    $\mu^{-1}(I_0\tv)\simeq k^\times\times I_0\tv$
    on which $\mu(z,y)=y$.
    Denote by $\mathrm{pr}_2$ the projection to second factor.
    Then $\widetilde{\cL}|_{\mu^{-1}(I_0\tv)}\simeq\mathrm{pr}_2^*\cL_{\xi,v}$.
    We conclude
    \begin{align*}
    \sum_i(-1)^i[\sH^i(\mu_!\widetilde{\cL}|_{I_0\tv})]
    =\sum_i(-1)^i[\sH^i(\mathrm{pr}_{2,!}\mathrm{pr}_2^*\cL_{\xi,v})]
    =\sum_i(-1)^iH_c^i(\bGm)[\cL_{\xi,v}]
    =(q-1)[\cL_{\xi,v}].
    \end{align*}

    Next we compute the second sum over $I_0\tv'$.
    Here $x\in k, y\in I_0\tv'$.
    We have $\mu^{-1}(I_0\tv')\simeq k\times I_0\tv'$.
    Denote by $\mathrm{pr}_2$ the projection to the second factor.
    Write $y=tu\tv'$ where $t\in T_0$, $u\in I_0^+$.
    Then $n_\alpha^{-1}u_\alpha(-x)y\in P_\alpha^+\Ad_{n_\alpha}(t)\tv$.
    The image of $(x,y)$ in $T_0/T_v$ is $\Ad_{n_\alpha}(t)$.
    Thus the restriction of $\widetilde{\cL}$ to
    $\mu^{-1}(I_0\tv')\simeq k\times I_0\tv'$
    is $\mathrm{pr}_2^*\cL_{s_\alpha\xi,v'}$.
    We obtain
    \[
    \sum_i(-1)^i[\sH^i(\mu_!\widetilde{\cL}|_{I_0\tv'})]
    =\sum_i(-1)^i[\sH^i(\mathrm{pr}_{2,!}\mathrm{pr}_2^*\cL_{s_\alpha\xi,v'})]
    =\sum_i(-1)^iH_c^i(\bA^1)[\cL_{s_\alpha\xi,v'}]
    =q[\cL_{s_\alpha\xi,v'}].
    \]
    This completes the computation for type IIb.
\end{proof}

To completely determine the Hecke of $H$ on $M$,
we also need to describe the action of 
$[\cL_w]$
for $w\in\Omega$.

\begin{lem}\label{l:Omega Hecke action}
    For $w\in \Omega$, 
    $[\cL_w][\cL_{\xi,v}]=[\cL_{w\xi,wv}]$,
    where $w$ acts on $\xi$ via conjugation on $\hat{X}(T_0)$
    and on $\sV$ via left multiplication.
\end{lem}
\begin{proof}
    Let $\dot{w}\in N_{LG}(LT_0)\cap N_{LG}(I_0)$  
    be a lift of $w$.
    Note that
    \[
    I_0\dot{w}I_0\times^{I_0}\cO_v=
    \dot{w}\times^{I_0}I_0x_v
    \simeq \dot{w}I_0x_v=I_0\dot{w}x_v.
    \]
    Thus the convolution action is just
    the direct image of the left multiplication by $\dot{w}$.

    The above last equality $\dot{w}I_0x_v=I_0\dot{w}x_v$
    induces a conjugation action of $\dot{w}$ on $I_0$,
    thus an action on the abelian quotient $T_0\simeq I_0/I_0^+$.
    So the action of $w$ on $\xi$ is the conjugation action 
    on a character of $T_0$ that represents $\xi$.
\end{proof}

\subsection{The polynomials $b_{\eta,u;\xi,v}$ and $c_{\eta,u;\xi,v,i}$}\label{polynomials b and c}

Consider the partial ordering on $I_0$-orbits by closure relation
and write $u<v$ if $\cO_{u}\subset\overline{\cO}_{v}$
and $u\neq v$.

According to~\eqref{D(IC)}, we have 
$\mathbb D_\delta(\IC_{\xi,v})\cong\IC_{-\xi,v}(\delta(v))$
and $\mathbb D_\delta(j_{v,!}\cL_{\xi,v})\cong j_{v,*}\cL_{-\xi,v}[2\delta(v)](\delta(v))$.
It follows that
\begin{equation}\label{eq:b coeff}
	D_\delta[\cL_{\xi,v}]=q^{-\delta(v)}[\cL_{-\xi,v}]+
	\sum_{u<v}b_{\eta,u;\xi,v}[\cL_{\eta,u}]\in M
\end{equation}
and 
\begin{equation}\label{eq:c coeff}
	[\sH^i\IC_{\xi,v}]=\delta_{i,-\delta(v)}[\cL_{\xi,v}]+
	\sum_{u<v}c_{\eta,u;\xi,v,i}[\cL_{\eta,u}]\in M
\end{equation}
where $b_{\eta,u;\xi,v},c_{\eta,u;\xi,v,i}\in\mathbb Z[q^{\frac{1}{2}},q^{-\frac{1}{2}}]$.
The polynomials $b_{\eta,u;\xi,v}$ can be viewed as an affine analog of 
the $R_{\gamma,\delta}$ in \cite{LV}
and the $c_{\eta,u;\xi,v,i}$ are related to the 
Affine Kazhdan-Lusztig-Vogan polynomials in Section \ref{AKLV}.

Apply the bijection $h:M'\cong M$ to equality
$D[\IC_{\xi,v}]=q^{-\delta(v)}[\IC_{-\xi,v}]$
and plug in the relations \eqref{eq:b coeff}, \eqref{eq:c coeff}.
By comparing coefficients of $[\cL_{\eta,u}]$ for $u<v$, we obtain
\begin{equation}\label{eq:b c coeff relation}
	\begin{split}
		&\sum_i(-1)^ic_{\eta,u;-\xi,v,i}
		-q^{\delta(v)-\delta(u)}
		\sum_i(-1)^i\bar{c}_{-\eta,u;\xi,v,i}\\
		&=(-1)^{\delta(v)}q^{\delta(v)}b_{\eta,u;\xi,v}
		+q^{\delta(v)}\sum_{u<z<v}b_{\eta,u;\zeta,z}
		\sum_i(-1)^i\bar{c}_{\zeta,z;\xi,v,i}
	\end{split}
\end{equation}
where 
$\bar{c}\in\mathbb Z[q^{\frac{1}{2}},q^{\frac{-1}{2}}]$ means inverting $q^{\frac{1}{2}}$.

\begin{lem}\label{l:c coeff}
Assume $c_{\eta,u;\xi,v,i}\in\bN q^{\frac{1}{2}(i+\delta(v))}$.
	Then the coefficients $c_{\eta,u;\xi,v,i}$'s
		can be expressed as a product $a\prod_jb_j$ 
		where $a\in\bZ q^\bZ$,
		$b_j$ or $b_j^{-1}$ is a monomial of $q^{\frac{1}{2}}$
		in some $b_{\eta,u;\xi,v}$.
	
\end{lem}
\begin{proof}
	It follows from the same argument as in \S3.4.3 of \emph{loc. cit.}
	by comparing degrees of $q$ 
	on both sides of \eqref{eq:b c coeff relation}
	and descending induction on $\delta(u)$.
\end{proof}

\subsection{An algorithm}\label{ss:algorithm}
We first describe an algorithm that computes $b_{\eta,u;\xi,v}$.
By Lemma \ref{l:c coeff} and Theorem \ref{purity} to be proved later,
this also provides a way to compute $c_{\eta,u;\xi,v,i}$
using \eqref{eq:b c coeff relation}.

Let $\cO_v$ be an $I_0$-orbit,
$\xi\in X(\pi_0(T_0))$.
If $\cO_v=\overline{\cO}_v$ is closed,
clearly $b_{\eta,u;\xi,v}=0$.
Otherwise by Lemma \ref{l:orbit type b},
there exists simple affine root $\alpha$
such that $\cO_v$ is of type b for $s_\alpha$.
In view of Corollary \ref{c:Iwahori closure and G(O) into Iwahori}.(i),
we can take induction on $\dim\cO_v$.

\subsubsection{}
Case 1: $\cO_v$ is of type IIb for some $s_\alpha$.
By Lemma \ref{l:orbit type}
$P_\alpha\cO_v=\cO_v\sqcup\cO_{v'}$ 
in which $\cO_v$ is open.
We adopt the trick in \cite[\S5.2]{MS}.
From \eqref{eq:aff Hecke act}, we have relation
\[
[\cL_{s_\alpha}][\cL_{\xi,v}]=(q-1)[\cL_{\xi,v}]+q[\cL_{s\xi,v'}].
\]
Apply $D$ to the above.
Similar as Proposition 3.3.2, Lemma 4.4.7 of \emph{loc. cit.}, we have
\[
D[\cL_{s_\alpha}]=q^{-1}([\cL_s]+(1-q)),\quad
D([\cL_{s_\alpha}][\cL_{\xi,v}])
=(D[\cL_{s_\alpha}])(D[\cL_{\xi,v}]).
\]
Note that $D$ inverts $q$.
We obtain
\[
(q^{-1}[\cL_{s_\alpha}]+q^{-1}-1)D[\cL_{\xi,v}]
=(q^{-1}-1)D[\cL_{\xi,v}]+q^{-1}D[\cL_{s\xi,v'}],
\]
i.e. $[\cL_{s_\alpha}]D[\cL_{\xi,v}]=D[\cL_{s\xi,v'}]$.
Since in the affine Hecke algebra, 
$[\cL_{s_\alpha}]^{-1}=q^{-1}[\cL_{s_\alpha}]+q^{-1}-1$,
we get
\begin{equation}\label{eq:IIb induction}
	D[\cL_{\xi,v}]=(q^{-1}[\cL_{s_\alpha}]+q^{-1}-1)D[\cL_{s\xi,v'}].
\end{equation}
Using induction hypothesis and \eqref{eq:aff Hecke act},
we can compute $b_{\eta,u;\xi,v}$ for any $u<v$.

\subsubsection{}
Case 2: $\cO_v$ is not of type IIb for any $s_\alpha$.
We adopt the trick in \S7.2 of \emph{loc. cit.}.
Let $J$ be the set of simple affine roots such that
$\cO_v$ is of type IIIb or IVb for $s_\alpha$, $\alpha\in J$.
By Lemma \ref{l:orbit type by psi},
$\psi_v\alpha=-\alpha$ for any $\alpha\in J$ and
$\psi_v\alpha>0$ for any $\alpha\in\Delta_\aff-J$.
Note that the action of $\psi_v$ on affine roots 
is the composition of $\theta$ with $w_v$,
and $\theta$ preserves $\Delta_\aff$ by assumption.
Then $J$ is a proper subset of simple affine roots,
since otherwise $w_v$ would map $\Delta_\aff$
to $-\Delta_\aff$, which is impossible.
Let $P=P_J$ be the standard parahoric subgroup 
associated to $J$,
with pro-unipotent radical $P^+$ and Levi subgroup $L_P$. 
Note that $P, P^+, L_P$ are all $\psi_v$-stable.
The set $J$ gives the simple roots of $L_P$.

\begin{lem}
	We have $Px_v=\overline{\cO}_v$.
\end{lem}
\begin{proof}
On the one hand, for any $\alpha\in J$ 
since $\cO_v$ is of type b for $s_\alpha$,
$\cO_v$ is open in $P_\alpha x_v$.
Thus $\overline{P_\alpha x_v}=\overline{\cO}_v$
is stable under $P_\alpha$-action.
Therefore $\overline{\cO}_v$ is stable under $P$-action.
We get $Px_v\subset\overline{\cO}_v$.

On the other hand, $P$ is stable under $\psi_v$.
By Lemma \ref{l:tau(P) image} and the same proof of
Lemma \ref{l:orbit type b} for Iwahori subgroup,
$\tau(P)=\tau_v(P)n_v=P^{\on{inv}\circ\psi_v,\circ}n_v$
is closed in $P$ and further closed in $LG$.
Thus $\tau(Px_v)$ is closed in $\tau(LX)$
and $Px_v$ is closed in $LX$.
\end{proof}

In view of the above, we can define quotient map
\[
\pi:\overline{\cO}_v=Px_v
\simeq P/P^{\psi_v}\rightarrow L_P/L_P^{\psi_v}.
\]
We can see $\pi$ is smooth with fibers isomorphic to $P^+/(P^+)^{\psi_v}$.
It induces a bijection between $I_0$-orbits on $\overline{\cO}_v$
to the orbits of $B_{L_P}:=I_0\cap L_P$ on $X_{L_P}:=L_P/L_P^{\psi_v}$. 
Observe that $X_{L_P}$ is a finite dimensional symmetric variety,
whose set of $B_{L_P}$-orbits are parametrized by 
$\{\pi u\mid u\leq v\}$.
We denote $\cO_{\pi u}$ the $B_{L_P}$-orbit corresponding to 
$\pi u$ and $\delta(\pi u)=\dim\cO_{\pi u}$.
We have the following reduction to the finite dimensional case:
\begin{lem}\label{reduction to finite case}
	\begin{itemize}
		\item [(i)]
		$\pi^*\cL_{\eta,\pi u}=\cL_{\eta,u}$,\ $u\leq v$, 
            $\eta\in\hat{X}(T_0/T_u)$,
		
		\item [(ii)]
		$\IC_{\xi,v}\cong\pi^*\IC_{\xi,\pi v}[\delta(v)-\delta(\pi v)]$,
		
		\item [(iii)]
		$\mathbb D_\delta (j_{v,!}\cL_{\xi,v})\cong\pi^*(\mathbb D(j_{\pi v,!}\cL_{\xi,\pi v}))\langle2\delta(v)-2\delta(\pi v)\rangle$.
	\end{itemize}
\end{lem}
\begin{proof}
Part (i) follows from  $T_u=T_0^{\psi_v}=T_{\pi u}$
and Part (iii) follows from~\eqref{D(IC)}.
Since we are using the $*$-adapted perverse $t$-structures on placid schemes (thus, $*$-pull back along smooth maps are $t$-exact), we have 
$(j_{v})_{!*}\cL_{\xi,v}\cong\pi^*((j_{\pi v})_{!*}\cL_{\xi,\pi v})$
and it follows that 
\[\IC_{\xi,v}=(j_{v})_{!*}\cL_{\xi,v}[\delta(v)]\cong\pi^*((j_{\pi v})_{!*}\cL_{\xi,\pi v}[\delta(\pi v)])[\delta(v)-\delta(\pi v)]=\pi^*\IC_{\xi,\pi v}[\delta(v)-\delta(\pi v)].\]
Part (ii) follows.
\end{proof}

In view of the above, 
the computation of $b_{\eta,u;\xi,v},c_{\eta,u;\xi,v,i}$
can be reduced to the finite dimensional symmetric variety $X_{L_P}$,
which has been done in \cite{MS}.
As a corollary of the algorithm we obtain 
the following property of $b_{\eta,u;\xi,v}$:
\begin{lem}\label{degree bound}
	$q^{\delta(v)}b_{\eta,u;\xi,v}$
        is a polynomial in $q$ with coefficients in $\mathbb Z$ of degree at most $\delta(v)-\delta(u)$.
\end{lem}
\begin{proof}
We do induction on $\delta(v)$.
	When $u\geq v$, $b_{\eta,u;\xi,v}=0$.
	Suppose the statement holds for all $b_{\eta',u';\xi',v'}$
	with $\delta(v')<\delta(v)$.
	
	When $\cO_v$ is of type IIb for $s_\alpha$,
	by Lemma \ref{l:orbit type},
	we have 
	$\delta(v)-\delta(v')=\on{codim}_{\cO_{v}/\cO_{v'}}=1$.
	Then the desired claim follows from \eqref{eq:IIb induction}, the induction procedure, and  the observation that 
	all the coefficients appearing in the affine Hecke action
	as in the list \eqref{eq:aff Hecke act}
	are polynomials in $q$ of degree less or equal to $1$.
	
	When $\cO_v$ is not of type IIb for any $s_\alpha$,
	we have seen before that the problem can be reduced to
	a symmetric variety of a generalized Levi subgroup.
	Then the statement in 
    \cite[Theorem 1.10]{LV}
	implies that $q^{\delta(\pi v)}b_{\eta,\pi u;\xi,\pi v}\in\mathbb Z[q]$ is of degree at most 
    $\delta(\pi v)-\delta(\pi u)$.
    According to \ref{dimension theory}
    we have \[\delta(\pi v)-\delta(\pi u)=\on{codim}_{\overline\cO_{\pi u}/\overline\cO_{\pi v}}=
    \on{codim}_{\overline\cO_u/\overline\cO_v}=
    \delta(v)-\delta(u)\] 
    and together with Lemma \ref{reduction to finite case} (iii) we obtain 
	\[q^{\delta(v)}b_{\eta, u;\xi,v}=q^{\delta(\pi v)}b_{\eta,\pi u;\xi,\pi v}\in\mathbb Z[q]\]
    is of degree at most $\delta(v)-\delta(u)=\delta(\pi v)-\delta(\pi u)$.

\end{proof}

\subsubsection{Applications to spherical orbits}\label{sss:G(O) C coeff}
Note that the above results for Iwahoric orbits 
also provide an algorithm to compute  
 the multiplicities for $L^+G$-equivariant $\IC$-complexes.
Precisely,
let $\cL_{\chi,\lambda}$ be a $L^+G$-equivariant local system 
on $L^+G$-orbit $LX_\lambda$,
where $\chi$ is a representation of the component group of stabilizers
in $L^+G$ on $LX_\lambda$.
Consider the associated $\IC$-complex
\[\IC_{\chi,\lambda}:=(j_{\lambda})_{!*}\cL_{\chi,\lambda}[\delta(\lambda)]\in\on{Perv}_\delta(L^+G\backslash LX)\]
We can write down similar formula as \eqref{eq:c coeff}
in the Grothendieck group of
mixed $L^+G$-constructible sheaves:
\begin{equation}\label{eq:C coeff}
	[\sH^i\IC_{\chi,\lambda}]=\delta_{i,-\delta(\lambda)}[\cL_{\chi,\lambda}]+
	\sum_{\mu<\lambda}C_{\chi',\mu;\chi,\lambda,i}[\cL_{\chi',\mu}]
\end{equation}
where $C_{\chi',\mu;\chi,\lambda,i}\in\mathbb Z[q^{\frac{1}{2}},q^{\frac{-1}{2}}]$,
and each $\cL_{\chi',\mu}$ is a $L^+G$-equivariant local system on $LX_\mu$.

Let $h$ be the map from the Grothendieck group of 
mixed $L^+G$-constructible sheaves
into that of mixed $I_0$-constructible sheaves and
denote $h[\cL_{\chi',\mu}]=\sum_{(\eta,u)\subset(\chi',\mu)}[\cL_{\eta,u}]$.
Denote by $\cO_{v_\lambda}$ be unique open $I_0$-orbit in $LX_\lambda$
and let $\cL_{\xi_\chi,v_\lambda}=\cL_{\chi,\lambda}|_{\cO_{v_\lambda}}$.
Then 
\begin{equation}\label{IC=IC}
   \IC_{\chi,\lambda}=\IC_{\xi_\chi,v_\lambda}. 
\end{equation}
By comparing \eqref{eq:c coeff}
with the image under $h$ of \eqref{eq:C coeff},
we see the left hand sides are the same.
The matching of coefficients on the right hand sides provides relation
\begin{equation}\label{eq:c v.s. C}
	c_{\eta,u;\xi_\chi,v_\lambda,i}=
	\begin{cases}
		\delta_{i,-\delta(\lambda)},\hspace{1.2cm}(\eta,u)\subset(\chi,\lambda),\\
		C_{\chi',\mu;\chi,\lambda,i},\qquad (\eta,u)\subset(\chi',\mu),\ \mu<\lambda,\\
		0,\hspace{2.3cm}\text{otherwise}.
	\end{cases}
\end{equation}

From the above, we obtain 
\begin{lem}
	$c_{\eta,u;\xi_\chi,v_\lambda,i}$ depends only on the 
	$L^+G$-orbit of $u$.
\end{lem}

In view of the above,
to compute
$C_{\chi',\mu;\chi,\lambda,i}$,
it suffices to take the open $I_0$-orbit $I_0v_\lambda\subset LX_\lambda$
and any $I_0$-orbit $I_0u\subset LX_\mu$,
which uniquely determines $(\xi_\chi,v_\lambda)\subset(\chi,\lambda)$
and $(\eta,u)\subset(\chi',\mu)$,
then compute $c_{\eta,u;\xi_\chi,v_\lambda,i}$.

\section{Parity vanishing and Poincar\'e polynomials of IC-complexes
}\label{main results}
In this section, we prove the parity vanishing of 
equivariant $\IC$-complexes
on the spherical and Iwahori orbits.
We will first prove for Iwahori orbits
following the strategy of Mars-Springer \cite{MS} 
in the finite dimensional situation,
then deduce the result for spherical orbits.
Along the way, we provide an algorithm 
that computes the Poincar\'e polynomials of 
Iwahori-equivariant $\IC$-complexes
in an inductive procedure.
We keep the same notations and assumptions as in \S\ref{s:Iwahori orbits}.

\subsection{Purity and parity vanishing}
Let $\IC_{\xi,v}:=(j_{v})_{!*}\cL_{\xi,v}[\delta(v)]$ 
(resp. $\IC_{\chi,\lambda}:=(j_{\lambda})_{!*}\cL_{\chi,\lambda}[\delta(\lambda)]$)
be the IC-complex
associated to the $I_0$-equivariant local system $\cL_{\chi,v}$ on $\cO_v$ (resp. 
 $L^+G$-equivariant local system $\cL_{\chi,\lambda}$ on $LX_\lambda$).
According to Section~\ref{Affine LV modules},
each $\IC_{\xi,v}$ has a canonical Weil structure $\Phi=\{\Phi_n,\ \text{for for $n$ divisible by $n_0$}\}$, 
where $\Phi_n:(F^n)^*\IC_{\xi,v}\cong\IC_{\xi,v}$
are systems of isomorphisms  
satisfying $(\Phi_n)^m=\Phi_{nm}$. 
The  discussion in \emph{loc. cit.} also implies that $\IC_{\chi,\lambda}$  carries a canonical Weil structure.

We recall the notion of pointwise purity (see  Definition \ref{pointwise pure complexes} for a more general discussion).
A perverse sheaf $(\cF,\Phi)\in\on{Perve}_\delta(I_0\backslash LX)^{{\on{Weil}}}$ 
is called $*$-pointwise pure (resp. $!$-pointwise pure) of weight $w$ if for any $I_0$-orbit $j_v:\cO_v\subset LX$
and a point $x\in(\cO_v)^{F^n}$, the eigenvalues of $\Phi^n$ on the stalk $\sH^m(j_v^*\cF)_x$ (resp. $\sH^m(j_v^!\cF)_x$)
are algebraic numbers in $\overline{\mathbb Q}_\ell^\times$ all of whose 
complex conjugates have absolute value 
$(q^n)^{\frac{w+m}{2}}$. It is called pointwise pure if it is both $*$ and $!$-pointwise pure.
We have similar notions of pointwise purity for $\on{Perve}_\delta(G(\cO)\backslash LX)^{{\on{Weil}}}$.

We have the following key pointwise purity result.
\begin{thm}\label{purity}
The complex $\IC_{\xi,v}$ (resp. $\IC_{\chi,\lambda}$) is pointwise pure of weight $\delta(v)$ (resp. $\delta(\lambda)$). 
\end{thm}
\begin{proof}
Note that the local system $\cL_{\xi,v}$ is pointwise pure of weight $0$ in the sense of Definition \ref{pointwise pure complexes} (2)
and hence by Example \ref{example of pure complex} the IC-complex $\IC_{\xi,v}=(j_{v})_{!*}\cL_{\xi,v}[\delta(v)]$
is pure of weight $\delta(v)$ in the sense of 
Definition \ref{pointwise pure complexes} (i).
Then the existence of contracting slices in 
Proposition \ref{slice spherical} and Proposition \ref{Iwahori Slice}
together with the criterion in 
Proposition \ref{purity criterion} imply $\IC_{\xi,v}$ is 
$*$-pointwise pure of weight  $\delta(v)$.

On the other hand, Lemma \ref{* to !} implies 
$\mathbb D_\delta(\IC_{-\xi,v})$ is $!$-pointwise pure of weight $-\delta(v)$
and 
the isomorphism $\mathbb D_\delta(\IC_{-\xi,v})(-\delta(v))\cong \IC_{\xi,v}$
implies   $\IC_{\xi,v}$ is $!$-pointwise pure of $-\delta(v)+2\delta(v)=\delta(v)$.

The same proof works for the case  $\IC_{\chi,\lambda}$.
\end{proof}

We are ready to prove the parity vanishing of $\IC$-complexes.
\begin{thm}\label{t:IC parity vanishing}
	For any $(\xi,v), (\eta,u)\in\sD$, we have 
	\begin{itemize}
		\item [(i)] 
		$c_{\eta,u;\xi,v,i}\in\bN q^{\frac{1}{2}(i+\delta(v))}$.
		Moreover, $c_{\eta,u;\xi,v,i}=0$ 
		if $i+\delta(v)$ is odd.
		
		\item [(ii)]
		$\sH^i\IC_{\xi,v}=0$ if $i+\delta(v)$ is odd.

            \item[(iii)]
            $\sH^i\IC_{\chi,\lambda}=0$ if $i+\delta(\lambda)$ is odd.	
	\end{itemize}
\end{thm}
\begin{proof}
	(i):
	The first claim follows from the $*$-pointwise purity of $\IC_{\xi,v}$ in Theorem \ref{purity}.
	The second vanishing property follows from 
	Lemma \ref{l:c coeff} and Lemma \ref{degree bound}

	(ii):
	This follows from the definition \eqref{eq:c coeff}.

       (iii):
       It follows from (ii) and~\eqref{IC=IC}.	
\end{proof}

\begin{rem}\label{relation with LV}
	The $I_0$-orbits on the closed orbit $L^+G/L^+K\subset LX$
	correspond one-to-one to 
	the $B_0$-orbits on $X=G/K$.
	Thus Theorem \ref{t:IC parity vanishing} is a generalization of
	results of Lusztig-Vogan \cite{LV} and Mars-Springer \cite{MS}.
\end{rem}

\subsection{Affine Kazhdan-Lusztig-Vogan polynomials}\label{AKLV}
For $I_0$-equivariant local systems 
$\cL_{\eta,u}$ and $\cL_{\xi,v}$,
define Poincar\'e polynomial
\begin{equation}\label{eq:Iwahori Poincare poly}	P_{\eta,u;\xi,v}:=\sum_i[\cL_{\eta,u},\sH^i(j_{!*}\cL_{\xi,v})|_{\cO_{\bar{u}}}]q^{\frac{i}{2}}.
\end{equation}
By Theorem \ref{t:IC parity vanishing}.(i), 
$[\cL_{\eta,u},\sH^i(j_{!*}\cL_{\xi,v})|_{\cO_{\bar{u}}}]q^{\frac{i}{2}}
=c_{\eta,u;\xi,v,i-\delta(v)}$.
Thus
\[
P_{\eta,u;\xi,v}=\sum_ic_{\eta,u;\xi,v,i-\delta(v)}.
\]
Here we take convention that 
$c_{\xi,v;\xi,v,i}=\delta_{i,-\delta(v)}$, 
and $c_{\eta,u;\xi,v,i}=0$ if
$u$ is not in the closure of $\cO_v$
or $u$ is in the closure of $\cO_v$
but $\cL_{\eta,u}$ is not a subquotient of $\cL_{\xi,v}|_{\cO_{\bar{u}}}$.
Then
\[
h[j_{!*}\cL_{\xi,v}]=\sum_{\eta,u}P_{\eta,u;\xi,v}[\cL_{\eta,u}].
\]

\begin{thm}\label{t:I_0 Poincare poly uniqueness}
	For any pair $(\eta,u),(\xi,v)$, 
	$P_{\eta,u;\xi,v}$ is a polynomial in $q$ 
	with non-negative integer coefficients.
	It is the unique family of polynomials in $q^{\frac{1}{2}}$ 
	satisfying the following conditions:
	\begin{itemize}
		\item [(i)] $P_{\xi,v;\xi,v}=1$.
		\item [(ii)] If $\cO_u\neq\cO_v$, 
		$\deg_q(P_{\eta,u;\xi,v})\leq\frac{1}{2}(\delta(v)-\delta(u)-1)$.
		\item [(iii)] $C_{\xi,v}:=\sum_{\eta,u}P_{\eta,u;\xi,v}[\cL_{\eta,u}]$
		satisfies $D_\delta C_{\xi,v}=q^{-\delta(v)}C_{-\xi,v}$.
	\end{itemize}
\end{thm}
\begin{rem}
    Note that $\delta(v)-\delta(u)$
    is equal to the codimension of $\overline\cO_u\subset\overline\cO_v$ which is independent of the choice of the 
    dimension theory $\delta$.
\end{rem}
\begin{proof}
	By Theorem \ref{t:IC parity vanishing},
	we know $P_{\eta,u;\xi,v}$ satisfies all the above except (ii).
	By the property of intermediate extension,
	$c_{\eta,u;\xi,v,i-\delta(v)}\neq0$ only if $\cO_u\subset\overline{\cO}_v$
	and $\delta(u)<-(i-\delta(v)$,
	so that $\frac{1}{2}i\leq\frac{1}{2}(\delta(v)-\delta(u)-1)$.
	Then (ii) follows from Theorem \ref{t:IC parity vanishing}.(i).
	
	Conversely, suppose $P_{\eta,u;\xi,v}$ is a family of polynomials
	in $q^{\frac{1}{2}}$ satisfying (i)-(iii).
	Denote by $\overline{P_{\eta,u;\xi,v}}$ 
	the polynomial in $q^{\frac{-1}{2}}$
	obtained from $P_{\eta,u;\xi,v}$
	via $q^{\frac{1}{2}}\mapsto q^{\frac{-1}{2}}$.
	By the same discussion as \eqref{eq:b c coeff relation},
	we obtain 
	\begin{equation}\label{eq:I_0 Poincare poly recurrence}
		P_{\eta,u;-\xi,v}-q^{\delta(v)-\delta(u)}
		\overline{P_{-\eta,u;\xi,v}}
		=
		q^{\delta(v)}\sum_{u<z<v}b_{\eta,u;\zeta,z}\overline{P_{\zeta,z;\xi,v}}.
	\end{equation}
	By condition (ii),
	the degrees of $q$ in $P_{\eta,u;-\xi,v}$ 
	range from $0$ to $\frac{1}{2}(\delta(v)-\delta(u)-1)$,
	while the degrees of $q$ in 
	$q^{\delta(v)-\delta(u)}
	\overline{P_{-\eta,u;\xi,v}}$
	range from $\frac{1}{2}(\delta(v)-\delta(u)+1)$
	to $\delta(v)-\delta(u)$.
	Thus \eqref{eq:I_0 Poincare poly recurrence} 
	determines $P_{\eta,u;\xi,v}$
	by doing induction on  $\delta(v)-\delta(u)$
	with base case given by condition (i).
\end{proof}

\subsection{Relative 
Kostka-Foulkes polynomials}\label{Def of KF poly}
We can deduce similar results for $L^+G$-orbits from the above.
We resume the discussion in \S\ref{sss:G(O) C coeff}.
For $L^+G$-equivariant local systems $\cL_{\chi,\lambda}$,
we defined coefficients $C_{\chi',\mu;\chi,\lambda,i}$ parallel to $c_{\eta,u;\xi,v,i}$.
Recall we denoted the restriction of $\cL_{\chi,\lambda}$
to the open $I_0$-orbit in $LX_\lambda$
by $\cL_{\xi_\chi,v_\lambda}$.

Similarly, we define $B_{\chi',\mu;\chi,\lambda}\in\mathbb Z[q^{\frac{1}{2}},q^{\frac{-1}{2}}]$ by
\[
D_\delta[\cL_{\chi,\lambda}]
=q^{-\delta(\lambda)}[\cL_{-\chi,\lambda}]
+\sum_{\mu<\lambda}B_{\chi',\mu;\chi,\lambda}[\cL_{\chi',\mu}].
\]
We denote by $(\xi,v)\subset(\chi,\lambda)$
if $\cO_v\subset LX_\lambda$ and $\cL_{\xi,v}=\cL_{\chi,\lambda}|_{\cO_v}$.
By breaking $L^+G$-equivariant local systems into $I_0$-equivariant ones
in the Grothendieck group
and comparing with definition of $b_{\eta,u;\xi,v}$,
we obtain relation
\begin{equation}
	B_{\chi',\mu;\chi,\lambda}
	=\sum_{(\eta,u)\leq(\xi,v)\subset(\chi,\mu)}b_{\eta,u;\xi,v},
	\quad\quad\forall\ (\eta,u)\subset(\chi',\mu),
\end{equation}
where we take convention $b_{\xi,v;\xi,v}=q^{-\delta(v)}$
and $B_{\chi,\lambda;\chi,\lambda}=q^{-\delta(\lambda)}$.
From this we see 
$B_{\chi',\mu;\chi,\lambda}$ are integer coefficients
Laurent polynomials in $q$
with a pole of order at most $\delta(\lambda)$.

We define $P_{\chi',\mu;\chi,\lambda}$ similarly as $P_{\eta,u;\xi,v}$.
Parallel to \eqref{eq:I_0 Poincare poly recurrence}, 
we have recurrence relation
\begin{equation}\label{eq:G(O) Poincare poly recurrence}
	P_{\chi',\mu;-\chi,\lambda}-q^{\delta(\lambda)-\delta(\mu)}
	\overline{P_{-\chi',\mu;\chi,\lambda}}
=q^{\delta(\lambda)}\sum_{\mu<\gamma<\lambda}B_{\chi',\mu;\chi'',\gamma}\overline{P_{\chi'',\gamma;\chi,\lambda}}.
\end{equation} 

By the same argument as the Iwahori case, we obtain:
\begin{thm}\label{t:G(O) Poincare poly uniqueness}
	For any pair $(\chi',\mu),(\chi,\lambda)$, 
	$P_{\chi',\mu;\chi,\lambda}$ is a polynomial in $q$ 
	with non-negative integer coefficients.
	It is the unique family of polynomials in $q^{\frac{1}{2}}$ 
	satisfying the following conditions:
	\begin{itemize}
		\item [(i)] $P_{\chi,\lambda;\chi,\lambda}=1$.
		\item [(ii)] If $\mu\neq\lambda$, 
		$\deg_q(P_{\chi',\mu;\chi,\lambda})\leq\frac{1}{2}(\delta(\lambda)-\delta(\mu)-1)$.
		\item [(iii)] $C_{\chi,\lambda}:=\sum_{\chi',\mu}P_{\chi',\mu;\chi,\lambda}[\cL_{\chi',\mu}]$
		satisfies $D_\delta C_{\chi,\lambda}=q^{-\delta(\lambda)}C_{-\chi,\lambda}$.
	\end{itemize}
\end{thm}

\section{Applications}\label{Formality}
We discuss applications of our main results to 
relative Langlands duality \cite{BZSV}.

\subsection{Semisimplicity
criterion}\label{Semi cri}
In the group case, an important application of the 
parity vanishing of $\IC$-complexes is the 
semisimplicity of the Satake category
$\on{Perv}(L^+G\backslash\Gr)$.
Using the parity vanishing in Theorem \ref{t:IC parity vanishing} and \cite{CN3}
,  we prove the following 
semisimplicity
criterion and Langlands dual description
for the relative Satake category 
$\on{Perv}_{}(L^+G\backslash LX)$:

\begin{thm}\label{semi}
   Assume the codimensions of $L^+G$-orbits in the same connected component of $LX$ are even.
\begin{itemize}
    \item [(i)] The relative Satake category  $\on{Perv}_{\delta}(L^+G\backslash LX)$ is semisimple.
\item [(ii)] Assume further that the the 
$L^+G$-stabilizers on $LX$ are connected.
Then there is an equivalence of abelian categories
\[\on{Perv}_{\delta}(L^+G\backslash LX)\cong\on{Rep}(\check G_X)\]
where the $\on{Rep}(\check G_X)$
is the category of finite dimensional
complex representations of the dual group 
$\check G_X$ of X \cite{GN}.
\end{itemize}

\end{thm}
\begin{proof}
(i) The standard proof in the group case,  
   see, e.g., \cite[Proposition 5.1.1]{ZhuIntroGr},
   only uses the parity vanishing of $\IC$-complexes 
   and parity of the dimension of spherical orbits in each connected component, thus is applicable in our setting.

(ii)
The assumption implies 
$\on{Perv}_\delta(L^+G\backslash LX)$
is a semisimple abelian category 
whose irreducible objects are 
$\IC$-complexes with trivial local systems.
\cite[Corollary 13.8.]{CN3} implies that 
$\on{Perv}(L^+G\backslash LX)$
contains a full subcategory 
$\on{Rep}(\check G_X)$
consisting of $\IC$-complexes 
with trivial coefficients
supported on the orbit closures
$\overline{LX}_\lambda$ with 
$\lambda\in\Lambda_S^+$
of the form
$-\theta(\mu)+\mu$, $\mu\in\Lambda_T$.
Now the desired claim follows from the Lemma \ref{support} below.
\end{proof}

\begin{exam}\label{new examples}
Assume $LX$ is connected (for example, when $G$ is simply connected).
According to the codimension formula in Proposition \ref{codim formula},
the evenness assumption in the theorem is satisfied if and only if 
for $LX_\lambda\subset\overline{LX}_\mu$
we have 
$\on{codim}_{\overline{LX_\lambda}}(\overline{LX}_\mu)=\delta(\mu)-\delta(\lambda)=\langle\rho,\mu-\lambda\rangle\in 2\mathbb Z$. 
Using this formula,
one can check that when
$X=\SL_{2n}/\Sp_{2n},\mathrm{Spin}_{2n}/\on{Spin}_{2n-1},\mathrm{E}_6/\mathrm{F}_4$ with
relative dual groups $\check G_X=\PGL_{2n},\PGL_2$, $\PGL_3$ (the splitting rank cases), both the evenness and connectedness assumptions
are satisfied and 
we obtain the following new instances of 
abelian relative Satake equivalence:
\begin{enumerate}
    \item $\on{Perv}(L^+\SL_{n}\backslash L(\SL_{n}/\Sp_{n}))\cong\on{Rep}(\PGL_n)$,
    \item $\on{Perv}(L^+\mathrm{Spin}_{2n}\backslash L(\mathrm{Spin}_{2n}/\on{Spin}_{2n-1}))\cong\on{Rep}(\PGL_2)$,
    \item $\on{Perv}(L^+\mathrm{E}_6\backslash L(\mathrm{E}_6/\mathrm{F}_4))\cong\on{Rep}(\PGL_3)$.
\end{enumerate}

\end{exam}

\begin{rem}\label{SO_2}

    The assumption in the theorem is necessary.
    For example, consider the complex symmetric variety 
    $X=\SL_2/\SO_2$. 
    Using the codimension formula above one can
    check that there are orbits with codimension equal to one. 
   On the other hand, from \cite[theorem 1.8]{BAF}, the category 
    $\on{Perv}_\delta(L^+\SL_2\backslash LX)$
    has a full subcategory equivalent to the 
    abelian category 
    $\wedge^\bullet T^*\mathbb C^2\on{-mod}^{\SL_2}$
    of finite dimensional $\SL_2$-equivariant modules
    over the exterior algebra $\wedge^\bullet T^*\mathbb C^2$. In particular, $\on{Perv}(L^+\SL_2\backslash LX)$ is not semisimple.
Another way to see this is to use the real-symmetric equivalence 
\[\on{Perv}_\delta(L^+\SL_2\backslash LX)\cong\on{Perv}(L^+\SL_2(\mathbb R)\backslash\Gr_{\SL_2(\mathbb R)})\]
in \cite{CN3}, where the right hand side is the real Satake category for $\SL_2(\mathbb R)$.
There is a real spherical orbit closure in 
$\Gr_{\SL_2(\mathbb R)}$ homeomrphic to
the real two dimensional pinched torus
and the 
extension by zero $j_!(\mathbb C[1])$
along the open orbit
provides 
a non-semisimple object
in $\on{Perv}(L^+\SL_2(\mathbb R)\backslash\Gr_{\SL_2(\mathbb R)})$.

\end{rem}

\begin{lem}\label{support}
    Assume the codimensions of $L^+G$-orbits in the same connected component of $LX$ are even.
    Then every element $\Lambda_S$
    is of the form
    $-\theta(\mu)+\mu$ for some $\mu\in\Lambda_T$.
\end{lem}
\begin{proof}
    From diagram~\eqref{component}, we can write 
    $\lambda=(-\theta(\mu_1)+\mu_1)+\lambda'$
    where $\lambda'\in\Lambda_S\cap Q$ (recall we assume $K$ is connected).
    Let $\Delta_L$ be the set of  simple coroots of the Levi $L\subset G$ that centralizes $2\check\rho_M$.
    According to \cite[Lemma 2.1.1]{NadlerRealGr}, we have $-\theta(\Delta_L)=\Delta_L$
    and $\lambda'=\sum_{\alpha\in\Delta_L} n_\alpha\alpha$. 
    Note that we have 
    $-\theta(\alpha)\neq\alpha$ for any $\alpha\in\Delta_L$, 
    otherwise 
    $\alpha\in\Lambda_S^+\cap Q$
    and Proposition \ref{codim formula}  would imply 
    \[\on{codim}_{\overline{LX_\alpha}}(\overline{LX}_0)=\langle\rho,\alpha\rangle=1\]
    contracditing the evenness assumption (note that $LX_\alpha$ and $LX_0$ are in the same component).
    Thus
     we can write 
    $\lambda'=\sum n_{\alpha}(-\theta(\alpha)+\alpha)$
    where $\alpha$ runs through representatives of $(-\theta)$-orbits on $\Delta_L$. The lemma follows.
\end{proof}
\begin{rem}
    The converse of Lemma \ref{support} is not true.
    For example when $X=\GL_{n}/\GL_p\times\GL_q$
    with $n=p+q$ even, every element $\Lambda_S^+$ is of the form
    $-\theta(\mu)+\mu$ but there are orbits with 
    odd codimension $n-1$.
\end{rem}

\subsection{Formality of dg extension algebras}\label{Ext algebras}
Let $D^{}(L^+G\backslash\Gr)$ be the derived Satake category for 
$G$ and 
let $D^{}(L^+G\backslash LX)$ be the 
derived relative Satake category of 
dg derived category of 
$L^+G$-equivariant complexes on $LX$.
Here we use the $*$-sheaves theory,
see, Appendix \ref{s:appendix}.
We have the Hecke action, denoted by $\star$, of the  Satake category  $D^{}(L^+G\backslash LG/L^+G)$ on $D^{}(L^+G\backslash LX)$ defined similary as in Section \ref{category of sheaves} (replacing $I_0$ by $L^+G$).
We have the monoidal abelian Satake equivalence 
\[\on{Rep}(\check G)\cong\on{Perv}^{}(L^+G\backslash\Gr_G):\ \ \ V\to \ \IC_V.\]
By restricting the action to $\on{Perv}^{}(L^+G\backslash\Gr_G)\cong\on{Rep}(\check G)$, we obtain 
a monoidal action of $\on{Rep}(\check G)$ on $D^{}(L^+G\backslash LX)$.
Let $\mathrm{e}_{L^+X}\in D^{}(L^+G\backslash LX)$ be the constant sheaf on 
the closed $L^+G$-orbit
$LX_0=L^+X$.
Let $\IC_{\cO(\check G)}$
(an ind-object in $\on{Perv}^{}(L^+G\backslash\Gr_G)$)
be the image of the regular representation 
$\cO(\check G)$
under the abelian Satake equivalence.
Since $\cO(\check G)$ is a ring object in $\on{Rep}(\check G)$, the 
RHom space 
\begin{equation}\label{def of A_X}
A_X:=R\Hom_{D(L^+G\backslash LX)}(\mathrm{e}_{L^+X},\IC_{\cO(\check G)}\star\mathrm{e}_{L^+X})
\end{equation}
is naturally a dg-algebra, known as the de-equivariantized Ext algebra.
The formality conjecture in relative Langlands duality is the following assertion 
for $X$ a spherical variety(see, e.g., \cite[Conjecture 8.1.8]{BZSV}):
\begin{conj}\label{conj}
	The dg-algebra $A_X$
	is formal, that is, it is quasi-isomorphic to the 
	graded algebra 
	$A_X\cong\on{Ext}^\bullet_{D(L^+G\backslash LX)}(\mathrm{e}_{L^+X},\IC_{\cO(\check G)}\star\mathrm{e}_{L^+X})$
	with trivial differential.
\end{conj}

\begin{thm}\label{t:formality}
	For $X$ a symmetric variety, Conjecture \ref{conj} holds true.
\end{thm}
\begin{proof}
	The same proof as in Proposition \ref{decomposition theory} 
        shows that 
	$\IC_{\cO(\check G)}\star\mathrm{e}_{L^+X}$
        is a pure complex of weight zero isomorphic to
	a direct sum of  $\IC$-complexes $\IC_{\chi,\lambda} 
        (\frac{\delta(\lambda)}{2})$ shifted by $[2n](n)$.
	By Theorem \ref{purity}, the $\IC$-complexes  $\IC_{\chi,\lambda} 
        (\frac{\delta(\lambda)}{2})$ are pointwise pure of weight zero 
	and it follows that the
	Ext group 
	$\on{Ext}^i_{D(L^+G\backslash LX)}
	(\mathrm{e}_{L^+X},\IC_{\chi,\lambda}(\frac{\delta(\lambda)} 
        {2}))\cong 
        H^i_{LG^+}(L^+X,j^!\IC_{\chi,\lambda}(\frac{\delta(\lambda)}{2}))$ is pure of weight $i$ (here $j:L^+X\to LX$ is the inclusion).
        Now the desired formality of $A_X$ 
        follows from the result of \cite[Section 6.5]{BF}.
		
\end{proof}

\section{Disconnected groups}\label{disconnected cases}
In the proofs from Section \ref{S:Iwahori orbits} to Section \ref{Formality}, the connectedness of $K$ is used only to insure that the points in $LX(k)=X(F)$ 
land in the image of the natural map $\pi:LG\to LX$ (see Proposition \ref{p:sym var orbits} (ii)).
In the general case of not necessary connected $K$, the image $LX_\mathcal L=\pi(LG)\subset LX$ is 
an union of components
$LX$, and  all the results in the previous sections 
are valid if we restrict to 
orbits in $LX_\mathcal L$.
More precisely:
\begin{enumerate}
\item
 All the results from Section \ref{S:Iwahori orbits} to Section \ref{Formality}
are valid when we consider $L^+G$ or $I_0$-orbits in $LX_\mathcal L$.

\item
The formality conjecture Theorem \ref{t:formality} holds true for general $K$.
Indeed,
the union of components $LX_\mathcal L\subset LX$ 
	is $LG$-invariant and 
	$L^+X\subset LX_\mathcal L$.  
	It follows that 
	$\IC_V\star\mathrm{e}_{L^+X}$ lands in the 
	full-subcategory $D^{}(L^+G\backslash LX_\mathcal L)\subset D^{}(L^+G\backslash LX)$
	and we have an isomorphism of dg algebras
	$A_X\cong R\Hom_{D(L^+G\backslash LX_\mathcal L)}(\mathrm{e}_{L^+X},\IC_{\cO(\check G)}\star\mathrm{e}_{L^+X})$.
	Now the desired claim follows from part $(1)$ above.  
\end{enumerate}

\appendix
\section{Sheaves on placid schemes}
\label{s:appendix}

We collect some basic definitions 
and properties of placid (ind) stacks, 
dimension theory, and 
sheaves theory on  placid (ind) stacks.
The main references for this section are
\cite{BKV,FKV,D,R}.

\subsection{Placid schemes and placid stacks
}\label{Placid schemes}

All schemes are over a fixed algebraically closed field $k$.
When talking about limits and colimits, all
the index sets are assumed to be countable.

\begin{enumerate}

\item 
A placid presentation of a scheme $Y$ is 
an isomorphism 
$Y \cong \lim_{j} Y_j$, where the limit is taken with respect to a filtered set $J$,
each $Y_j$, $j\in J$ is a scheme of finite type and the transition maps 
$Y_{j'} \to Y_{j}$ for $j\to j'$ are affine and  smooth.
We say that
a scheme $Y$ is \emph{placid} if 
it has an \'etale covering by schemes admitting placid presentations.
    
\item
Let $f:X\to Y$ be a finitely presented morphism.
Assume $Y\cong\lim_{j} Y_j$ admits a placid presentation. Then 
there is 
an index $j_0$, and a morphism 
$f_{j_0}:X_{j_0}\to Y_{j_0}$ of schemes of finite type such that $f$ is the pull-back of $f_{j_0}.$
In particluar, $X\cong\lim_{j\geq j_0}X_{j_0}\times_{Y_{j_0}}Y_j$
is a plaicd presentation of $X$.

\item
We call a morphism $f:X\to Y$ \emph{pro-smooth}
if locally in the \'etale topology $X$ has a presentation $X=\lim_ jX_j$ over $Y$ such that 
the all the projection $X_j\to Y$ are smooth and finitely presented and all the transition maps $X_j\to X_{j'}$ are smooth, finitely presented and affine.\footnote{
In \cite[Section 1.3.1 (c)]{BKV} and \cite[Section 1.2 (b)]{FKV}, the authors used the term \emph{smooth morphisms} for our pro-smooth morphisms.}
By definition,
a scheme $X$ is placid if and only if there is a 
pro-smooth morphism $f:X\to Y$ to a finite type scheme $Y$.

Let $f:X\to Y$ be a pro-smooth morphism. Then $Y$ is placid implies  $X$ is placid 
(see \cite[Section 1.1.3]{BKV}).
 
\item An ind-scheme $Y$ is called placid if it can be written as a filtered colimit
$Y\cong\on{colim}_i Y^i$
where each $Y^i$ is placid  
and the transition maps $Y^i\to Y^{i'}$ 
finitely presented closed embedding.

\item
A morphism $X\to Y$ from a scheme $X$
to an ind-scheme $Y\cong\on{colim}_i Y^i$ is called finitely presented if it factors through a finitely presented morphism $X\to Y^i$. If $Y$ is placid then part (2) implies $X$ is placid.


\item 
Let $H$ be placid group scheme acting on a placid ind-scheme $Y$.
We call $Y$ an $H$-placid ind-scheme if there exists 
a placid presentation $Y\cong\on{colim}_i Y^i$
such that 
each $Y^i$ is stable under the $H$-action.

   \item 

By a \emph{stack}, we mean a stack in groupoids 
in the \'etale topology.
A morphism $f:\mathcal X\to\mathcal Y$ between stacks is called a \emph{covering} if 
it is a surjective morphism of stacks.
A stack $\mathcal Y$ is called \emph{placid} if there is a pro-smooth representable covering
$Y\to\cY$ from a placid scheme $Y$.
An ind-stack $\cY$ is called placid 
if there is a presentation 
$Y\cong\on{colim}_i\cY^i$
where each $\cY^i$ is a placid stack and the transition maps
$\cY^i\to\cY^{i'}$
are  finitely presented representable closed embeddings.

The quotient stack $\mathcal Y\cong H\backslash Y$ of an $H$-placid ind scheme $Y$  is an example of placid ind-stack.

\item 
A morphism $f:\mathcal X\to\mathcal Y$ between stacks is called pro-smooth if there are 
coverings $Y\to\mathcal Y$
and $X\to\mathcal X\times_{\mathcal Y}Y$ from schemes $X$ and $Y$ such that 
the composed map
$X\to\mathcal X\times_{\mathcal Y}Y \to Y $ is pro-smooth.

\end{enumerate}

\subsection{Formally smoothness}
\begin{defe}
 \begin{itemize}
   \item [(i)]
   A morphism $f:X\to Y$ between ind-schemes is called formally smooth if for any $k$-algebra $R$ and any ideal $I\subset R$ with $I^2=0$ and a commutative solid diagram 
   \[
   \xymatrix{\on{Spec}(R/I)\ar[r]\ar[d]&X\ar[d]\\
   \on{Spec}(R)\ar[r]\ar@{-->}[ru]&Y}
   \]
   there is  
   a dotted arrow 
   which makes the diagram commute.
        \item [(ii)]
  An ind-scheme $X$ is called formally smooth if the natural map $f:X\to\on{Spec}(k)$ is formally smooth.
   \end{itemize}

\end{defe}

\begin{lem}\label{pro implies formal}
Let $f:X\to Y$ be a pro-smooth morphism between schemes.
Then $f$ is formally smooth.
\end{lem}
\begin{proof}
It follows from the infinitesimal lifting criterion of formally smoothness.
\end{proof}

\begin{lem}\label{formally smooth criterion}
Let $f:X\to Y$ be a morphism between formally smooth
placid ind-schemes. Then
$f$ is formally smooth
if the differential map $df:T_X\to f^*T_Y$ between 
tangent sheaves
is surjective.

\end{lem}
\begin{proof}
    In the case when both $X$ and $Y$ are ind-schemes of ind-finite type this is proved in 
    \cite[Lemma in page 332]{BD}.
    The argument in \emph{loc. cit.} works equally  to the placid ind-schemes setting: 
    Let $I\subset R$ be an  ideal in a $k$-algebra $R$ such that $I^2=0$.
    For a morphism $s:S=\on{Spec}(R/I)\to X$, let 
    $E_X(R,I,s)$ (resp. $E_Y(R,I,s)$)
    be the set of extensions of $s$ to a morphism 
    $\on{Spec}(R)\to X$ (resp. $\on{Spec}(R)\to Y$).
    Formally smooth of $f$ means the map 
    $f_*:E_X(R,I,s)\to E_Y(R,I,s)$ is surjective for all
    possible $R,I,s$ as above.
    Since $X$ (resp. $Y$) is formally smooth,
    according to \cite[Section 16.5.14]{Gro 1}, $E_X(R,I,s)$ (resp. $E_Y(R,I,s)$) is a torsor 
    over $\on{Hom}(s^*\Omega^1_X,I)$
    (resp. $\on{Hom}(s^*f^*\Omega^1_Y,I)$)
    where $\Omega^1_X$ (resp. $\Omega^1_Y$)
    is the cotangent sheaf. 
    Since $X$ (resp. $Y$) is placid, according to \cite[Theorem 6.2 (iii)]{D}, $\Omega^1_X$ (resp. $\Omega^1_Y$) is a Tate sheaf and it follows that
    there is an isomorphism
    $\on{Hom}(s^*\Omega^1_X,I)\cong\on{Hom}(I^\vee,s^*T_X)$ (resp. $\on{Hom}(I^\vee,s^*f^*T_Y)$) where $I^\vee=\Hom_{R/I}(I,R/I)$.  Moreover, the surjectivity of the differential $df$
    induces a surjective map 
    $\on{Hom}(I^\vee,s^*T_X)\to \on{Hom}(I^\vee,s^*f^*T_Y)$ and it follows that 
    the map $f_*$ between the corresponding torsors is also surjective. The lemma follows. 
    
\end{proof}

\begin{lem}\label{fixed points}
 Let $\sigma:X\to X$ be an involution on
a formally smooth  ind-scheme $X$.
Then the fixed points $X^\sigma$ is again a formally smooth ind-scheme.

\end{lem}
\begin{proof}
Let $R,I$ be as in Lemma \ref{formally smooth criterion} and let $s:\on{Spec}(R/I)\to X^\sigma$.
The involution $\sigma$ acts naturally on 
$\on{Hom}(s^*\Omega_X^1,I)$ and we have 
$H^1(\mu_2,\on{Hom}(s^*\Omega_X^1,I))=0$
since $\on{char} k\neq 2$ (here $-1\in\mu_2$ acts on $\on{Hom}(s^*\Omega_X^1,I)$ by $\sigma$).
Now we can apply 
 \cite[Lemma 9.4]{Gro 2} to conclude that 
$s$ admits a lifting $\tilde s:\on{Spec}(R)\to X^\sigma$.
The lemma follows.

\end{proof}

\subsection{Placid morphisms}\label{placid morphism}
\begin{defe}
\begin{itemize}
    \item [(i)]
A morphism $f:X\to Y$ between placid schemes is called \emph{placid} 
 if for any 
 \'etale cover 
 $X'\cong\lim X'_i\to X$ (resp. $Y'\cong\lim Y'_j\to Y$)
 admitting placid presentation and any map 
$f':X'\to Y'$ fitting into the following 
 commutative diagram (not necessarily Cartesian)
\[\xymatrix{X'\ar[r]^{f'}\ar[d]&Y'\ar[d]\\
X\ar[r]^f&Y}\]
the composition $X'\to Y'\to Y'_j$
factors as  
 $X'\to X'_i\stackrel{f'_{i,j}}\to Y'_j$ where $f'_{i,j}:X'_i\to Y'_j$ is smooth. 
 \item [(ii)]
 A morphism $f:\calX\to\mathcal Y$ bewteen placid stacks is called placid if there are pro-smooth representable coverings $Y\to\mathcal Y$ 
 and $X\to \calX\times_{\mathcal Y}Y$
 from placid schemes $X$ and $Y$ 
 such that composed map
$X\to \calX\times_{\mathcal Y}Y\to Y$ 
is placid.
 \end{itemize}
\end{defe}

\begin{exam}
Pro-smooth morphisms are placid.
However,
placid morphisms might not be pro-smooth.
More precisely, let 
$f:X\to Y$ be  a  placid morphism between schemes admitting placid presentations.
 For any index $j$ let 
$f_{i,j}:X_i\to Y_j$ be the smooth morphism as above.  We can form the 
base change $f_j:X_i\times_{Y_j}Y\to Y$ and 
$f\cong\lim f_j:X\cong\lim_j X_i\times_{Y_j}Y\to Y$ is a presentation over $Y$
where all the projections $f_j$ are smooth but the transition maps 
$X_i\times_{Y_j}Y\to X_{i'}\times_{Y_{j'}}Y$ might not be smooth.
\end{exam}

We have the following useful criterion to recognize placid morphisms.

\begin{lem}\label{formally smooth}
Let $f:X\to Y$ be a formally smooth morphism between placid schemes.
Then $f$ is placid.

\end{lem}
\begin{proof}
The proof is a variation of the argument in \cite[Lemma 1.1.4. (b)]{BKV}.

We can assume $X\cong\lim X_i$ and $Y\cong\lim Y_j$ admit placid presentations.
Since $Y_j$ is of finite type,
the composition $f_j:X\to Y\to Y_j$
factors as 
\[f_j:X\stackrel{\pi_i}\to X_i\to Y_j\] for some $f_{i,j}:X_i\to Y_j$.
Since the projection $X\to X_i$ is pro-smooth 
the cotangent complex $L_{X/X_i}$
is a flat module concentrated in degree zero.
Since both $f:X\to Y$ and the
projection $Y\to Y_j$ are formally smooth (Lemma \ref{pro implies formal}), the map
$f_j$ is formally smooth and by \cite[Tag. 0D10]{Stacks} the cotangent complex 
$L_{X/Y_j}$ satisfies a) $H^{-1}(L_{X/Y_j})=0$
and b) $H^0(L_{X/Y_j})\cong\Omega^1_{X/Y_j}$ is locally projective (in particular, it is flat).
Using the distinguished triangle
\[\pi_i^*L_{X_i/Y_j}\to L_{X/Y_j}\to L_{X/X_i}\]
we see that 
$H^{-1}(\pi_i^*L_{X_i/Y_j})=0$
and $H^0(\pi_i^*L_{X_i/Y_j})$ is a flat module.
Since $\pi_i$ is faithfully flat and $f_{i,j}:X_i\to Y_j$ is finitely presented,
we conclude that  $H^{-1}(L_{X_i/Y_j})=0$
and $H^0(L_{X_i/Y_j})$ is projective,
and  
\cite[Tag. 0D10 and 02GZ]{Stacks} implies $f_{i,j}$ is smooth.

 \end{proof}

It is known that a smooth morphism 
$f:X\to Y$ is an 
 open morphism, in particular, the image $f(X)\subset Y$
 has non-empty intersection $f(X)\cap U\neq\emptyset$
with any
open dense subset $U\subset Y$.
We have the following generalization to placid morphisms:
\begin{lem}\label{quasi-flatness}
Let $f:X\to Y$ be a placid morphism.
Assume  $X\cong\lim_{i}X_i$ admits a placid presentation with surjective 
projection maps $X\to X_i$.
Then the image of $f$
has non-empty intersection with any open dense finitely presented open
embedding $u:U\to Y$.
\end{lem}
\begin{proof}
Since the pre-images of open dense subsets under \'etale morphisms are open dense, 
we can assume $Y\cong\lim_j Y_j$  admits a placid presentation.
As $f$ is placid, the composed map $f_j=\pi_j\circ f:X\to Y\to Y_j$
factors as 
$X\to X_i\to Y_j$ where
 $f_{i,j}:X_i\to Y_j$ is smooth.
Since $X\to X_i$ is surjective, we see that
the image 
$f_j(X)=f_{i,j}(X_i)\subset Y_j$ is open.

Choose an open subset  $U_j\subset Y_j$ such that 
$U=\pi_j^{-1}(U_j)$.
If $f(X)\subset Z=Y\setminus U$ is contained in the closed complement, then 
$Z_j=\pi_j(Z)$ will contain
the open subset $f_j(X)=\pi_j(f(X))\subset Z_j$. This is
 impossible since it will imply
 its pre-image $Z=\pi^{-1}_j(Z_j)$
has non-empty intersection with the open dense subset $U$.
The lemma follows.

\end{proof}

\subsection{Dimension theory}\label{dimension theory}

\begin{enumerate}
\item For a scheme $Y$ of finite type over $k$ and $y\in Y$ we denote by $\dim_y(Y)$
the maximum of dimensions of irreducible components of $Y$ containing $y$.
The dimension function $\dim_Y:Y\to\mathbb Z$ is defined as 
$\dim_Y(y)=\dim_y(Y)$.
Let $f:X\to Y$ be a finitely presented morphism between finite type schemes.
We define $\dim_f=\dim_{X/Y}=\dim_X-f^*\dim_Y:X\to\mathbb Z$.

\item
A finite type scheme $Y$ is called 
\emph{locally equidimensional}
if each connected component of $Y$ is equidimentional, equivalently,
the dimension function $\dim_Y:Y\to\mathbb Z$ is locally constant.
 A placid presentation $Y\cong\on{lim}_jY_j$
 is called locally equidimentional
if each $Y_j$ is locally equidimensional.
 A placid scheme $Y$ 
  is called  locally equidimentional
 if there exists  an \'etale covering by a scheme with locally equidimensional
 placid presentation.
 A placid ind-scheme $Y$ is called locally equidimentional if there is a placid presentation 
 $Y\cong\on{colim} Y^i$ such that each
  $Y^i$ is  locally equidimentional.

\item 
Let $f:X\to Y$ be a finitely presented morphism between schemes admitting placid presentations.
According to Section 2 of \cite{BKV,FKV}, 
there is a well-defined function 
$\dim_f=\dim_{X/Y}:X\to\mathbb Z$ such that
$\dim_f=\pi_{j_0}^*\dim_{f_{j_0}}$
where $f_{j_0}:X_{j_0}\to Y_{j_0}$ is as in \eqref{Placid schemes}(4)
and 
$\pi_{j_0}:X\to X_{j_0}$ is the projection map.

\item
For every locally finitely presented morphism $f:X\to Y$ of placid scheme,
there is a unique function $\dim_f=\dim_{X/Y}:X\to\mathbb Z$ such that for every commutative diagram
\[\xymatrix{X'\ar[r]^{f'}\ar[d]^h&Y'\ar[d]\\
X\ar[r]^f&Y}\]
where $f':X'\to Y'$ is a finitely presented morphism of schemes admitting placid presentations 
and the vertical maps are \'etale morphisms,
we have $\dim_{f'}=h^*\dim_f$.

\item
For every representable locally finitely presented morphism 
$f:\mathcal X\to\mathcal Y$
of placid stacks, 
there is a unique function 
$\dim_f=\dim_{\mathcal X/\mathcal Y}:[\mathcal X]\to\mathbb Z$, where $[\mathcal X]$ is the underlying set of $\calX$, such that for every 
Cartesian diagram
\[\xymatrix{X\ar[r]^{f'}\ar[d]^h&Y\ar[d]\\\mathcal X\ar[r]^f&\mathcal Y}\]
where the vertical maps are pro-smooth coverings, we have 
$\dim_{f'}=h^*\dim_f$.

\item 
Let $f:\calX\to\mathcal Y$ be
a finitely presented representable locally closed embedding. We define $\on{codim}_f=-\on{dim}_f$
to be called the codimension function of $f$.

\item 
According to \cite[Section 2.4]{FKV},
the dimension function is additive, that is, for every pair $\calX\stackrel{f}\to\mathcal Y\stackrel{g}\to\mathcal Z$
morphisms as in (5), we have an equality 
\[\dim_{g\circ f}=\dim_f+f^*\dim_g.\]

\item
Consider a 
Cartesian diagram of placid stacks
\[
\xymatrix{\mathcal X'\ar[r]^{f'}\ar[d]^h&\mathcal Y'\ar[d]^b\\\mathcal X\ar[r]^f&\mathcal Y}
\]
such that $f$ is locally finitely presented
representable and $b$ is placid.
According to \cite[Lemma 2.3.7]{BKV}, we have 
\[h^*\dim_f=\dim_{f'}.\]

\item
Let $Y$ be a  placid ind-scheme.
A \emph{dimension theory} $\delta$ on $Y$ is an assignment
to each finitely presented  locally equidimensional subscheme
$S\subset Y$
a locally constant function
$\delta_{S}:S\to\mathbb Z$
such that for any such pair of subschemes $S\subset S'\subset Y$ we have 
\[
\delta_{S}-\delta_{S'}|_{S}=\on{dim}_{S/S'}:S\to\mathbb Z.
\]

\item 
Given two dimension theories $\delta$ and $\delta'$,
their difference $\delta-\delta':Y\to\mathbb Z$
sending $y\to \delta_S(y)-\delta'_S(y)$, where $S$ is any
finitely presented subscheme containing $y$,
is a well defined locally constant function on $Y$.
Thus the set of dimension theories on 
$Y$ is a pseudo torsor over the space of 
locally constant functions on $Y$.

\item
Given a locally equidimensional 
placid ind-scheme $Y\cong\on{colim}_i Y^i$,
a dimension theory $\delta$ on $Y$ is uniquely characterized by
the assignment $\delta_{Y^i}$ for each $Y^i$.
Indeed, for any finitely presented locally equidimensional subscheme
$S\subset Y$, we have a finitely presented embedding $S\subset Y^i$ for a large enough $i$
and we define $\delta_S=\delta_{Y^i}|_S-\dim_{S/Y^i}$.

\end{enumerate}

\begin{exam}\label{dim for placid and ind-finite type}
Assume $Y\cong\on{colim}_i Y^i$
  is  ind-finite type, that is, each $Y^i$ is locally equidimensional 
  of finite type.
  Then we have a canonical dimension theory 
  given by the dimension of $Y^i$,
$\delta_{Y^i}=\dim_{Y^i}:Y^i\to\mathbb Z$.

\end{exam}

\subsection{Sheaf theories}
\label{sheaf theories}

\begin{enumerate}
\item
We will be working with $\overline{\mathbb Q}_\ell$-linear dg-categories. Unless specified otherwise, all dg-categories will be assumed cocomplete, i.e., containing all small colimits,
and all functors between dg-categories will be assumed continuous, i.e., preserving all small colimits. 
For any finite type scheme $Y$ we denote by $D(Y)=\on{Ind}(D^b_c(Y))$ the ind-completion of 
the bounded 
 dg derived category $D^b_c(Y)
 $ of $\ell$-adic constructible complexes on $Y$.
For any $\cF\in D(Y)$ and any integer $n\in\mathbb Z$ we will write $\cF\langle n\rangle=\cF[n](n/2)$. More generally, for any locally constant function $f:Y\to\mathbb Z$,
we define $\cF\langle f\rangle$ such that for any connected 
component $Y_\beta\subset Y$ we have 
$\cF\langle f\rangle|_{Y_\beta}=\cF\langle f(Y_\beta)\rangle|_{Y_\beta}$.
We have the Verdier duality fucntor 
$\mathbb D_Y:D(Y)\cong D(Y)^{op}$
satisfying 
$f^!\mathbb D_Y\cong \mathbb D_X f^*$.
Given a smooth map $f:X\to Y$ between schemes of finite types, we have 
$f^!\cong f^*\langle 2\dim_f\rangle $ where $\dim_f$ is the relative dimension $f$. 

\item
For a scheme $Y$ admitting a  placid presentation 
$Y\cong\on{lim}_jY_j$,  we define 
the $*$-version of dg derived category of constructible complexes on $Y$ as the colimit 
$D(Y)\cong\on{colim}^*_j D(Y_j)$.
For any placid stack  $\cY$ we define 
$D(\cY)\cong\on{lim}^*_{[n]\in\Delta_s^{op}} D(X^{[n]})$
where $\{X^{[n]}\}_{[n]\in\Delta_s^{op}}$ is the 
 \v Cech nerve of a smooth covering $X\to\cY$
 admitting a placid presentation
 (note that by~\eqref{Placid schemes}(4) each $X^{[n]}$ admits a placid presentation).
For any  placid ind-stack $\cY\cong\on{colim}_i\cY^i$,
we define the $*$-version of dg derived category of constructible complexes on $\cY$ as
$D(\cY)=\on{colim}_{i,!} D(\cY^i)$
where the first  colimit is with respect to the 
$!$-pushforward along the 
closed embedding $\cY^i\to\cY^{i'}$.

We can also introduce $!$-version of $D(\cY)$ 
denoted by $D_!(\cY)$  using the $!$-pullbacks in the above definitions.

 \item
Let $f:\mathcal X\to\cY$ be a 
finitely presented representable morphism between placid stacks.
Then we have (derived) functors
$f_*,f_!,f^*,f^!$ between 
$D(\mathcal X)$ and $D(\cY)$  (resp. $D_!(\mathcal X)$ and $D_!(\cY)$).

\end{enumerate}

\subsection{Verdier dualities}\label{Appendix:Verdier}

\begin{lem}\label{l:dim}

\begin{itemize}

    \item [(i)]
   Let $Y$ be an $H$-placid ind-scheme. There is a Verdier duality functor
\begin{equation}\label{Verdier}
      \mathbb D'_Y: D(H\backslash Y)\cong D_{!}(H\backslash Y)^{op}
\end{equation}
satisfying $(\mathbb D_Y')^2\cong\on{Id}$.
\item [(ii)]
Let $Y$ is locally equidimensional
   $H$-placid scheme.
   There is a canonical equivalence
   \[\Phi_{Y}:D_!(H\backslash Y)\cong D^{}(H\backslash Y).\]
\item [(iii)]
Let $Y\cong\on{colim}_i Y^i$ be a locally equidimensional
   $H$-placid ind-scheme.
Then a dimension theory $\delta$ on 
	$Y$ induces  equivalences
   \[\Phi_\delta\cong\on{colim}_i\Phi_{Y^i}\langle2\delta_{Y^i}\rangle:D_!(H\backslash Y)\cong D^{}(H\backslash Y).\]
\end{itemize}

\end{lem}

\begin{proof}

Proof of (ii).
Let $\{X^{[n]}\}_{[n]\in\Delta_s^{op}}$ be  the 
 \v Cech nerve of a smooth covering $X\to H\backslash Y$ admitting a placid presentation.
The assignment
	 sending 
	$\{\cF^{[n]}_j\}\in D_!(X^{[n]})\cong\on{colim}^!_j D(X^{[n]}_j)$,
	to  
	$\{\cF^{[n]}_j\langle-2\dim(X^{[n]}_j)\}\in D(X^{i,[n]})\cong\on{colim}^*_j D(X^{i,[n]}_j)$  
   induces an equivalence.
   
Proof of (i).
Pick an $H$-placid presentation 
$Y\cong\on{colim}_i Y^i$.
Let $\{X^{i,[n]}\}_{[n]\in\Delta_s^{op}}$ be the 
 \v Cech nerve of a smooth covering $X^i\to H\backslash Y^i$ admitting a placid presentation.
The isomorphism 
$f^!\mathbb D_Y\cong \mathbb D_X f^*$ for maps between finite type schemes implies that the assignment sending
$\{\cF^{i,[n]}_j\}\in D(X^{i,[n]})\cong\on{colim}^*_j D(X^{i,[n]}_j)$ 
to 
$\{\mathbb D_{}\cF^{i,[n]}_j\}\in D_!(X^{i,[n]})\cong\on{colim}^!_j D(X^{i,[n]}_j)$
induces a well-defined 
 functor
\begin{equation}\label{Verdier}
    \mathbb D_Y': D(H\backslash Y)\cong D_{!}(H\backslash Y)^{op}
\end{equation}
satisfying $(\mathbb D'_Y)^2\cong\on{Id}$.  

Proof of (iii).
The assignment
	 sending 
	$\{\cF^{i,[n]}_j\}\in D_!(X^{i,[n]})\cong\on{colim}^!_j D(X^{i,[n]}_j)$ $\ $
	to
	$\{\cF^{i,[n]}_j\langle-2\dim(X^{i,[n]}_j)+2\delta_{Y^i}\rangle\}\in D(X^{i,[n]})\cong\on{colim}^*_j D(X^{i,[n]}_j)$  
   induces an equivalence.

\end{proof}

The lemma above implies the following.

\begin{prop}\label{Verdier duality}
\begin{itemize}
    \item [(i)]
 Let 
   $Y$ be a locally equidimensional $H$-placid scheme. 
There is an equivalence
   \[
   \mathbb D_Y:=\Phi_Y\circ\mathbb D'_Y: D(H\backslash Y)\cong D(H\backslash Y)^{op}
   \]
   satisfying $(\mathbb D_Y)^2\cong\on{id}$.
   
  \item [(ii)]  
  Let 
   $Y\cong\on{colim}_i Y^i$ be a locally equidimensional $H$-placid ind-scheme. 
   A dimension theory $\delta$ on $Y$ induces an equivalence
   \[\mathbb D_\delta:=\Phi_\delta\circ\mathbb D'_Y\cong\on{colim}_i\mathbb D_{Y^i}\langle2\delta_{Y^i}\rangle: D(H\backslash Y)\cong D(H\backslash Y)^{op}\]
    satisfying $(\mathbb D_\delta)^2\cong\on{id}$.
    \end{itemize}
\end{prop}

\begin{exam}\label{ind of finite type}
Assume 
$Y\cong\on{colim}_iY^i$ is of ind-finite type.
 Then, by  Example \ref{dim for placid and ind-finite type}, there is a canonical dimension theory 
$\delta_{Y^i}=\dim_{Y^i}$  for $Y$ and one can check that  $D(H\backslash Y)\cong D_!(H\backslash Y)$,
$\Phi_{\delta_{}}\cong\on{id}$, 
and 
$\mathbb D_\delta\cong\mathbb D'_Y$ is the usual Verdier duality.

\end{exam}

We will need the following properties of Verdier dualities later.
\begin{lem}\label{property of D}
 Let $f:S\to Y$ be a finitely presented morphism 
between locally equidimensional $H$-placid schemes.
Then there are canonical isomorphisms 
\begin{itemize}
    \item 
[(i)] 
$f^!\mathbb D_Y\cong\mathbb D_S f^*\langle2\dim_{f}\rangle
$,  
\item [(ii)]
$\mathbb D_Y f_!\cong f_*\mathbb D_S \langle2\dim_{f}\rangle$.
\end{itemize}
\end{lem}
\begin{proof}
We give a proof of (i). The case of (ii) is similar.
Pick a smooth covering $X\to H\backslash Y$ admitting a placid presentation
and let $V=(H\backslash S)\times_{H\backslash Y}X$ be the base change.
The map $f$ induces a map 
$f^{[n]}\cong\lim_j f^{[n]}_j:V^{[n]}\cong\lim_jV^{[n]}_j\to X^{[n]}\cong\lim_j X^{[n]}_j$
between the 
 \v Cech nerves.
Note that we have 
\begin{equation}\label{dim equality}
2\dim_f=2\dim_{f^{[n]}_j}=-2\dim_{X^{[n]}_j}+2\dim_{V^{[n]}_j}
\end{equation}
and it follows that
\begin{align*}
f^!\mathbb D_Y(\cF)
&\ \cong f^!(\on{lim}^*_{[n]}\on{colim}^*_j (\mathbb D_{X^{[n]}_j}(\cF_j^{[n]})\langle-2\dim_{X^{[n]}_j}\rangle))\\
&\ \cong \on{lim}^*_{[n]}\on{colim}^*_j ((f^{[n]}_j)^!\mathbb D_{X^{[n]}_j}(\cF_j^{[n]})\langle-2\dim_{X^{[n]}_j}\rangle)\\
&\stackrel{~\eqref{dim equality}}\cong\on{lim}^*_{[n]}\on{colim}^*_j (\mathbb D_{V^{[n]}_j}(f^{[n]}_j)^*(\cF_j^{[n]})\langle-2\dim_{V^{[n]}_j}+2\dim_f\rangle)\\
&\ \cong\mathbb D_{V}(\on{lim}^*_{[n]}\on{colim}^*_j ((f^{[n]}_j)^*(\cF_j^{[n]})))\langle2\dim_f\rangle
\cong\mathbb D_V f^*\cF\langle2\dim_f\rangle.
\end{align*}
\end{proof}

\subsection{t-structures and perverse sheaves}\label{t-structures}
Let $Y$ be a placid scheme.
Since $*$-pullback is $t$-exact with respect to the 
standard $t$-structures on finite type schemes, 
the category $D(Y)$ carries a canonical standard $t$-structure $(D^{\leq0}_{}(Y),D^{\geq0}_{}(Y))$
and we denote by $\on{Shv}(Y)$ the corresponding abelian category of 
constructible sheaves on $Y$.
Assume $Y$ is an $H$-placid scheme where $H$ is a placid affine group scheme.
We can pull back the standard $t$-structure 
on $D_{}(Y)$ 
along the forgetful map 
$D(H\backslash Y)\to D_{}(Y)$
and obtain a
$t$-structure $(D^{\leq0}_{}(H\backslash Y),D^{\geq0}_{}(H\backslash Y))$
on $D(H\backslash Y)$
with abelian heart 
\[\on{Shv}(H\backslash Y)=D^{\leq0}_{}(H\backslash Y)\cap D^{\geq0}_{}(H\backslash Y)\]
of $H$-equivariant  constructible sheaves on $Y$.
We denote by 
\[\cH^0:D_{}(H\backslash Y)\to\on{Shv}(H\backslash Y)\]
the corresponding cohomological functor.

Let $Y$ be a finite type scheme.
Following \cite[Section 6]{BKV}, we can introduce the $*$-\emph{adapted perverse
$t$-structure} 
$({^p}D^{\leq0}_{}(Y),{^p}D^{\geq0}_{}(Y))$
on $D_{}(Y)$ such that for any smooth morphism 
$f:X\to Y$, the pullback $f^*$ is $t$-exact.\footnote{In \emph{loc. cit.}, the authors considered the 
$!$-adapted $t$-structure instead of the $*$-adapted one.} 
If $Y$ is locally equidimensional, we have 
\begin{equation}\label{*-adpated t-structure}
    {^p}D^{\leq0}_{}(Y)={^{p_{cl}}}D^{\leq0}(Y)[-\dim_Y]
\end{equation}
where ${^{p_{cl}}}D^{\leq0}(Y)$ is the classical perverse $t$-structure and 
$\dim_Y:Y\to\mathbb Z$ the dimension function (which is locally constant).
In the general case, we consider the 
canonical equidimensional stratification $\{Y_\beta\} $ of $Y$ and define ${^p}D^{\leq0}_{}(Y)$
as the gluing of ${^p}D^{\leq0}_{}(Y_\beta)$, see 
\cite[Section 6.2.3]{BKV} for more details.

Let $Y$ be a placid scheme.
Then we can consider the \emph{perverse
$t$-structure}
$({^p}D^{\leq0}_{}(Y),{^p}D^{\geq0}_{}(Y))$
 on $D_{}(Y)$
given by
\begin{equation}\label{standard t-structure}
    {^p}D^{\leq0}_{}(Y)=\on{lim}^{*}_{[n]}\on{colim}_j^{*}({^{p_{}}}D^{\leq0}(X^{[n]}_j))
\end{equation}
where the colimit is taken with respect to $*$-pullback (here  $\{X^{[n]}\cong\on{lim}_j X^{[n]}_j\}_{[n]\in\Delta_s^{op}}$ is  the 
 \v Cech nerve for an \'etale cover $X\to Y$ admitting a placid 
presentation).
Assume $Y$ is $H$-placid then, similarly, 
we can pull back the perverse $t$-structure on $D_{}(Y)$ along the forgetful map 
$D_{}(H\backslash Y)\to D_{}(Y)$
and obtain a
perverse $t$-structure 
$(^{p}D^{\leq0}_{}(H\backslash Y),{^{p}}D^{\geq0}_{}(H\backslash Y))$
with abelian heart of 
$H$-equivariant perverse sheaves on $Y$:
\[\on{Perv}(H\backslash Y)={^p}D^{\leq0}(H\backslash Y)\cap{^p}D^{\geq0}_{}(H\backslash Y).\]
We denote by
\[{^p}\cH^0:D(H\backslash Y)\to\on{Perv}(H\backslash Y)\]
the corresponding cohomological functor.

Let $Y\cong\on{colim}_iY^i$ be a locally equidimensional $H$-placid ind-scheme equipped with 
a dimension theory $\delta_{Y^i}:Y^i\to\mathbb Z$. 
It follows from 
$\delta_{Y^i}=\delta_{Y^{i'}}|_{Y^i}+\dim_{Y^i/Y^{i'}}$
and~\eqref{*-adpated t-structure}
that the functor 
$(\iota_{i,i'})_!:D(H\backslash Y^i)\to D(H\backslash Y^{i'})$
(where $\iota_{i,i'}:Y^i\to Y^{i'}$ is the transition map) satisfies
\[
(\iota_{i,i'})_!({^p}D^{\leq0}_{}(H\backslash Y^i)[\delta_{Y^i}])\subset {^p}D^{\leq0}_{}(H\backslash Y^{i'})[\delta_{Y^{i'}}],
\]
\[
(\iota_{i,i'})_!({^p}D^{\geq0}_{}(H\backslash Y^i)[\delta_{Y^i}])\subset {^p}D^{\geq0}_{}(H\backslash Y^{i'})[\delta_{Y^{i'}}].
\]
Hence there is a perverse $t$-structure
$({^{p}}D^{\leq0}_{}(H\backslash Y)[\delta],{^{p}}D^{\geq0}_{}(H\backslash Y)[\delta])$
on the colimit $D_{}(H\backslash Y)\cong\on{colim}_{i,!}(D(H\backslash Y^i))$
characterized by 
\begin{equation}\label{perverse t-structure}
    {^{p}}D^{\leq0}_{}(H\backslash Y)[\delta]=\on{colim}_{i,!}({^{p_{}}}D^{\leq0}(H\backslash Y^{i})[\delta_{Y^i}]),
\end{equation}
with the abelian heart 
of $H$-equivariant perverse sheaves on $Y$
\[\on{Perv}_\delta(H\backslash Y)=\on{colim}_{i,!}\on{Perv}(H\backslash Y^i)[\delta_{Y^i}]={^p}D^{\leq0}(H\backslash Y)[\delta]\cap{^p}D^{\geq0}_{}(H\backslash Y)[\delta].\]
We denote by 
\[{^{p}}\cH^0_\delta:D_{}(H\backslash Y)\to\on{Perv}_\delta(H\backslash Y)\]
the corresponding cohomological functor.

\begin{lem}\label{fully faithful}
Let $Y\cong\on{colim}_iY^i$ be a locally equidimensional $H$-placid ind-scheme equipped with 
a dimension theory $\delta$. 
\begin{itemize}
    \item [(i)] The Verdier duality $\mathbb D_\delta:D_{}(H\backslash Y)\cong D_{}(H\backslash Y)^{op}$
~\eqref{Verdier} is $t$-exact with respect to the perverse 
$t$-structure~\eqref{perverse t-structure} and induces an 
equivalence 
$\mathbb D_\delta:\on{Perv}_\delta(H\backslash Y)\cong\on{Perv}_\delta(H\backslash Y)^{op}$.
\item [(ii)]
Assume $H$ is irreducible
and admits a placid presentation $H\cong\lim_j H_j$
where the kernel $H\to H_j$ is pro-unipotent.
Then the forgetful functor 
$\on{For}:\on{Perv}_\delta(H\backslash Y)\to\on{Perv}_\delta(Y
)$
is fully-faithful.
\end{itemize}

\end{lem}
\begin{proof}
    Proof of (i). It follows from
    \begin{align*}
        &\mathbb D_\delta(^{p}D^{\leq0}_{}(H\backslash Y^i)[\delta_{Y^i}])\cong
    \Phi_\delta\circ\mathbb D'_Y(\on{lim}^*_{[n]}\on{colim}^*_{j}{^{p,cl}D^{\leq0}}(X^{i,[n]}_j)[-\dim_{X^{i,[n]}_j}+\delta_{Y^i}])\\
    \cong&\Phi_\delta(\on{lim}^!_{[n]}\on{colim}^!_{j}{^{p,cl}}D^{\leq0}_{}(X^{i,[n]}_j)[\dim_{X^{i,[n]}_j}-\delta_{Y^i}])\\
    \cong&\on{lim}^*_{[n]}\on{colim}^*_{j}{^{p,cl}}D^{\leq0}_{}(X^{i,[n]}_j)[\dim_{X^{i,[n]}_j}-\delta_{Y^i}]\langle-2\dim_{X^{i,[n]}_j}+2\delta_{Y^i}\rangle\\
    \cong&\on{lim}^*_{[n]}\on{colim}^*_{j}{^{p,cl}}D^{\leq0}_{}(X^{i,[n]}_j)[-\dim_{X^{i,[n]}_j}+\delta_{Y^i}]
    \cong
    {^{p}}D^{\leq0}_{}(H\backslash Y^i)[\delta_{Y^i}].
    \end{align*}

    Proof of (ii). The proof is similar to the case of schemes of finite types : consider the following diagram of functors
    \[\xymatrix{\on{Perv}_\delta(H\backslash Y)\ar[r]^{a^{*}}\ar[d]^{\on{For}}&\on{Perv}_\delta(H\backslash (H\times Y))\ar[d]^{\on{For}}&\on{Perv}_{\delta}(Y)\ar[l]_{\ \ \ \ \simeq}\ar[dl]^{p^{*}}\\
\on{Perv}_\delta(Y)\ar[r]^{p^{*}}&\on{Perv}_\delta(H\times Y)}\]
    where $H$ acts on $H\times Y$
    on the first factor and 
    $a,p:H\times Y\to Y$ are the action  and projection maps respectively.
    The assumption implies that 
    both $a,p$ are strongly pro-smooth with all the projection and transition maps being smooth with connected fibers. 
    It follows that 
    $a^{*}$, $p^{*}$ are fully-faithful \cite{BBDG}. Thus   the vertical forgetful maps are fully-faithful.
\end{proof}

\begin{exam}(IC-complexes)\label{perverse local systems}
We preserve the setup of Lemma \ref{fully faithful}.
Let $\cO=H\cdot y$ be the $H$-orbit 
through $y\in Y^i$
and let $S=H_y$ be the stabilizer of $y$. 
Choose a placid presentation $H\cong\on{lim}_jH_j$
and let $S_j=\on{Im}(S\subset H\to H_j)$
be the image of $S$ along the projection $H\to H_j$.
Then the natural projection map 
$H_{j'}/S_{j'}\to H_{j}/S_{j}$
is smooth affine and  unipotent and we get a 
placid presentation 
$\cO\cong H/S\cong\on{lim} H_j/S_j$. In particular, $\cO$
is irreducible placid scheme.
 
Let $\pi_0(S_j)$ be the group of connected components of $S_j$ and 
for any representation $\xi\in\on{Rep}(\pi_0(S_j))$ of $\pi_0(S_j)$ we denote by $\cL_\xi$
the corresponding $H_j$-equivariant local system 
on $H_j/S_j$.
There is an equivalence 
\[
\on{Rep}(\pi_0(S_j))\cong\on{Perv}(H\backslash \cO)\cong\on{Shv}(H\backslash \cO), \ \ \xi\to\cL_{\xi,\cO}:=\pi_j^*\cL_\xi.
\] 
Assume the inclusion $j:\cO\to Y^i$
is a finitely presented open embedding. Then we have the \emph{standard sheaf} 
(resp. \emph{co-standard complex})
\begin{equation}\label{standard-costandard}
j_!\cL_{\xi,\cO}\in\on{Shv}(H\backslash Y^i)\ \ \ (resp.\ \ j_*\cL_{\xi,\cO}\in D(H\backslash Y^i))
\end{equation}
and the intermediate extension 
\begin{equation}\label{IC}
    j_{!*}\cL_{\xi,\cO}:=
    \on{Im}({^p}\cH^0(f_!(\cF))\to {^p}\cH^0(f_*(\cF)))\in\on{Perv}(H\backslash Y^i).
\end{equation}
According to the definition of perverse $t$-structure~\eqref{perverse t-structure},
we have 
\begin{equation}\label{shifted IC}
\IC_{\xi,\cO}:=j_{!*}\cL_{\xi,\cO}[\delta_\cO]\in \on{Perv}_\delta(H\backslash Y)=\on{colim}_{i,!}\on{Perv}(H\backslash Y^i)[\delta_{Y^i}].
\end{equation}
Finally, it follows from  Lemma \ref{property of D} that 
\begin{equation}\label{}
\mathbb D_{Y^i}( j_{!*}\cL_{\xi,\cO})\cong  j_{!*}\cL_{-\xi,\cO},\ \ \ \ \ \mathbb D_{Y^i}(j_{!}\cL_{\xi,\cO})\cong j_*\cL_{-\xi,\cO},
\end{equation}
and hence from Proposition \ref{Verdier duality}
\begin{equation}\label{D(j_!)}
\mathbb D_\delta(j_{!}\cL_{\xi,\cO})\cong\mathbb D_{Y^i}(j_{!}\cL_{\xi,\cO})\langle2\delta_\cO\rangle\cong j_*\cL_{-\xi,\cO}\langle2\delta_\cO\rangle,
\end{equation}
\begin{equation}\label{D(IC)}
\mathbb D_\delta(\IC_{\xi,\cO})\cong \mathbb D_{Y^i}(j_{!*}\cL_{\xi,\cO}[\delta_\cO])\langle2\delta_\cO\rangle\cong
(j_{!*}\cL_{-\xi,\cO}[-\delta_\cO])\langle2\delta_\cO\rangle
\cong
\IC_{-\xi,\cO}(\delta_\cO).
\end{equation}
\end{exam}

\subsection{Pointwise purity}\label{A: pointwise purity}
Let $Y$ be a locally equidimensional $H$-placid scheme.
In this section we
assume the following:
\begin{enumerate}
    \item $H$ is irreducible and 
    there is a placid presentation $H\cong\lim_j H_j$
    where the kernels $H\to H_j$ are pro-unipotent.
    \item 
    $Y$ is a union of finitely many $H$-orbits 
    and for any orbit $\cO\subset Y$
the inclusion $j:\cO\to Y$
is finitely presented.
\item Assume $H$ and $Y$ are defined over a finite field $\mathbb F_q$.

\end{enumerate}

We write $F:Y\to Y$ for the Frobenius morphism and 
we denote by $D(H\backslash Y)^F$ the corresponding category of 
$F$-equivariant complexes whose objects consist of pairs $(\cF,\Phi)$
where $\cF\in D(H\backslash Y)$ and $\Phi:F^*\cF\cong\cF$.
Similarly, we denote by $\on{Shv}(H\backslash Y)^F$, $\on{Perv}(H\backslash Y)^F$, etc, the corresponding categories of 
$F$-equivariant objects.

\begin{defe}\label{pointwise pure complexes}
    \begin{itemize}

    \item [(i)]
    A complex $(\cF,\Phi)\in D(H\backslash Y)^F$ is called pure of weight $w$ if there 
    exists a smooth cover $\pi:X\to H\backslash Y$ defined over $\mathbb F_q$ admitting a placid presentation 
    $X\cong\lim_j X_j$ and a pure complex   $\cF_{j_0}\in D(X_{j_0})^F$ of weight $w$ in the sense of \cite{BBDG} such that 
    $\pi^*\cF\cong\pi_{j_0}^*\cF_{j_0}$. Here $\pi_{j_0}:X\to X_{j_0}$ is the transition map.

    \item [(ii)]
    An $F$-equivariant local system $(\cL,\Phi)\in\on{Shv}(H\backslash\cO)^F$
    is called pointwise pure of weight $w$
    if for any positive integer $n$
    and any
    $x\in\cO^{F^n}$, the eigenvalues 
    of $\Phi^n$ on the stalk $\cL_x$
   are algebraic numbers in $\overline{\mathbb Q}_\ell^\times$ all of whose 
complex conjugates have absolute value 
$(q^n)^{\frac{w}{2}}$. If $\cL$ is pointwise pure of weight $w$, then we say $\cL[m]$ is pointwise pure of weight $w+m$.
\item [(iii)]
A complex $(\cF,\Phi)\in D(H\backslash Y)^F$
is said to be $*$-pointwise pure (resp. $!$-pointwise pure)
of weight $w$ if for any $H$-orbit $j:\cO\to Y$ and integer $m$, the local system $\sH^m(j^*\cF)$
(resp. $\sH^m(j^!\cF)$)
is pointwise pure of weight $w+m$
(Note that the assumption (ii) above implies that the functor $j^!$ is well-defined). It is said to be pointwise pure if 
it is both $*$-pointwise pure and $!$-pointwise pure.

    \end{itemize}
\end{defe}

\begin{lem}\label{* to !}
If $\cF\in D(H\backslash Y)^F$ is $*$-pointwise pure of weight $w$,
    then its Verdier dual $\mathbb D_Y(\cF)$ is $!$-pointwise pure of weight $-w$.
   \end{lem}
\begin{proof}
It follows from Lemma \ref{property of D}.
\end{proof}

\begin{exam}\label{example of pure complex}
Let $j:\cO\to Y$ be an open orbit and let 
$j_{!*}(\cL_{\xi,\cO})$ be the intermediate extension in Example \ref{perverse local systems}.
If $\cL_{\xi,\cO}$ is a pointwise pure local system of weight $w$, then $j_{!*}(\cL_{\xi,\cO})$
is pure of weight $w$.

\end{exam}

We have the following criterion  for $*$-pointwise purity
generalizing  \cite[Proposition 2.3.3]{MS} to the setting of 
$H$-placid schemes.
To this end, we first recall the notion of
contracting slices.

\begin{defe}\label{def of slice}
Let $\cO=H\cdot x\subset Y$ be the
$H$-orbit through a point $x\in Y$.
 A contracting slice at $x$ with respect to $\cO$
  is a locally closed finite type subscheme $S\subset Y$ defined over $\mathbb F_q$ such that 
 \begin{itemize}

     \item [(i)] $S$ contains $x$,
     \item [(ii)] the action map $a:H\times S\to Y$
     is formally smooth,
     \item [(iii)] there is a homomorphism $\lambda:\mathbb G_m\to H$
    such that $S$ is preserved by $\lambda(\mathbb G_m)$ 
    and the resulting $\mathbb G_m$-action on $S$ is attracting 
    with $x$ as its unique fixed point.
   
 \end{itemize}
 
\end{defe}

\begin{prop}\label{purity criterion}
Let $(\cF,\Phi)\in D(H\backslash Y)^F$ be a pure complex of weight $w$.
Assume there is a contracting slice at each point $x\in Y$ with respect to the 
orbit $\cO=H\cdot x$.
Then $\cF$ is $*$-pointwise pure of weight $w$.
\end{prop}
\begin{proof}
Let $i:S\subset Y$ be a contracting slice at $x\in Y^{F^n}$.
We claim that $\cF|_S=i^*\cF\in D(S)^F$
is pure of weight $w$ in the sense of \cite{BBDG}.
Assumption (iii) in Definition \ref{def of slice} implies that  $\cF|_S$ is equivariant with respect to the contracting $\mathbb G_m$-action 
on $S$ and we can apply  \cite[Proposition 2.3.3]{MS} to conclude that  the eigenvalues of 
$\Phi^n$ on the stalk $\sH^m(\cF|_x)\cong \sH^m((\cF|_S)_x)$
are algebraic numbers in $\overline{\mathbb Q}_\ell^\times$ all of whose 
complex conjugates have absolute value 
$(q^n)^{\frac{w}{2}}$. The proposition follows.

Proof of the claim.
Consider the action map $a:H\times S\to Y$.
Let $X\cong\lim_j X_j\to Y$
be an \'etale cover admitting a placid presentation
and let $V=(H\times S)\times_{Y}X$ 
be the base change. 
Since the projection $f:V\to H\times S$ is an \'etale cover, in particular finitely presented,
by Section \ref{Placid schemes} (4), there is an \'etale cover 
 $f_{j_0}:V_{j_0}\to H_{j_0}\times S$ 
 such that $f$ is the base change of $f_{j_0}$
 and 
 $V\cong\lim_{j\geq j_0} V_j=V_{j_{0}}\times_{(H_{j_0}\times S)}
 (H\times S)$ is a placid presentation of $V$.
Moreover, 
since the projection map $b:V\to X$ is formally smooth (being the base change of the 
formally smooth map $a$), it follows from Lemma \ref{formally smooth}  that 
the composed map $V\to X\to X_j$
factors as $V\to V_{j'}\stackrel{c}\to X_j$ 
with $c$ a smooth map.
We summarize the above discussion by 
the following commutative diagram
\begin{equation}\label{reduction to finite types}
  \xymatrix{H_{j'}\times S&H\times S\ar[r]^a\ar[l]_q&Y\\
V_{j'}\ar[u]_{f_{j'}}\ar[rrd]^c&V\ar[r]^{b}\ar[l]_g\ar[u]_{f}&X\ar[u]_{h
}\ar[d]^{p}\\
 &&X_j}  
\end{equation}
where the upper squares are Cartesian diagrams,
the  maps $h,f,f_{j'}$ are \'etale, and  
the map $c$ is smooth.
From the definition of pure complexes in Definition \ref{pointwise pure complexes},  there is a  
 large enough $j$ such that 
\[h^*\cF \cong
p^*\cF_{j}\]
where $\cF\in D(X_j)^F$ is pure of weight $w$.
Note that there is an isomorphism
\begin{equation}\label{purity F}
    (f_{j'})^*(\overline{\mathbb Q}_{\ell,H_{j'}}\boxtimes\cF|_{S})\cong c^*\cF_j.
\end{equation}
Indeed, since $g$ is pro-smooth with contractible 
fibers it suffices to check~\eqref{purity F} after applying $g^*$,
but since $a^*(\cF)\cong\overline{\mathbb Q}_{\ell,H}\boxtimes\cF|_{S}$ 
by the $H$-equivariance of $\cF$, 
we have 
\[g^*(f_{j'})^*(\overline{\mathbb Q}_{\ell,H_{j'}}\boxtimes\cF|_{S})\cong f^*((\overline{\mathbb Q}_{\ell,H}\boxtimes\cF|_S)\cong f^*a^*(\cF)\cong b^*h^*\cF\cong b^*p^*\cF_j\cong g^*c^*\cF_j.\]
Since $c$ is a smooth map (between finite type schemes) the complex $c^*\cF_j$ is pure of weight $w$
and the isomorphism~\eqref{purity F} together with the fact that   
 $f_{j'}$ is \'etale 
 imply that $\overline{\mathbb Q}_{\ell,H_{j'}}\boxtimes\cF|_{S}$
 is pure of weight $w$.
 Finally, since $e^*\cong e^!\langle 2\dim(H_{j'})\rangle$
 on $\overline{\mathbb Q}_{\ell,H_{j'}}\boxtimes\cF|_{S}$,
 where $e:S\to H_{j'}\times S$ is the embedding $e(x)=(1,x)$ ($1\in H_{j'}$ is the unit), we conclude that 
 $\cF|_{S}\cong e^*(\overline{\mathbb Q}_{\ell,H_{j'}}\boxtimes\cF|_{S})$
 is pure of weight $w$.
The claim follows.

\end{proof}

\end{document}